\documentclass{jgcc} 

\pdfoutput=1
\usepackage{lastpage}
\jgccdoi{18}{2}{5}{18044 }
\jgccheading{}{\pageref{LastPage}}{}{}{Apr.~17,~2026}{Oct.~7,~2026}{}  


\usepackage{graphicx}
\graphicspath{{./Graphs/Cliques_4_4/}{./Graphs/Cliques_5_3/}{./Graphs/Cliques_6_2/}{./Graphs/Cliques_7_1/}{./Graphs/Disconnected_Graphs/}}
\usepackage{amsmath,amsthm,amssymb,cite,breqn,appendix}
\usepackage{alltt,placeins,caption,subcaption,tikz,float}
\usetikzlibrary{arrows.meta}
\usepackage{orcidlink}

\newtheorem*{mainresult*}{Main Result}
\theoremstyle{plain}

\usepackage{hyperref}

\begin{document}

\title[Classifying Character Degree Graphs with Eight Vertices]{Classifying Prime Character Degree Graphs With Eight Vertices}

\author[M. L. Lewis]{Mark L. Lewis\,
 \orcidlink{0000-0001-9627-6922}
}
\address{Department of Mathematical Sciences, Kent State University | Kent, Ohio 44242, USA}
\email{lewis@math.kent.edu} 

\author[A. Summers]{Andrew Summers\,
\orcidlink{0009-0005-5561-6309}
}
\address{Department of Mathematical Sciences, Kent State University | Kent, Ohio 44242, USA}
\email{asumme19@kent.edu} 





\begin{abstract}
  In this paper, an effort is made to classify which prime character degree graphs having eight vertices occur for some finite solvable group. To approach this, we compile known results and constructions from the literature which are used to develop a general algorithm to begin classifying graphs of any order. We then apply the algorithm to the graphs of order eight. Of the 12,346 non-isomorphic graphs with eight vertices, 1,229 are disconnected and are fully classified. Meanwhile, thirty-seven of the 11,117 non-isomorphic connected graphs are shown to occur; thirty-four of which are constructed via joins of smaller graphs and three of which have diameter three. Fifty-eight graphs are shown not to occur, several of which fall into previously studied families, while the classification of 204 graphs is still unknown.
\end{abstract}

\maketitle

\section*{Introduction}\label{S:intro}

 Throughout this paper, $G$ will be a finite solvable group. Let $\textup{Irr}(G)$ and $\textup{cd}(G)$ denote the set of irreducible ordinary characters and irreducible ordinary character degrees of $G$ respectively. That is, $\textup{cd}(G)=\{\chi(1):\chi\in\textup{Irr}(G)\}$. We define $\rho(G)$ to be the set of all primes which divide some element of $\textup{cd}(G)$. With this notation in mind, we define the prime character degree graph (or simply character degree graph) of $G$. Denoted by $\Delta(G)$, this simple graph has vertex set $\rho(G)$ and an edge between two primes $p,q\in\rho(G)$ if $pq\mid a$ for some $a\in\textup{cd}(G)$. Thus, it follows that the primes in $\rho(G)$ and the vertices of $\Delta(G)$ are in one-to-one correspondence and will frequently be used interchangeably. There has been significant investigation into character degree graphs (for
 a summary see \cite{lewis2008overview} and for background see \cite{isaacs1994character} as well as Chapter 5, Section 18 of \cite{manz1993representations}).

 The main goal of this paper is twofold. First, to develop an algorithmic approach to classifying character degree graphs with a fixed number of vertices. In doing so, we approach our second goal of identifying precisely which graphs are resistant to the techniques found in the literature and discussing why these current methods fail. To approach these goals, we compile known results and constructions, both classical and contemporary, from the literature and organize them into an algorithm which we apply to the graphs of order eight. We note that this algorithm is not specific to graphs of order eight and can potentially be used to begin a classification of graphs with any number of vertices. Once this algorithm is applied, we are left in a position to more closely identify characteristics of unclassified graphs in an effort to draw attention to the need for new techniques.
 
 Such classifications have previously taken place for graphs of smaller order. In \cite{huppert1991research}, Huppert lists all graphs $\Gamma$ with four or fewer vertices that ``occur'', that is, $\Gamma=\Delta(G)$ for some finite solvable group $G$. In \cite{lewis2004classifying}, Lewis completes this classification for graphs of order five with the exception of one graph. Bissler and Lewis classify the occurring graphs with six vertices in \cite{bissler2019classifying}, with nine graphs eluding classification. Most recently, Laubacher, Medwid and Schuster investigated graphs of order seven in \cite{laubacher2023classifying}, where forty-four graphs remained unclassified. We summarize the results of these classifications, along with our own results in \autoref{fig:Table} of Section \ref{Sec:Conc}, and note that the same table also appears in \cite{laubacher2026bowtie}.

 Using the techniques introduced in \cite{MCKAY201494}, one can list 12,346 non-isomorphic graphs with eight vertices. Of these, 11,117 are connected, while the remaining 1,229 are disconnected. Thankfully, a significant number of these graphs can be tamed by applying a combination of P\'alfy's landmark result from \cite{palfy1998character} and its generalization in \cite{akhlaghi2018character}. Both results will be stated, along with several other important preliminaries, in the following \hyperref[Sec:Machinery]{section}. The classification algorithm will be presented and discussed in Section \ref{Sec:Algorithm}. In Section \ref{Sec:Disconn}, the disconnected graphs will be fully classified, while Section \ref{Sec:Conn} will explore the connected graphs, sorted by largest occurring clique size. Finally, in Section \ref{Sec:Conc}, we discuss conclusions drawn from the unclassified graphs and the the current state of known techniques. Our main result is as follows:
 \begin{mainresult*}
 The graphs with eight vertices which occur as $\Delta(G)$ for some finite solvable group $G$ are the graphs in figures $\ref{fig:OccDis}, \ref{fig:AGraphs}, \ref{fig:BGraphsDirProd}, \ref{fig:OccD3}, \ref{fig:CGraphsDirProd}, \ref{fig:DOcc}$ and possibly those in \hyperref[fig:UncB]{Figure $13$} as well as the graphs $C_k$ and $D_k$ mentioned in the final paragraphs of subsections $\ref{Sec:Unc5and3}$ and $\ref{Sec:Unc4and4}$.
 \end{mainresult*}

 Finally, we point out that the work in this paper was done by the second author as a Ph.D. candidate under the supervision of the first author at Kent State University, and the contents may appear in the second author's Ph.D. dissertation.

\section{Machinery}\label{Sec:Machinery}

  In an effort to make this paper as self-contained as possible, many of the important results used to classify occurring graphs will be given in this section in addition to some common graph-theoretic terminology and a method for constructing new character degree graphs from known occurring ones.
\subsection{Overall Results}\label{Sec:OverallResults}
As previously stated, a major result which eliminates large numbers of graphs is due to P\'alfy:

\begin{thm}[P\'alfy's Condition \cite{palfy1998character}]\label{thm:PalfysC} Let $G$ be a solvable group and $\pi$ be a set of primes contained in $\Delta(G)$. If $\lvert\pi\rvert=3$, then there exists an irreducible  character of $G$ with degree divisible by at least two of the primes from $\pi$. (In other words, any three vertices of the prime character degree graph of a solvable group span at least one edge.)
\end{thm}

Graphs for which Theorem \ref{thm:PalfysC} holds are said to satisfy P\'alfy's condition. Using more graph theoretic terminology, P\'alfy's condition states that the complement graph $\overline{\Delta}(G)$ contains no triangles. This result was generalized further in \cite{akhlaghi2018character} to show that, in fact, $\overline{\Delta}(G)$ contains no cycles of any odd length.

\begin{thm}[Theorem A of \cite{akhlaghi2018character}]\label{thm:Akh} Let $G$ be a finite solvable group. Then the graph $\overline{\Delta}(G)$ does not contain any cycle of odd length.
\end{thm}

Recall that a complete graph is a graph in which all vertices are pairwise adjacent to one another. We denote the complete graph on $n$ vertices using the common notation $K_n$. Noting that a \textbf{clique} of a graph $\Gamma$ is simply a complete subgraph of $\Gamma$, we also have the following important corollary to Theorem \ref{thm:Akh}: 

\begin{cor}[Corollary B of \cite{akhlaghi2018character}]\label{Cor:AkhB} Let $G$ be a finite solvable group. Then $\rho(G)$ is covered by two subsets, each inducing a clique in $\Delta(G)$.
\end{cor}

Corollary \ref{Cor:AkhB} gives a natural way of sorting the graphs we wish to classify. Namely, we will sort the connected graphs with eight vertices by the largest occurring clique size (as done in \cite{laubacher2023classifying}). These different clique size pairs will be represented in subsections \ref{Sec:Cl7and1} through \ref{Sec:Cl4and4}.

It is quite common to encounter disconnected graphs or subgraphs when studying prime character degree graphs. Note that by P\'alfy's condition, if $\Delta(G)$ is disconnected, it must have exactly two connected components, each of which is a complete graph. An additional result of P\'alfy's strengthens the already strong conditions placed on a disconnected character graph by showing that one connected component must be sufficiently larger than the other.

\begin{thm}[Theorem 3 of \cite{palfy2001character}]\label{thm:PalfysI}
    Suppose that the character degree graph of the solvable group $G$ has two connected components and let the cardinalities of the two components be $n$ and $N$ with $n\leq N$. Then $N\geq2^n-1$.
\end{thm}

This result, often referred to as P\'alfy's inequality, will help significantly when classifying the disconnected character degree graphs in Section \ref{Sec:Disconn}. Additionally, in \cite{lewis2025number}, a function is presented which gives the number of component size pairs that satisfy P\'alfy's inequality in terms of the order of the graphs being studied. These results together allow for the non-occurring disconnected graphs of any order to be quickly eliminated.

Finally, recall that a \textbf{cut vertex} of a graph $\Gamma$ is a vertex whose removal causes $\Gamma$ to have a greater number of connected components. Lewis and Meng showed in \cite{lewiscutvertex} that an occurring graph has at most one cut vertex. 

\subsection{Graphs of Diameter Three}\label{Sec:GraphsofDiameterThree}
Recall that the \textbf{diameter} of a graph is the longest of all shortest paths between any two vertices. In \cite{wolf1989diameter}, it was shown that for a finite solvable group $G$, $\Delta(G)$ has diameter at most three, although no examples were known at the time. In \cite{lewis2002solvable}, the first author constructed an example of a group whose character degree graph has six vertices and diameter three, confirming that such groups do exist. Moreover, it was shown in \cite{lewis2002solvable5} that six is the minimum number of vertices in any occurring graph of diameter three. As such, we can expect to encounter graphs of diameter three in our investigation. The following notation, introduced by Lewis in \cite{lewis2002solvable5}, gives rise to a result (Theorems 2, 3, and 4 of \cite{sass2016character}) which will help tame them substantially.

Suppose $\Gamma$ is a graph of diameter three, where $p$ and $q$ are distance three from one another. One can partition the vertex set of $\Gamma$ into four non-empty disjoint subsets: $\textup{V}(\Gamma)=\rho_1\cup\rho_2\cup\rho_3\cup\rho_4$ where the $\rho_i$ are defined as follows:
\begin{enumerate}
    \item[] $\rho_4$ is the set of all vertices which are distance three from $p$ (i.e. $q\in\rho_4$);
    \item[] $\rho_3$ is the set of all vertices which are distance two from $p$;
    \item[] $\rho_2$ is the set of all vertices which are adjacent to $p$ and some vertex in $\rho_3$;
    \item[] $\rho_1$ is the set containing $p$ and all vertices which are adjacent to $p$ but not to any vertex in $\rho_3$.
\end{enumerate}

This partition depends on the choices of $p$ and $q$. Following \cite{sass2016character}, we relabel if needed so that $\lvert\rho_1\cup\rho_2\rvert\leq\lvert\rho_3\cup\rho_4\rvert$. Two important results which help to significantly tame graphs of diameter three are as follows:

\begin{lem}[Theorems 2 and 4 of \cite{sass2016character}]\label{lem:SassRho} Let $G$ be a solvable group where $\Delta(G)$ has diameter three. One then has the following:
    \begin{enumerate}
        \item[i.] $\lvert\rho_3\rvert\geq 3$;
        \item[ii.] If $\lvert\rho_1\cup\rho_2\rvert=n$, then $\lvert\rho_3\cup\rho_4\rvert\geq 2^n$;
    \end{enumerate}
\end{lem}

The construction of Lewis in \cite{lewis2002solvable}, which led to a graph of diameter three, was generalized by Dugan in \cite{dugan2007solvable}. This construction is quite helpful in showing the occurrence of certain graphs which have diameter three and are covered by cliques of sizes $n-2$ and $2$. We will outline this construction in greater detail when using it in subsection \ref{Sec:Cl6and2}.

\subsection{Admissible Vertices and Families of Graphs}\label{Sec:Admissible}
Another useful tool for showing a graph does not occur is that of admissible vertices, which were introduced in \cite{bissler2025family}. 

\begin{defi}[from \cite{bissler2025family}]\label{defi:AdmissV}
    Let $\Gamma$ be a graph and $p\in\textup{V}(\Gamma)$ be a vertex of $\Gamma$. Consider the following three conditions:
    \begin{enumerate}
        \item[i.] The subgraph of $\Gamma$ obtained by removing $p$ and all incident edges does not occur;
        \item[ii.] All subgraphs of $\Gamma$ obtained by removing one or more edges incident to $p$ do not occur;
        \item[iii.] All subgraphs of $\Gamma$ obtained by removing $p$, the edges incident to $p$, and one or more edges between two adjacent vertices of $p$ do not occur.   
    \end{enumerate}

    If the vertex $p$ satisfies i. and ii. then it is called \textbf{admissible}. If $p$ also satisfies iii. then it is called \textbf{strongly admissible}.
\end{defi}

Much can be said about both $G$ and $\Delta(G)$ once admissible or strongly admissible vertices have been identified. The main result which we will make use of is the following:

\begin{lem}[Lemma 2.3 of \cite{bissler2025family}]\label{lem:Admiss}
    Assume that $\Gamma$ is a non-empty graph. If $\Gamma$ is a graph in which every vertex is admissible, then $\Gamma$ is not $\Delta(G)$ for any solvable group $G$.
\end{lem}

Additionally, there have been several families of graphs with specific constructions which have been investigated; often utilizing admissible vertices to show that their members cannot occur.

The first such family, denoted $\Gamma_{k,t}$ (where $k\geq t\geq 1$), is constructed by taking the complete graphs $K_t$ and $K_k$ and mapping edges injectively from $K_t$ to $K_k$. The result is a graph of diameter at most two which satisfies P\'alfy's condition. The findings for this family are summarized in the following result.

\begin{thm}[Main Theorem of \cite{bissler2019classifyingfamilies}]\label{thm:Gamma} The graph $\Gamma_{k,t}$ occurs as the prime character degree graph of a solvable group precisely when $t=1$ or $k=t=2$. Otherwise $\Gamma_{k,t}$ does not occur as the prime character degree graph of any solvable group.
\end{thm}

This family was generalized further by Laubacher and Medwid in \cite{laubacher2021prime}, by DeGroot, Laubacher and Medwid in \cite{degroot2022prime} and finally in full generality by Bissler, Debowski, Hoelker, Laubacher, Ravaglia and Sivanesan in \cite{bissler2026prime}. The results of these investigations are two similar families denoted $\Sigma_{k,n}^{mL}$ and $\Sigma_{k,n}^{mR}$ where edges are mapped between the complete graphs $K_k$ and $K_{k+n}$ in a one-to-two fashion and the $L$ or $R$ superscript denotes which complete graph interacts with $m$ ``special vertices''. We refer to \cite{laubacher2021prime}, \cite{degroot2022prime}, and \cite{bissler2026prime} for the full construction details. The findings for these families are summarized in the following results.

\begin{thm}[Theorem A of \cite{bissler2026prime}]\label{thm:FamA}
    The graph $\Sigma^{mL}_{k,n}$ occurs as the prime character degree graph of a solvable group when $(k,n,m)\in\{(1,1,i)\mid 1\leq i\leq 35\}$, and possibly when $(k,n,m)\in\{(1,1,i)\mid i>35\}$. Otherwise, $\Sigma^{mL}_{k,n}$ does not occur as the prime character degree graph of any solvable group.
\end{thm}

\begin{thm}[Theorem B of \cite{bissler2026prime}]\label{thm:FamB}
    The graph $\Sigma^{mR}_{k,n}$ occurs as the prime character degree graph of a solvable group when $(k,n,m)\in\{(1,1,i)\mid i\in\mathbb{N}\}$, and possibly when $(k,n,m)\in\{(2,1,i),(2,2,i)\mid i\in\mathbb{N}\}$. Otherwise, $\Sigma^{mR}_{k,n}$ does not occur as the prime character degree graph of any solvable group.
\end{thm}

\begin{thm}[Theorem C of \cite{bissler2026prime}]\label{thm:FamC}
    The graph $\Sigma^{0}_{k,n}$ occurs as the prime character degree graph of a solvable group when $(k,n)\in\{(1,1),(2,1),(2,2)\}$. Otherwise, $\Sigma^{0}_{k,n}$ does not occur as the prime character degree graph of any solvable group.
\end{thm}

Additionally, in \cite{bissler2025family}, Bissler, Laubacher and Lewis ruled out a family of graphs whose vertices satisfy certain adjacency conditions.

\begin{thm}[Main Theorem of \cite{bissler2025family}]\label{thm:BissFam}
    Let $\Gamma$ be a graph satisfying P\'alfy's condition, with $k\geq 5$ vertices. Assume that there exist two vertices $p_1$ and $p_2$ in $\Gamma$, such that $p_1$ and $p_2$ are of degree two, $p_1$ is adjacent to $p_2$, and they share no common neighbor. Then $\Gamma$ is not the prime character degree graph of any solvable group.
\end{thm}

Next, we briefly discuss a well-known way of constructing new occurring character degree graphs from occurring graphs of smaller order. Given groups $G$ and $H$, their direct product is denoted $G\times H$. Following \cite{isaacs1994character}, we note the identity $\textup{cd}(G\times H)=\textup{cd}(G)\cdot\textup{cd}(H)$. As a consequence, one sees that $\rho(G\times H)=\rho(G)\cup\rho(H)$ and that distinct primes $p_1,p_2\in\rho(G\times H)$ are adjacent in $\Delta(G\times H)$ if and only if they are adjacent in $\Delta(G)$, they are adjacent in $\Delta(H)$, or one lies in $\rho(G)$ while the other lies in $\rho(H)$. When $\rho(G)\cap\rho(H)=\emptyset$, this shows that the graph $\Delta(G\times H)$ is merely the join of the graphs $\Delta(G)$ and $\Delta(H)$. That is, $\Delta(G\times H)=\Delta(G)\vee\Delta(H)$.

One particularly useful instance of this involves complete graphs. Namely, we have that the complete graph $K_n\cong K_{n-1}\vee K_1$. Given that any non-abelian $p$-group will have a character degree graph consisting of a single vertex, the above graph join construction can be used inductively to show that the complete graph $K_n$ occurs for all $n\in\mathbb{N}$. We note that since any prime $p$ may be used, one can always guarantee disjoint vertex sets when constructing joins involving complete graphs. Such constructions will be useful when investigating the connected graphs in subsections \ref{Sec:Cl7and1} through \ref{Sec:Cl4and4}.

Finally, recall that a graph is called \textbf{$k$-regular} if every vertex has order $k$. In other words, each vertex is adjacent to exactly $k$ other vertices. The complete graph $K_n$ is the unique $(n-1)$-regular graph with $n$ vertices and always occurs by the graph join construction of the previous paragraph. In general, there is a $k$-regular graph with $n$ vertices for all pairs $k$ and $n$ such that $n-1\geq k$ and $nk$ is even. When considering which regular graphs can occur as a character degree graph, Theorem A of \cite{MORRESIZUCCARI2014215} puts heavy restrictions on $k$ while Theorem B reveals some additional structure about the character degree graph given some information about the underlying group. These results are as follows:

\begin{thm}[Theorems A and B of \cite{MORRESIZUCCARI2014215}]\label{thm:Reg}
    If $\Gamma$ is a non-complete and regular character degree graph of a finite solvable group, then $\Gamma$ is $(n-2)$-regular. Moreover, if $\Gamma$ is $(n-2)$-regular and $G$ has no normal non-abelian Sylow subgroups, then $G$ is a direct product of groups having disconnected character degree graphs.
\end{thm}

\section{Algorithm}\label{Sec:Algorithm}
The goal of this section is to outline a general algorithm for investigating which character degree graphs (of any order) occur. We do so by describing the general process which will be used to classify graphs throughout the remainder of the paper.

After identifying the total number of graphs of a given order, the most efficient way to eliminate large numbers of graphs is via the conditions of P\'alfy and Akhlaghi et al. (\cite{palfy1998character} and \cite{akhlaghi2018character}). Afterwards, the connected and disconnected graphs can be dealt with via separate techniques. In the case of eight vertices, the initial 12,346 graphs can be narrowed down to four disconnected possibilities and 299 connected possibilities by applying these conditions. From there, P\'alfy's inequality and the $c(n)$ function which counts the number of component size pairs which satisfy P\'alfy's inequality (of \cite{palfy2001character} and \cite{lewis2025number} respectively) determine which disconnected possibilities do not occur, while the construction of Lewis from \cite{lewis2004classifying} proves to be an excellent tool for constructing examples of the occurring disconnected graphs. In our scenario, this leaves us with two occurring disconnected graphs and two which do not occur.

For the remaining connected graphs, the results of \cite{akhlaghi2018character} allow us to sort by the largest occurring clique size. For graphs having the largest possible size differences (seven and one in our case), we see via graph joins using the $\Gamma_{k,t}$ family of \cite{bissler2019classifyingfamilies} that all such graphs occur. Graphs whose cliques have smaller size differences can then be sorted by diameter.

Graphs having diameter three are substantially tamed by a combination of the $\rho_i$ notation from \cite{lewis2002solvable5} and the corresponding results of \cite{sass2016character}, which eliminate a majority of the possible graphs. Those which still may occur can often be constructed using the technique introduced in \cite{lewis2002solvable} and generalized by Dugan in \cite{dugan2007solvable}. In our case, we completely classify the graphs having eight vertices and diameter three using these techniques.

We note next that the only connected graph of diameter one is the complete graph which occurs by the remarks at the end of the previous section. This leaves connected graphs of diameter two, which at this time require the widest variety of techniques to approach and would benefit most from the development of new techniques. We investigate such graphs using a combination of previously studied families (\cite{bissler2019classifyingfamilies},\cite{bissler2025family},\cite{laubacher2021prime},\cite{degroot2022prime}), the results on admissible vertices introduced in \cite{bissler2025family}, the result limiting the number of cut vertices from \cite{lewiscutvertex}, the regularity results of \cite{MORRESIZUCCARI2014215} and joins of known occurring graphs with fewer vertices.

A summary of the techniques outlined above can be seen in \autoref{fig:Algorithm}.

\FloatBarrier

\tikzstyle{block}=[rectangle, rounded corners, minimum width=2cm, minimum height=1cm,  text centered, text width=2cm, draw=black, fill=blue!20]

\begin{figure}[!ht]
    \centering
    \begin{tikzpicture}[node distance=2cm, scale=0.81, transform shape]
        \node(start)[block]{All graphs of order $n$};
        \node(disconn)[block, below of=start, yshift=-1cm, xshift=-2cm]{Disconnected graphs};
        \node(conn)[block, below of=start, yshift=-1cm, xshift=2cm]{Connected graphs};
        \node(poss disc)[block, below of=disconn, yshift=-1.5cm, xshift=-4cm]{Possibly occurring disconnected graphs};
        \node(non disc)[block, below of=disconn, yshift=-1.5cm]{Non-occurring disconnected graphs};
        \node(clique small)[block, below of=conn, yshift=-1cm]{Cliques of smaller size differences};
        \node(clique n)[block, below of=conn, yshift=-1cm, xshift=4cm]{Cliques of sizes $(n-1)$ and $1$};
        \node(occ disc)[block, below of=poss disc, yshift=-8.5cm]{Occurring disconnected graphs};
        \node(occ)[block, below of=clique n, yshift=-9cm]{Occurring connected graphs};
        \node(diam1)[block, below of=clique small, yshift=-2.1cm, xshift=-1.75cm]{Graphs of diameter $1$ or $2$};
        \node(nonocc)[block, below of=diam1, yshift=-6cm]{Non-Occurring connected graphs};
        \node(diam3)[block, below of=clique small, yshift=-2cm, xshift=1.75cm]{Graphs of diameter $3$};

        \draw [-{Stealth[length=3mm]}] (start) -| node[draw=white, fill=white, below of=start, yshift=0.7cm, xshift=2cm] {P\'alfy/Akhlaghi's Condition (\cite{palfy1998character} and \cite{akhlaghi2018character})} (disconn);
        \draw [-{Stealth[length=3mm]}] (start) -| node[draw=white, fill=white, below of=start, yshift=0.7cm, xshift=-2cm] {P\'alfy/Akhlaghi's Condition (\cite{palfy1998character} and \cite{akhlaghi2018character})} (conn);
        \draw [-{Stealth[length=3mm]}] (disconn) -| node[draw=white, fill=white, below of=disconn, yshift=0.85cm, xshift=2cm] {P\'alfy's Inequality/$c(n)$ function (\cite{palfy2001character} and \cite{lewis2025number})} (poss disc);
        \draw [-{Stealth[length=3mm]}] (disconn) -- node[draw=white, fill=white, below of=disconn, yshift=2.3cm, xshift=-2cm] {P\'alfy's Inequality/$c(n)$ function (\cite{palfy2001character} and \cite{lewis2025number})} (non disc);
        \draw [-{Stealth[length=3mm]}] (poss disc) -- node[text width= 3cm, text centered, draw=white, fill=white, above of=occ disc, yshift=-2cm]{Galois field construction from \cite{lewis2004classifying} and families from \cite{lewis2001solvable}} (occ disc);
        \draw [-{Stealth[length=3mm]}] (conn) -- node[draw=white, fill=white, below of=conn, yshift=2.15cm, xshift=2cm]{Sort by largest clique size (via \cite{akhlaghi2018character})} (clique small);
        \draw [-{Stealth[length=3mm]}] (conn) -| node[draw=white, fill=white, below of=conn, yshift=0.75cm, xshift=-2cm]{Sort by largest clique size (via \cite{akhlaghi2018character})} (clique n);
        \draw [-{Stealth[length=3mm]}] (clique small) -| node[fill=white, draw=white, text width=4.5cm, text centered, below of=clique small, xshift=1.7cm,yshift=0.5cm] {Sort by diameter\\ (possible by \cite{wolf1989diameter} and \cite{lewis2002solvable})} (diam1);
        \draw [-{Stealth[length=3mm]}] (clique small) -| node[fill=white, draw=white, text width=4.5cm, text centered, below of=clique small, xshift=-1.7cm,yshift=0.5cm] {Sort by diameter\\ (possible by \cite{wolf1989diameter} and \cite{lewis2002solvable})} (diam3);
        \draw [-{Stealth[length=3mm]}] (clique n) -- node[fill=white, draw=white, text width=2.6cm, text centered, below of=clique n, yshift=3.5cm, xshift=0.5cm] {$\Gamma_{n-i,1}\vee K_{i-1}$ (via \cite{bissler2019classifyingfamilies} and \cite{isaacs1994character})} (occ);
        \draw [-{Stealth[length=3mm]}] (diam1) -- node[fill=white, draw=white, text width=5.5cm, text centered, below of=diam1, yshift=1.2cm, xshift=-1.2cm] {$\Gamma_{k,t}$/$\Sigma_{k,n}^{m*}$ families, multiple cut vertices, admissible vertices, regularity results\\ (\cite{bissler2019classifyingfamilies}, \cite{laubacher2021prime}, \cite{degroot2022prime}, \cite{bissler2026prime}, \cite{lewiscutvertex}, \cite{bissler2025family}, \cite{MORRESIZUCCARI2014215})} (nonocc);
        \draw [-{Stealth[length=3mm]}] (diam3) |- node[fill=white, draw=white, text width=3cm, text centered, right of=nonocc, xshift=-2.6cm, yshift=1cm] {$\rho_i$ results from \cite{lewis2002solvable5} and \cite{sass2016character}} (nonocc);
        \draw [-{Stealth[length=3mm]}] (diam1) -- node[fill=white, draw=white, text width=2.9cm, text centered, below of=diam1, yshift=2.9cm, xshift=-1.1cm] {$\Gamma_{k,t}$/$\Sigma_{k,n}^{m*}$ families, graph joins (\cite{bissler2019classifyingfamilies}, \cite{laubacher2021prime}, \cite{degroot2022prime}, \cite{bissler2026prime}, \cite{isaacs1994character})} (occ);
        \draw [-{Stealth[length=3mm]}] (diam3) -- node[fill=white, draw=white, text width=1.9cm, text centered, below of=diam3, yshift=2.5cm, xshift=0.04cm] {Skew ring construction from \cite{lewis2002solvable}/\cite{dugan2007solvable} and results from \cite{casolo2016groups}} (occ);
    \end{tikzpicture}
    \caption{The graph classification algorithm}
    \label{fig:Algorithm}
\end{figure}

\FloatBarrier

\clearpage

\section{Disconnected Graphs}\label{Sec:Disconn}
Up to isomorphism, there are 1,229 disconnected graphs with eight vertices. P\'alfy's condition, however, dictates that only those whose connected components consist of complete graphs may occur. This leaves only four possibilities to investigate, namely the disconnected graphs whose complete connected components have sizes seven and one, six and two, five and three, and four and four, respectively. As shown in \cite{lewis2025number}, the function $c(n):=\max\{\alpha\in\mathbb{Z}^+:n\geq 2^{\alpha}+\alpha -1\}$ calculates the number of component size pairs satisfying P\'alfy's inequality. Since $c(8)=2$, it follows that the two component size pairs with the greatest differences in size will satisfy P\'alfy's inequality and therefore the two disconnected graphs in \autoref{fig:NonOccDis} do not occur.

\begin{figure}[!ht]
   \centering
    \begin{subfigure}{.5\textwidth}
    \centering
    \includegraphics[width=.8\linewidth]{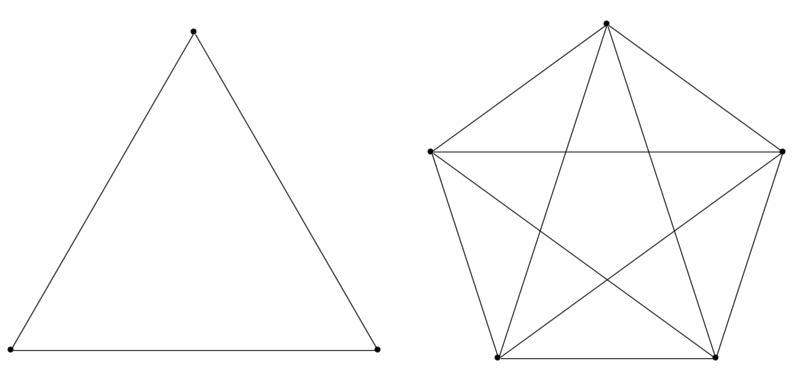}
    \end{subfigure}%
    \begin{subfigure}{.5\textwidth}
    \centering
    \includegraphics[width=0.8\linewidth]{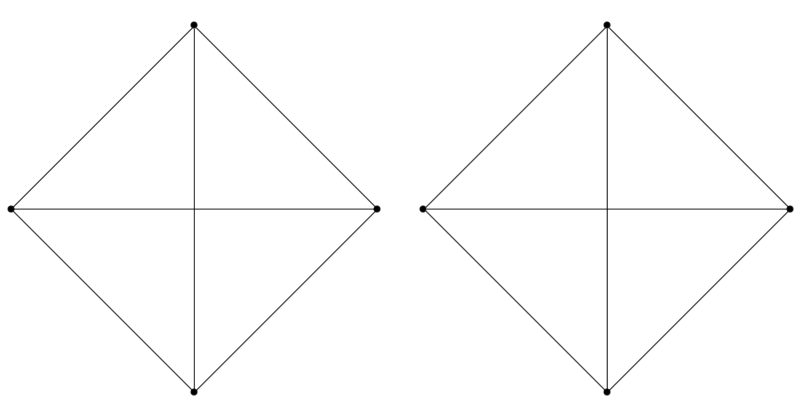}
    \end{subfigure}
    \caption{The two disconnected graphs which do not occur}
    \label{fig:NonOccDis}
\end{figure}

This leaves the two disconnected graphs with the largest differences in component sizes as possibly occurring. To see that the disconnected graphs in \autoref{fig:OccDis} do, in fact, occur, we will mimic the Galois field construction found in \cite{lewis2004classifying} for appropriate choices of prime powers. We note that the following construction is possible for primes other than 2, but a bit more caution is required when doing so. The prime 2, however, is sufficient for our examples.

\begin{figure}[!ht]
   \centering
    \begin{subfigure}{.5\textwidth}
    \centering
    \includegraphics[width=.8\linewidth]{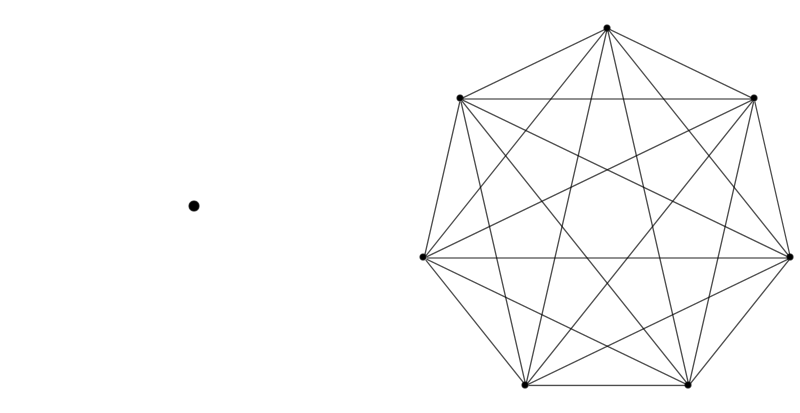}
    \end{subfigure}%
    \begin{subfigure}{.5\textwidth}
    \centering
    \includegraphics[width=0.8\linewidth]{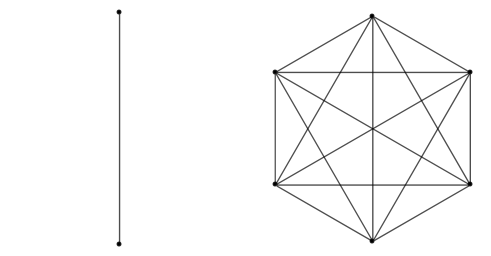}
    \end{subfigure}
    \caption{The two disconnected graphs which occur}
    \label{fig:OccDis}
\end{figure}

\subsection{Component Sizes 7 \& 1}\label{Sec:7and1}
Consider the finite field of order $2^{491}$ acted on by its full multiplicative group followed by its Galois group. That is, take $$G_1=\mathbb{F}_{2^{491}}\rtimes\mathbb{F}^{\times}_{2^{491}}\rtimes \textup{Gal}(\mathbb{F}_{2^{491}}/\mathbb{F}_2).$$
Note that $491$ is prime and that $2^{491}-1$ is the product of seven primes. Namely
\begin{align*}
2^{491}-1 = 983 &\times 7,707,719\\
&\times 110,097,436,327,057\\
&\times 6,976,447,052,525,718,623\\
&\times 19,970,905,118,623,195,851,890,562,673\\
&\times 3,717,542,676,439,779,473,786,876,643,915,388,439\\
&\times 14,797,326,616,665,978,116,353,515,926,860,025,681,383.
\end{align*}
 It follows that $\textup{cd}(G_1)=\{1,491,2^{491}-1\}$. Setting the prime factors of $2^{491}-1$ equal to $s,t,u,v,w,x,$ and $y$ respectively, we have $\Delta(G_1)$:

 \begin{figure}[!ht]
    \begin{center}
    \includegraphics[scale=0.4]{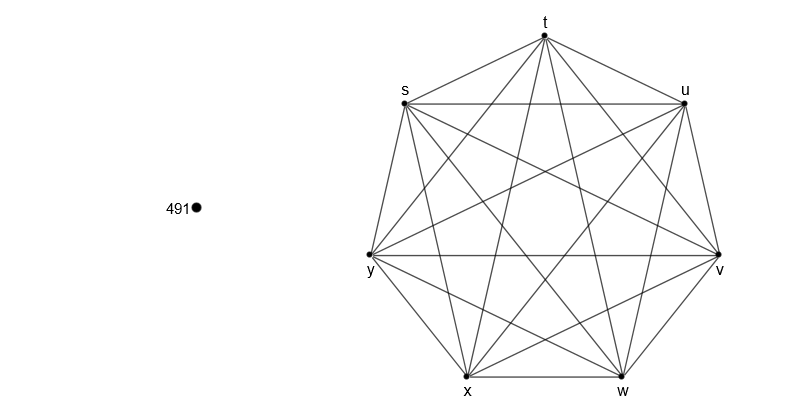}
    \caption{The occurring disconnected graph with component sizes seven and one}
    \label{fig:71L}
    \end{center}
\end{figure}

\FloatBarrier

\subsection{Component Sizes 6 \& 2}\label{Sec:6and2}
Now, consider the finite field of order $2^{143}$ acted on by its full multiplicative group followed by its Galois group. That is, take $$G_2=\mathbb{F}_{2^{143}}\rtimes\mathbb{F}^{\times}_{2^{143}}\rtimes \textup{Gal}(\mathbb{F}_{2^{143}}/\mathbb{F}_2).$$
Note that $143=11\times 13$ and that $2^{143}-1$ is the product of six primes. Namely 
\begin{align*}
    2^{143}-1 =23 &\times 89\times 8,191\times 724,153\times 158,822,951,431\\
    &\times 5,782,172,113,400,990,737.
\end{align*} 
It follows that $\textup{cd}(G_2)=\{1,11,13,143,2^{143}-1\}$. Setting $s=8,191$, $t=724,153$, $u=158,822,951,431$, and $v=5,782,172,113,400,990,737$, we have that $\Delta(G_2)$ is as follows:

 \begin{figure}[!ht]
    \begin{center}
    \includegraphics[width=0.5\textwidth]{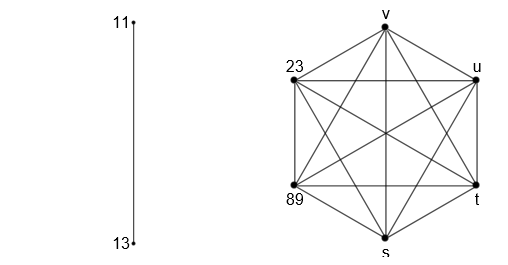}
    \caption{The occurring disconnected graph with component sizes six and two}
    \label{fig:62L}
    \end{center}
\end{figure}

\FloatBarrier

So, in total, we have that two disconnected graphs with eight vertices occur while the remaining 1,227 do not.

\section{Connected Graphs}\label{Sec:Conn}
As stated previously, the main results of \cite{palfy1998character} and \cite{akhlaghi2018character} are invaluable in ruling out large numbers of graphs. Of the 11,117 non-isomorphic connected graphs with eight vertices, only 299 satisfy both conditions and therefore may occur. It is these 299 graphs that will be the focus of the remainder of this paper. Corollary 2.3 (of \cite{akhlaghi2018character}) states that each graph's vertex set must be covered by two cliques, and thus cliques give a natural way to sort the remaining connected graphs. Following \cite{laubacher2023classifying}, we will sort the remaining graphs into appendices by the largest clique that appears. We note that for graphs of order larger than five, the size of the largest clique is not always enough to determine the clique sizes which cover a graph’s vertex set. For example, graph $B_{34}$ of \hyperref[fig:NonOccD3]{Figure 11} has a maximal clique of size six, but the vertex set is covered by cliques of sizes five and three. Since the largest clique size is six, it appears in subsection \ref{Sec:Cl6and2} and Appendix \hyperref[Sec:AppB]{B}.

Appendix \hyperref[Sec:AppA]{A} contains seven graphs with a largest clique size of seven. Con
sequently, these graphs are covered by cliques of sizes seven and one. These will be fully classified in subsection \ref{Sec:Cl7and1}. Appendix \hyperref[Sec:AppB]{B} contains forty-five graphs with
a largest clique size of six and will be investigated in subsection \ref{Sec:Cl6and2}. Appendix \hyperref[Sec:AppC]{C} houses one hundred fifty-one graphs whose maximal clique has five vertices and
will be the subject of subsection \ref{Sec:Cl5and3}. Finally, Appendix \hyperref[Sec:AppD]{D} contains ninety-six
graphs which have a maximal clique of size four and are therefore covered by two
cliques of size four. These graphs will be examined in subsection \ref{Sec:Cl4and4}. In total, thirty-seven graphs will be shown to occur while fifty-eight will be eliminated. This leaves two hundred and four unclassified graphs which possibly occur.

\subsection{Connected Graphs With a Largest Clique of Size Seven}\label{Sec:Cl7and1}
There are seven connected graphs whose largest occurring clique has seven vertices. It follows that their vertex sets must be covered by cliques of sizes seven and one. We show that all of these graphs occur in this section. Note first that the graph $A_1$ is precisely $\Gamma_{7,1}$ of \cite{bissler2019classifyingfamilies}, where it was shown that $\Gamma_{k,1}$ occurs for all $k\in\mathbb{N}$. We also have that the graph $A_7$ is simply the complete graph $K_8$, which is also known to occur. Finally, we see that $A_i\cong\Gamma_{8-i,1}\vee K_{i-1}$ for $2\leq i\leq 6$, so it follows that all connected graphs covered by cliques of sizes seven and one occur. These can be seen in \autoref{fig:AGraphs} or in Appendix \hyperref[Sec:AppA]{A}.

\begin{figure}[!ht]
    \centering
    \begin{subfigure}[t]{0.2\textwidth}
        \centering
        \includegraphics[height=1in]{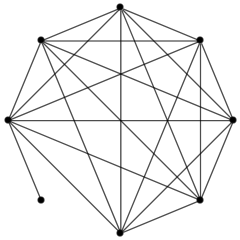}
        \caption*{$A_1$}
    \end{subfigure}
    \begin{subfigure}[t]{0.2\textwidth}
        \centering
        \includegraphics[height=1in]{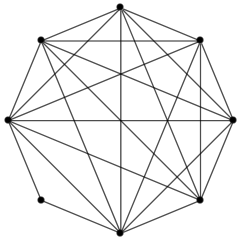}
        \caption*{$A_2$}
    \end{subfigure}
    \begin{subfigure}[t]{0.2\textwidth}
        \centering
        \includegraphics[height=1in]{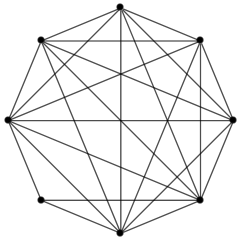}
        \caption*{$A_3$}
    \end{subfigure}
    \begin{subfigure}[t]{0.2\textwidth}
        \centering
        \includegraphics[height=1in]{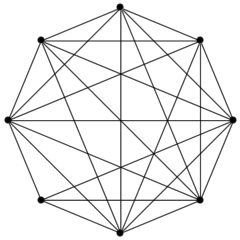}
        \caption*{$A_4$}
    \end{subfigure}
\end{figure}
\renewcommand\thefigure{\arabic{figure}}
\setcounter{figure}{5}
\begin{figure}[!ht]
    \centering
    \begin{subfigure}[t]{0.2\textwidth}
        \centering
        \includegraphics[height=1in]{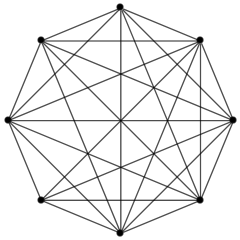}
        \caption*{$A_5$}
    \end{subfigure}
    \begin{subfigure}[t]{0.2\textwidth}
        \centering
        \includegraphics[height=1in]{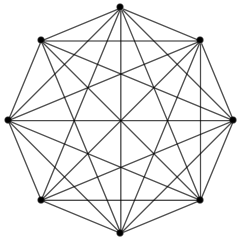}
        \caption*{$A_6$}
    \end{subfigure}
    \begin{subfigure}[t]{0.2\textwidth}
        \centering
        \includegraphics[height=1in]{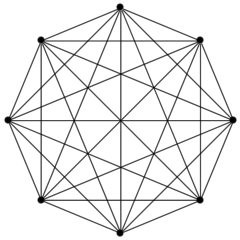}
        \caption*{$A_7$}
    \end{subfigure}
    \caption{Graphs $A_i$ for $1\leq i\leq 7$, all of which occur}
    \label{fig:AGraphs}
\end{figure}

\FloatBarrier

\subsection{Connected Graphs With a Largest Clique Size of Six}\label{Sec:Cl6and2}
There are forty-five graphs with a largest clique size of six. These graphs are listed in Appendix \hyperref[Sec:AppB]{B}. In this section, we show that nineteen of these graphs occur; sixteen of which will be constructed as joins of graphs, while the remaining three follow the construction in \cite{dugan2007solvable} for appropriate choices of primes. Additionally, we show that five graphs do not occur, leaving twenty-one graphs in Appendix \hyperref[Sec:AppB]{B} unclassified. These unclassified graphs will marked with an asterisk.

\subsection*{Graph Joins}
There are sixteen graphs with maximum clique size six which can be constructed via joins of occurring graphs. Namely, the graphs $B_i$ for $i\in\mathfrak{I}$ where $\mathfrak{I}:=\{2,6,12,15,21,32,35,37,38,39,40,41,42,43,44,45\}$. These graphs can be seen in \autoref{fig:BGraphsDirProd} or in Appendix \hyperref[Sec:AppB]{B}.

\begin{figure}[!ht]
    \centering
    \begin{subfigure}[t]{0.2\textwidth}
        \centering
        \includegraphics[height=1in]{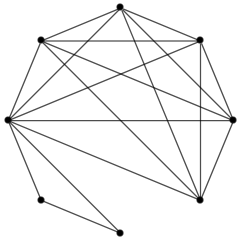}
        \caption*{$B_2$}
    \end{subfigure}
    \begin{subfigure}[t]{0.2\textwidth}
        \centering
        \includegraphics[height=1in]{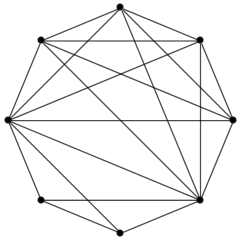}
        \caption*{$B_6$}
    \end{subfigure}
    \begin{subfigure}[t]{0.2\textwidth}
        \centering
        \includegraphics[height=1in]{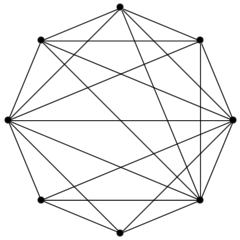}
        \caption*{$B_{12}$}
    \end{subfigure}
    \begin{subfigure}[t]{0.2\textwidth}
        \centering
        \includegraphics[height=1in]{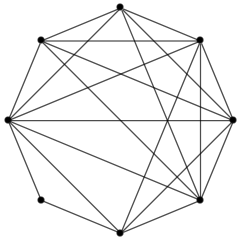}
        \caption*{$B_{15}$}
    \end{subfigure}
\end{figure}
\begin{figure}[!ht]
    \centering
    \begin{subfigure}[t]{0.2\textwidth}
        \centering
        \includegraphics[height=1in]{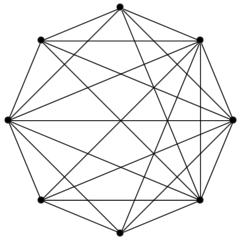}
        \caption*{$B_{21}$}
    \end{subfigure}
    \begin{subfigure}[t]{0.2\textwidth}
        \centering
        \includegraphics[height=1in]{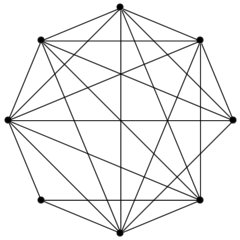}
        \caption*{$B_{32}$}
    \end{subfigure}
    \begin{subfigure}[t]{0.2\textwidth}
        \centering
        \includegraphics[height=1in]{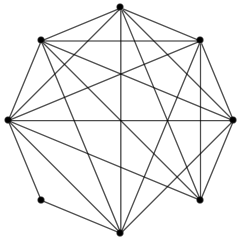}
        \caption*{$B_{35}$}
    \end{subfigure}
    \begin{subfigure}[t]{0.2\textwidth}
        \centering
        \includegraphics[height=1in]{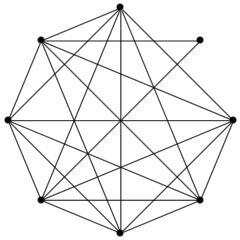}
        \caption*{$B_{37}$}
    \end{subfigure}
\end{figure}
\begin{figure}[!ht]
    \centering
    \begin{subfigure}[t]{0.2\textwidth}
        \centering
        \includegraphics[height=1in]{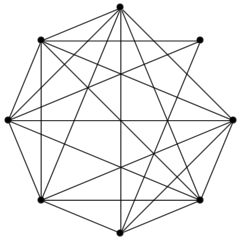}
        \caption*{$B_{38}$}
    \end{subfigure}
    \begin{subfigure}[t]{0.2\textwidth}
        \centering
        \includegraphics[height=1in]{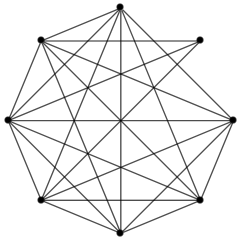}
        \caption*{$B_{39}$}
    \end{subfigure}
    \begin{subfigure}[t]{0.2\textwidth}
        \centering
        \includegraphics[height=1in]{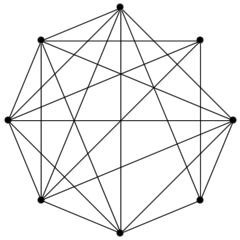}
        \caption*{$B_{40}$}
    \end{subfigure}
    \begin{subfigure}[t]{0.2\textwidth}
        \centering
        \includegraphics[height=1in]{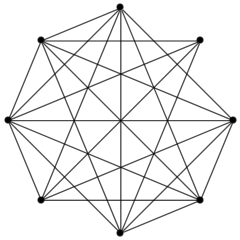}
        \caption*{$B_{41}$}
    \end{subfigure}
\end{figure}
\setcounter{figure}{6}
\begin{figure}[!ht]
    \centering
    \begin{subfigure}[t]{0.2\textwidth}
        \centering
        \includegraphics[height=1in]{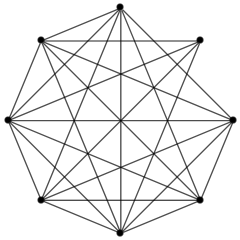}
        \caption*{$B_{42}$}
    \end{subfigure}
    \begin{subfigure}[t]{0.2\textwidth}
        \centering
        \includegraphics[height=1in]{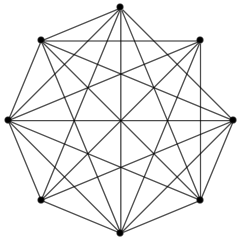}
        \caption*{$B_{43}$}
    \end{subfigure}
    \begin{subfigure}[t]{0.2\textwidth}
        \centering
        \includegraphics[height=1in]{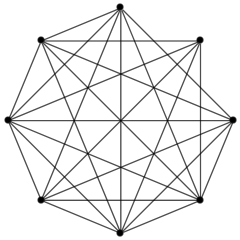}
        \caption*{$B_{44}$}
    \end{subfigure}
    \begin{subfigure}[t]{0.2\textwidth}
        \centering
        \includegraphics[height=1in]{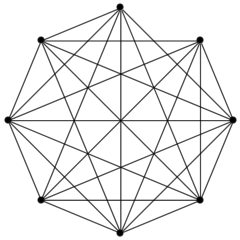}
        \caption*{$B_{45}$}
    \end{subfigure}
    \caption{The graphs $B_i$ for $i\in\mathfrak{I}$}
    \label{fig:BGraphsDirProd}
\end{figure}

\FloatBarrier

Note that graph $B_2$ is the join of the disconnected graph with component sizes five and two with $K_1$. Meanwhile, $B_{37}$ is the join of the disconnected graph with component sizes five and one with the empty (no edges) graph of order two. This graph can be realized by repeating the construction from \cite{bissler2019classifying}, which utilized the primes $2,3,5,17,257,$ and $65,537$ and taking the join with the empty graph constructed using the primes $p=7$ and $2^p-1=127$ and following the construction outlined in Section \ref{Sec:Disconn}. Graph $B_{38}$ can also be realized as the join of two disconnected graphs, namely the disconnected graph with component sizes four and one with the disconnected graph of component sizes two and one. This graph can be constructed using the values $m=2^4,2^m-1=3,\times5\times17\times257,p=11,$ and $2^p-1=23\times 89$, for example. Additionally, $B_{40}$ is the join of the disconnected graph with component sizes three and one with itself. For an example of primes achieving this graph, one could use $m=2^3,2^m-1=3\times5\times17,p=29,$ and $2^p-1=233\times1103\times2089$. All remaining graphs in \autoref{fig:BGraphsDirProd} are the join of an occurring graph in Appendix B of \cite{laubacher2023classifying} with $K_1$. 

\subsection*{Diameter Three Graphs}

There are seven graphs with maximum clique size six which have diameter three. In this section, we completely classify these graphs by showing that three occur and four do not. The occurring graphs can be seen in \autoref{fig:OccD3} and will be constructed in the manner of \cite{dugan2007solvable}. The non-occurring graphs can be seen in \autoref{fig:NonOccD3} and will be eliminated via techniques from \cite{sass2016character}. All seven graphs can also be found in Appendix \hyperref[Sec:AppB]{B}.

\subsection*{Constructions}
The graphs $B_{7}, B_{13},$ and $B_{36}$ can be constructed using the techniques in \cite{dugan2007solvable} which produce graphs of diameter three. This construction generalizes the construction done by Lewis in \cite{lewis2002solvable}. We outline this construction with appropriate choices of prime numbers to show that the graphs in \autoref{fig:OccD3} occur.

\renewcommand\thefigure{\arabic{figure}}
\setcounter{figure}{7}
\begin{figure}[!ht]
    \centering
    \begin{subfigure}[t]{0.2\textwidth}
        \centering
        \includegraphics[height=1in]{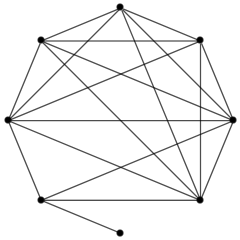}
        \caption*{$B_7$}
    \end{subfigure}
    \begin{subfigure}[t]{0.2\textwidth}
        \centering
        \includegraphics[height=1in]{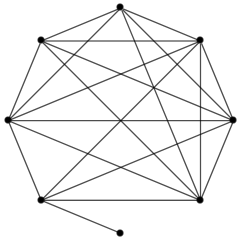}
        \caption*{$B_{13}$}
    \end{subfigure}
    \begin{subfigure}[t]{0.2\textwidth}
        \centering
        \includegraphics[height=1in]{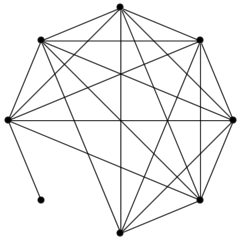}
        \caption*{$B_{36}$}
    \end{subfigure}
    \caption{Graphs $B_{7},B_{13},$ and $B_{36}$ which occur}
    \label{fig:OccD3}
\end{figure}

\FloatBarrier

Let $\Phi_n(x)$ be the $n^{th}$ cyclotomic polynomial and choose distinct primes $p,q,$ and $r$ such that $q<r$ and $\Phi_q(p),\Phi_r(p),$ and $\Phi_{qr}(p)$ are pairwise coprime. We note here that for the constructions of $B_{13}$ and $B_{36}$, it is sufficient to choose $q=3$, which results in a simpler computation of character degrees. For graph $B_{7}$, however, a choice of primes which produced the desired graph and included $q=3$ eluded us, so  $q=11$ is used instead. As noted in \cite{dugan2007solvable}, this complicates the calculation of character degrees and is the reason for listing $B_{7}$ separately below.

We start with the finite field $\mathbb{F}$ of order $p^{qr}$. Next, consider the skew polynomial ring $\mathbb{F}\{X\}$ with indeterminate $X$ and commutativity condition $Xf=f^pX$ for all $f\in\mathbb{F}$. Noting that the set generated by $X^{q+1}$ is an ideal, we may form the quotient ring $R:=\mathbb{F}\{X\}/(X^{q+1})$. Let $x$ be the image of $X$ in $R$. We then have that $xR=Rx$ is a nilpotent ideal such that $R/xR\cong\mathbb{F}$, making $xR$ a maximal ideal of $R$. Actually, $xR$ is the unique maximal ideal of $R$, and is therefore the Jacobson radical of $R$, which we denote by $J$. It follows from the quasiregularity of $J$ that $1+J$ is a subgroup of the group of units of $R$ in which each element has the form $1+f_1x+f_2x^2+\cdots+f_qx^q$. This implies that $\lvert 1+J\rvert=p^{q^2r}$. Following \cite{dugan2007solvable}, we let $P=1+J$.

Next, let $C$ be the unique subgroup of $\mathbb{F}^{\times}$ with order $\frac{p^{qr}-1}{p-1}$. Notice that there is an action via automorphisms of $C$ on $P$ given by $$(1+f_1x+f_2x^2+\cdots+f_qx^q)\cdot c=1+f_1c^{\frac{p-1}{p-1}}x+f_2c^{\frac{p^2-1}{p-1}}x^2+\cdots+f_qc^{\frac{p^q-1}{p-1}}x^q$$
for $c\in C$. Thus we can construct the semidirect product $T=P\rtimes C$.

Lastly, let $\mathfrak{G}$ be the Galois group of $\mathbb{F}$ over the prime subfield of order $p$. We construct the group $G$ as $G=T\rtimes\mathfrak{G}$. This results in a very large, but nonetheless finite group of order $qrp^{q^2r}\left(\frac{p^{qr}-1}{p-1}\right)$. Since we are concerned with graphs of order eight, we also require that the quantity $\frac{p^{qr}-1}{p-1}$ be the product of five primes (denoted below as $s,t,u,v,$ and $w$), distinct from each other as well as from $p,q,$ and $r$. Namely, we ask that $\frac{p^{qr}-1}{p-1}=stuvw$.

We are now ready to construct graphs $B_{13}$ and $B_{36}$. Recall that for these graphs $q=3$. Thus, the work done in \cite{dugan2007solvable} yields the following set of character degrees:

\begin{dmath*}
\textup{cd}(G)=\left\{1,3,r,3r,stuvw,p^{\frac{3r-1}{2}}stuvw,3p^{3r}\left(\frac{stuvw}{p^2+p+1}\right),p^{3r-3}\left(\frac{stuvw}{p^2+p+1}\right),\\3p^{3r-3}\left(\frac{stuvw}{p^2+p+1}\right),p^{3r-3}stuvw,p^{3r-2}stuvw \right\}.
\end{dmath*}
Now, let $p=23,q=3,$ and $r=13$. Say $G_{13}$ is the resulting group. It follows that $s=7, t=79, u=47,691,619, v=480,393,499,$ and \\ $w=459,408,054,528,299,360,264,076,035,007,841$. Thus, we have:

\begin{dmath*}
\textup{cd}(G_{13})=\left\{1,3,13,39,stuvw,23^{19}stuvw,3\cdot 23^{39}uvw, 23^{36}uvw,3\cdot 23^{36}uvw,\\23^{36}stuvw, 23^{37}stuvw \right\}.
\end{dmath*}
If instead we let $p=2,q=3,$ and $r=17$, it follows that $s=7, t=103, u=2,143, v=11,119,$ and $w=131,071$. Let the resulting group be $G_{36}$. Then we have:

\begin{dmath*}
\textup{cd}(G_{36})=\left\{1,3,17,51,stuvw,2^{25}stuvw,3\cdot 2^{51}tuvw, 2^{48}tuvw,3\cdot 2^{48}tuvw,\\2^{48}stuvw, 2^{49}stuvw \right\}.
\end{dmath*}
Using these sets of character degrees, we can construct $\Delta(G_{13})$ and $\Delta(G_{36})$, which are $B_{13}$ and $B_{36}$, respectively, as can be seen in \autoref{fig:Occ1336L}.

\begin{figure}[!ht]
    \centering
    \begin{subfigure}[t]{0.4\textwidth}
        \centering
        \includegraphics[height=1.8in]{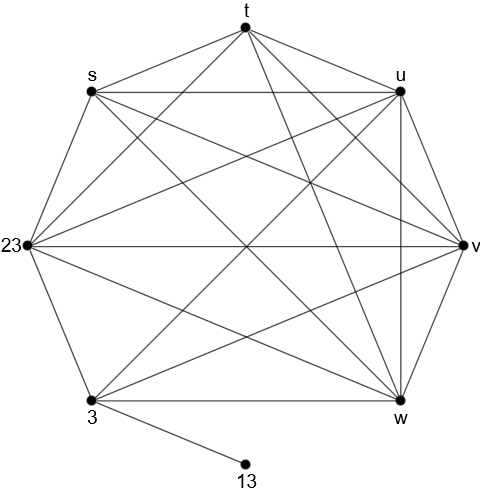}
        \caption*{$B_{13}$}
    \end{subfigure}
    \begin{subfigure}[t]{0.4\textwidth}
        \centering
        \includegraphics[height=1.8in]{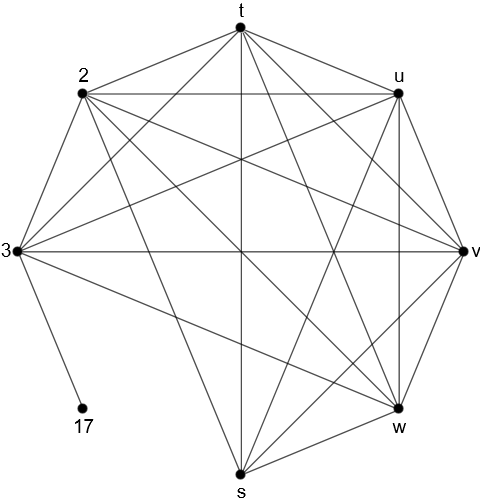}
        \caption*{$B_{36}$}
    \end{subfigure}
    \caption{The graphs $\Delta(G_{13})=B_{13}$ and $\Delta(G_{36})=B_{36}$}
    \label{fig:Occ1336L}
\end{figure}

\FloatBarrier

As mentioned previously, computing the character degrees when $q>3$ is more complicated, as there are more characters whose kernels correspond to the hyperplanes discussed in \cite{dugan2007solvable}. Under these circumstances, we have the following set of character degrees:

\begin{dmath*}
\textup{cd}(G)=\left\{1,q,r,qr,qp^{\frac{q-1}{2}qr}\left(\frac{stuvw}{p^{q-1}+p^{q-2}+\cdots+p+1}\right),\\p^{\frac{q-1}{2}(qr-q)}\left(\frac{stuvw}{p^{q-1}+p^{q-2}+\cdots+p+1}\right),qp^{\frac{q-1}{2}(qr-q)}\left(\frac{stuvw}{p^{q-1}+p^{q-2}+\cdots+p+1}\right),\\\left(p^{\frac{qr-1}{2}}\right)^istuvw,p^{\frac{q-1}{2}(qr-q+j-1)}stuvw\right\}.
\end{dmath*}
where $0\leq i\leq q-2$ and $0\leq j\leq q-1$.\\

Let $p=103, q=11,$ and $r=19$ and call the resulting group $G_7$. It follows that $s=199, t=27,457, u=24,837,228,904,511$,\\ $v=1,719,123,581,448,138,027,562,643,544,243,028,633$ and 
\begin{align*}
w= & 202,517,890,042,381,181,353,504,428,375,060,758,814,558,214,141,756-\\
& 480,294,956,207,468,325,371,830,271,245,325,465,652,657,228,450,147-\\
& 418,518,757,150,106,055,438,380,408,708,819,440,618,440,685,189,052-\\
& 052,263,963,649,184,438,396,170,168,226,675,258,860,911,793,155,342-\\
& 965,670,072,215,782,070,678,913,348,734,267,112,675,789,146,113,514-\\
& 521,901,130,275,814,893,800,628,444,591,073,088,415,226,218,929,809-\\
& 807,992,632,302,948,521,842,571,812,874,575,219,893,073,669,007,930-\\
& 985,649.
\end{align*}
Thus, we have the following set of character degrees for $G_7$:

\begin{dmath*}
\textup{cd}(G_7)=\left\{1,11,19,209,11\cdot 103^{5\times 209}vw,103^{5\times 198}vw,11\cdot 103^{5\times 198}vw,\\103^{104i}stuvw,103^{5\times (197+j)}stuvw\right\}
\end{dmath*}
where $0\leq i\leq 9$ and $0\leq j\leq 10$. Using this set of character degrees, we can construct $\Delta(G_7)$ as follows:

\begin{figure}[!ht]
    \begin{center}
        \includegraphics[width=0.4\textwidth]{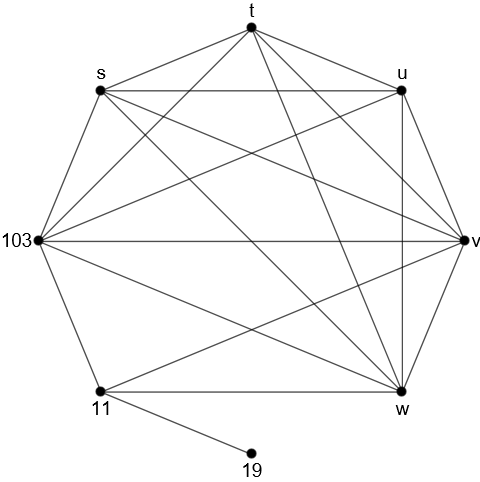}
        \caption{The graph $\Delta(G_7)=B_7$}
        \label{fig:B7L}
    \end{center}
\end{figure}

\FloatBarrier

\subsection*{Non-Occurring Diameter Three Graphs}
We now turn our attention to the four remaining diameter three graphs which
have a maximum clique size of six. These graphs can be found in \autoref{fig:NonOccD3} as well as Appendix \hyperref[Sec:AppB]{B}. To show that these graphs do not occur, we utilize the $\rho$ notation introduced by Lewis in \cite{lewis2002solvable5} and the corresponding results of Sass in \cite{sass2016character}.

\begin{figure}[!ht]
    \centering
    \begin{subfigure}[t]{0.2\textwidth}
        \centering
        \includegraphics[height=1in]{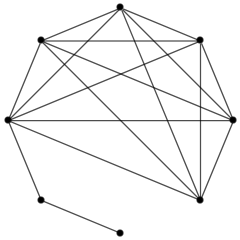}
        \caption*{$B_1$}
    \end{subfigure}
    \begin{subfigure}[t]{0.2\textwidth}
        \centering
        \includegraphics[height=1in]{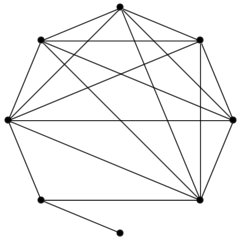}
        \caption*{$B_3$}
    \end{subfigure}
    \begin{subfigure}[t]{0.2\textwidth}
        \centering
        \includegraphics[height=1in]{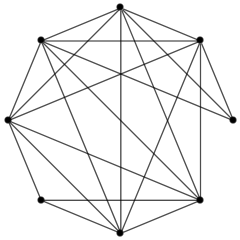}
        \caption*{$B_{30}$}
    \end{subfigure}
    \begin{subfigure}[t]{0.2\textwidth}
        \centering
        \includegraphics[height=1in]{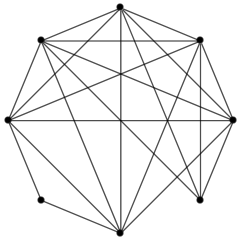}
        \caption*{$B_{34}$}
    \end{subfigure}
    \caption{The diameter three graphs $B_{1}, B_{3}, B_{30},$ and $B_{34}$ which do not occur}
    \label{fig:NonOccD3}
\end{figure}

\FloatBarrier

We first note that for each graph, there is exactly one pair of vertices which are distance three apart. Following \cite{lewis2002solvable5} and \cite{sass2016character}, we label these $p_1$ and $p_4$. Recall that $\rho_3$ is the set of vertices which are distance two from $p_1$. In Theorem 2 of \cite{sass2016character}, Sass showed that an occurring diameter three graph must satisfy $\lvert\rho_3\rvert\geq 3$. For graphs $B_1$ and $B_3$, regardless of which vertex in the pair is chosen as $p_1$, we have that $\lvert\rho_3\rvert<3$ and therefore $B_1$ and $B_3$ do not occur. We also wish to point out that the graph $B_1$ has two cut vertices and therefore could also be eliminated by the main result of \cite{lewiscutvertex}.

Recall also that $\rho_4$ is the set of vertices which are distance three from $p_1$. Additionally, $\rho_2$ is the set of vertices adjacent to both $p_1$ and a vertex in $\rho_3$. Lastly, $\rho_1$ is the set of vertices adjacent to $p_1$ and no vertex of $\rho_3$ as well as $p_1$ itself. Theorem 4 of \cite{sass2016character} states that, for an occurring diameter three graph, if $\lvert\rho_1\cup\rho_2\rvert=n$, then $\lvert\rho_3\cup\rho_4\rvert\geq 2^n$. For graph $B_{30}$ we have that $\lvert\rho_1\cup\rho_2\rvert=4$ but $\lvert\rho_3\cup\rho_4\rvert<16$. Similarly, for $B_{34}$ one can see that $\lvert\rho_1\cup\rho_2\rvert\geq 3$ but $\lvert\rho_3\cup\rho_4\rvert<8$. Thus $B_{30}$ and $B_{34}$ do not occur.

\subsection*{Previously Classified Graphs}
There is one remaining graph with maximum clique size six which has been previously studied. Graph $B_4$, pictured in \autoref{fig:Gamma62} and Appendix \hyperref[Sec:AppB]{B}, is isomorphic to $\Gamma_{6,2}$, which was shown not to occur in \cite{bissler2019classifyingfamilies}.

\begin{figure}[!ht]
    \begin{center}
        \includegraphics[width=0.2\textwidth]{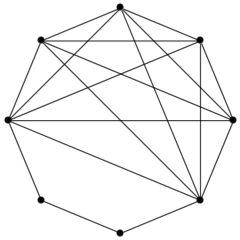}
        \caption{Graph $B_{4}=\Gamma_{6,2}$, which does not occur}
        \label{fig:Gamma62}
    \end{center}
\end{figure}

\FloatBarrier

We note that the graph $B_4$ has two vertices of order two which are adjacent to one another and share no common neighbor. Thus, $B_4$ can also be eliminated by the main result of \cite{bissler2025family} and therefore falls into two previously studied families.

\subsection*{Unclassified Graphs}
Of the forty-five graphs with a largest clique size of six, twenty-one remain un
classified. These are the graphs $B_j$ for $j\in\mathfrak{J}:=\{5,8,9,10,11,14,16,17,\\18,19,20,
22,23,24,25,26,27,28,29,31,33\}$. These graphs can be seen in \autoref{fig:UncB} below, as well as in Appendix \hyperref[Sec:AppB]{B}, where they will be marked with an asterisk.

\begin{figure}[!ht]
    \centering
    \begin{subfigure}[t]{0.15\textwidth}
        \centering
        \includegraphics[height=0.75in]{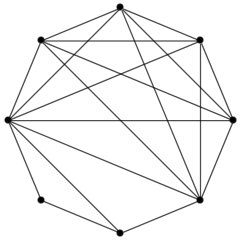}
        \caption*{$B_5$}
    \end{subfigure}
    \begin{subfigure}[t]{0.15\textwidth}
        \centering
        \includegraphics[height=0.75in]{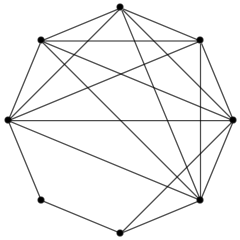}
        \caption*{$B_8$}
    \end{subfigure}
    \begin{subfigure}[t]{0.15\textwidth}
        \centering
        \includegraphics[height=0.75in]{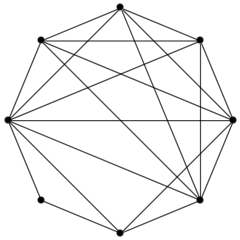}
        \caption*{$B_9$}
    \end{subfigure}
    \begin{subfigure}[t]{0.15\textwidth}
        \centering
        \includegraphics[height=0.75in]{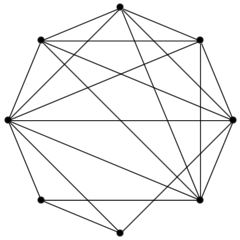}
        \caption*{$B_{10}$}
    \end{subfigure}
    \begin{subfigure}[t]{0.15\textwidth}
        \centering
        \includegraphics[height=0.75in]{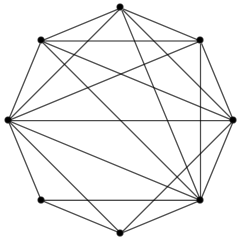}
        \caption*{$B_{11}$}
    \end{subfigure}
\end{figure}
\begin{figure}[!ht]
    \centering
    \begin{subfigure}[t]{0.15\textwidth}
        \centering
        \includegraphics[height=0.75in]{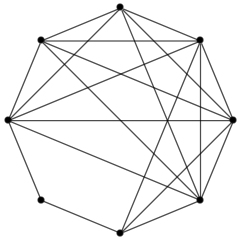}
        \caption*{$B_{14}$}
    \end{subfigure}
    \begin{subfigure}[t]{0.15\textwidth}
        \centering
        \includegraphics[height=0.75in]{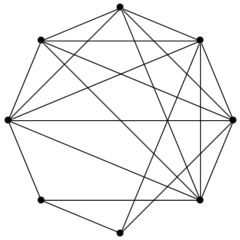}
        \caption*{$B_{16}$}
    \end{subfigure}
    \begin{subfigure}[t]{0.15\textwidth}
        \centering
        \includegraphics[height=0.75in]{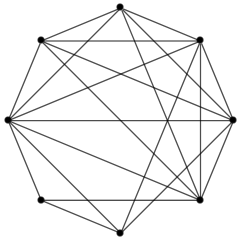}
        \caption*{$B_{17}$}
    \end{subfigure}
    \begin{subfigure}[t]{0.15\textwidth}
        \centering
        \includegraphics[height=0.75in]{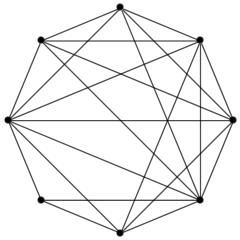}
        \caption*{$B_{18}$}
    \end{subfigure}
    \begin{subfigure}[t]{0.15\textwidth}
        \centering
        \includegraphics[height=0.75in]{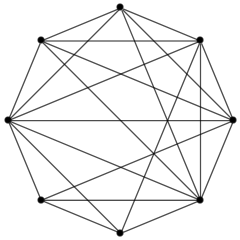}
        \caption*{$B_{19}$}
    \end{subfigure}
\end{figure}
\begin{figure}[!ht]
    \centering
    \begin{subfigure}[t]{0.15\textwidth}
        \centering
        \includegraphics[height=0.75in]{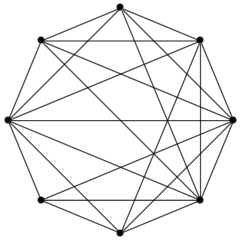}
        \caption*{$B_{20}$}
    \end{subfigure}
    \begin{subfigure}[t]{0.15\textwidth}
        \centering
        \includegraphics[height=0.75in]{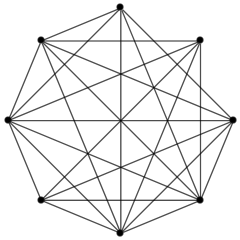}
        \caption*{$B_{22}$}
    \end{subfigure}
    \begin{subfigure}[t]{0.15\textwidth}
        \centering
        \includegraphics[height=0.75in]{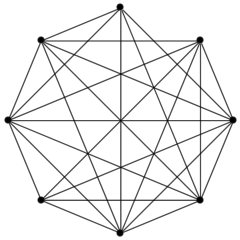}
        \caption*{$B_{23}$}
    \end{subfigure}
    \begin{subfigure}[t]{0.15\textwidth}
        \centering
        \includegraphics[height=0.75in]{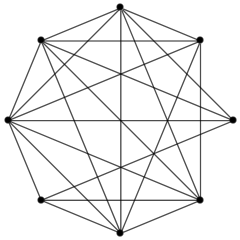}
        \caption*{$B_{24}$}
    \end{subfigure}
    \begin{subfigure}[t]{0.15\textwidth}
        \centering
        \includegraphics[height=0.75in]{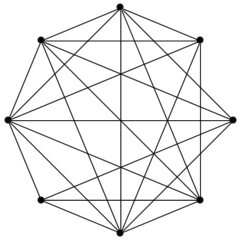}
        \caption*{$B_{25}$}
    \end{subfigure}
\end{figure}
\begin{figure}[!ht]
    \centering
    \begin{subfigure}[t]{0.15\textwidth}
        \centering
        \includegraphics[height=0.75in]{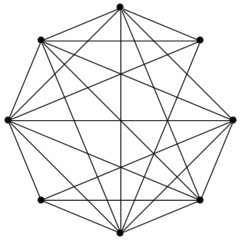}
        \caption*{$B_{26}$}
    \end{subfigure}
    \begin{subfigure}[t]{0.15\textwidth}
        \centering
        \includegraphics[height=0.75in]{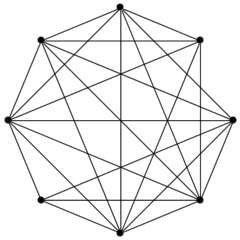}
        \caption*{$B_{27}$}
    \end{subfigure}
    \begin{subfigure}[t]{0.15\textwidth}
        \centering
        \includegraphics[height=0.75in]{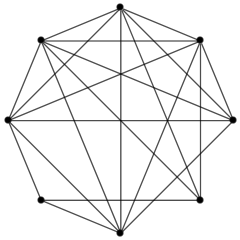}
        \caption*{$B_{28}$}
    \end{subfigure}
    \begin{subfigure}[t]{0.15\textwidth}
        \centering
        \includegraphics[height=0.75in]{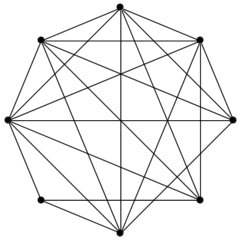}
        \caption*{$B_{29}$}
    \end{subfigure}
    \begin{subfigure}[t]{0.15\textwidth}
        \centering
        \includegraphics[height=0.75in]{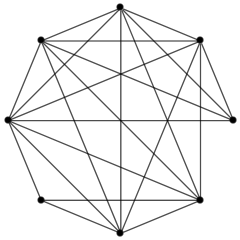}
        \caption*{$B_{31}$}
    \end{subfigure}
\end{figure}
\renewcommand\thefigure{\arabic{figure}}
\setcounter{figure}{12}
\begin{figure}[!ht]
\begin{center}
\includegraphics[width=0.15\textwidth]{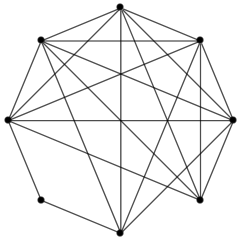}
\caption*{$B_{33}$}
\end{center}   
\caption{The unclassified graphs $B_j$ for $j\in\mathfrak{J}$}
\label{fig:UncB}
\end{figure}

\FloatBarrier

As was the case in \cite{laubacher2023classifying}, one major source of difficulty in dealing with the unclassified graphs above is the fact that they contain the occurring disconnected graph with component sizes six and two as a subgraph. Additionally, several graphs have connected subgraphs of order seven which either occur or remain unclassified based off of the work in \cite{laubacher2023classifying}.

\subsection{Connected Graphs With a Largest Clique Size of Five}\label{Sec:Cl5and3}
There are one hundred fifty-one graphs of order eight which have a maximum
clique size of five. This is the most common maximum clique size and thus Appendix \hyperref[Sec:AppC]{C} is the largest of the four appendices. In this section we classify thirty-seven graphs. Specifically, we show that the graphs in \autoref{fig:CGraphsDirProd} occur, while the graphs in figures \ref{fig:CNonOcc1}, \ref{fig:CNonOcc2}, and \ref{fig:CNonOcc3} do not. This leaves the one hundred fourteen graphs mentioned at the end of the section as possibly occurring.

\subsection*{Graph Joins}
Ten graphs with a largest clique of size five can be constructed via joins of occurring graphs. Namely, the graphs $C_i$ for $i\in\mathfrak{I}$ where $\mathfrak{I}:=\{35,61,63,112,113,147,148,149,150,151\}$. These graphs can be seen in \autoref{fig:CGraphsDirProd} or in Appendix \hyperref[Sec:AppC]{C}.

\begin{figure}[!ht]
    \centering
    \begin{subfigure}[t]{0.2\textwidth}
        \centering
        \includegraphics[height=1in]{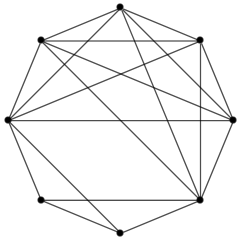}
        \caption*{$C_{35}$}
    \end{subfigure}
    \begin{subfigure}[t]{0.2\textwidth}
        \centering
        \includegraphics[height=1in]{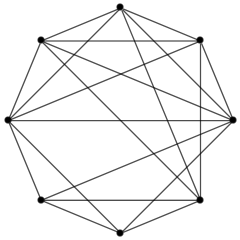}
        \caption*{$C_{61}$}
    \end{subfigure}
    \begin{subfigure}[t]{0.2\textwidth}
        \centering
        \includegraphics[height=1in]{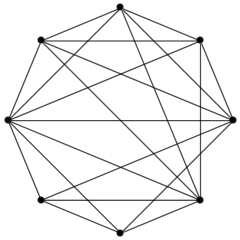}
        \caption*{$C_{63}$}
    \end{subfigure}
    \begin{subfigure}[t]{0.2\textwidth}
        \centering
        \includegraphics[height=1in]{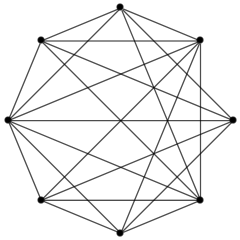}
        \caption*{$C_{112}$}
    \end{subfigure}
\end{figure}
\begin{figure}[!ht]
    \centering
    \begin{subfigure}[t]{0.2\textwidth}
        \centering
        \includegraphics[height=1in]{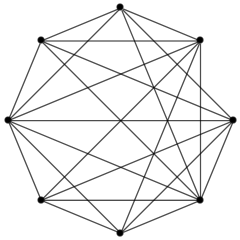}
        \caption*{$C_{113}$}
    \end{subfigure}
    \begin{subfigure}[t]{0.2\textwidth}
        \centering
        \includegraphics[height=1in]{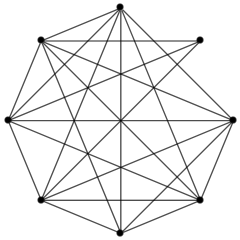}
        \caption*{$C_{147}$}
    \end{subfigure}
    \begin{subfigure}[t]{0.2\textwidth}
        \centering
        \includegraphics[height=1in]{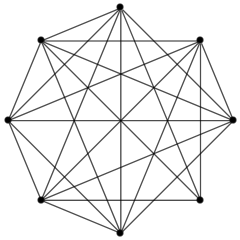}
        \caption*{$C_{148}$}
    \end{subfigure}
    \begin{subfigure}[t]{0.2\textwidth}
        \centering
        \includegraphics[height=1in]{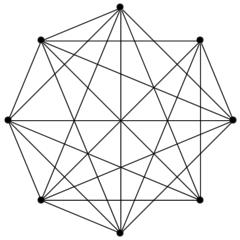}
        \caption*{$C_{149}$}
    \end{subfigure}
\end{figure}
\renewcommand\thefigure{\arabic{figure}}
\setcounter{figure}{13}
\begin{figure}[!ht]
    \centering
    \begin{subfigure}[t]{0.2\textwidth}
        \centering
        \includegraphics[height=1in]{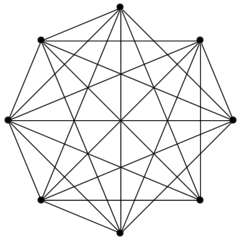}
        \caption*{$C_{150}$}
    \end{subfigure}
    \begin{subfigure}[t]{0.2\textwidth}
        \centering
        \includegraphics[height=1in]{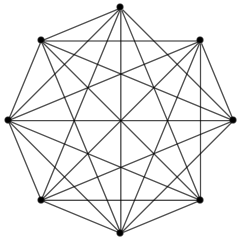}
        \caption*{$C_{151}$}
    \end{subfigure}
    \caption{The graphs $C_i$ for $i\in\mathfrak{I}$}
    \label{fig:CGraphsDirProd}
\end{figure}

\FloatBarrier

Note that the graph $C_{35}$ is the join of the disconnected graph with component sizes four and two with the empty graph of order two. As was the case in subsection \ref{Sec:6and2}, this graph can be achieved by taking joins of disconnected graphs constructed using the techniques from Section \ref{Sec:Disconn}. For example, one could choose the values $m=5\times7,2^m-1=31\times71\times127\times122,921,p=2,$ and $2^p-1=3$ to achieve this graph. Meanwhile, graph $C_{61}$ is the join of the disconnected graph with component sizes three and two with the disconnected graph having component sizes two and one. Here, one may choose $m=2\times7,2^m-1=3\times43\times127,p=11,$ and $2^p-1=23\times89$ to show that $C_{61}$ occurs. Graph $C_{112}$ is the join of the bowtie graph (shown to occur by Lewis in \cite{lewis2004classifying} and generalized in \cite{laubacher2026bowtie}) with the disconnected graph having component sizes two and one. In the original construction of the bowtie graph, the primes $2,3,5,11,$ and $31$ were used. Constructing the disconnected graph using the values $p=23$ and $2^p-1=47\times178,481$, we see that graph $C_{112}$ occurs. Additionally, graph $C_{147}$ is the join of the unique occurring diameter three graph with six vertices and the empty graph of order two. The original construction of the diameter three graph utilized the primes $2,3,5,7,31,$ and $151$. Constructing the disconnected graph using $p=13$ and $2^p-1=8191$ leads to the graph $C_{147}$. Both $C_{148}$ and $C_{149}$ can be constructed by taking the join of an occurring graph from the bottom row in Figure 4 of \cite{bissler2019classifying} with the empty graph of order two. In both cases, the six vertex graph can be constructed using joins of other smaller disconnected graphs, none of which use the primes $13$ or $8191$. Choosing $p=13$ yields $2^p-1=8191$ and again leads to an empty graph with a distinct vertex set, and thus the occurrence of $C_{148}$ and $C_{149}$. Finally, the graphs $C_{63},C_{113},C_{150},$ and $C_{151}$ can all be constructed by taking the join of a single vertex with one of the graphs from Figure 12 of \cite{laubacher2023classifying}.

\subsection*{Diameter Three Graphs}
There are twenty-three graphs with a maximum clique size of five and diameter
three. These graphs are $C_j$ where $j\in\{1,2,4,5,6,9,13,16,18,\\19,25,29,31,32,38,39,43,49,51,53,70,86,88\}$ and are visible in figures \ref{fig:CNonOcc1} and \ref{fig:CNonOcc2} as well as Appendix \hyperref[Sec:AppC]{C}. Again, using the $\rho$ notation introduced by Lewis in \cite{lewis2002solvable5} and the subsequent results from \cite{sass2016character}, we are able to completely classify the diameter three graphs with a largest clique of size five. In fact, we show in this section that none of them occur as $\Delta(G)$ for any finite solvable group $G$.

First, we consider the graphs $C_j$ where $j\in\mathfrak{J}_1:=\{1,2,4,5,6,9\}$. For each of these graphs, regardless of which vertices are chosen for $p$ and $q$, it follows that $\rho_3$, the set of vertices which are distance two from $p$, satisfies $\lvert\rho_3\rvert<3$, contradicting Theorem 2 of \cite{sass2016character}. It follows that none of the graphs in \autoref{fig:CNonOcc1} occur as $\Delta(G)$ for any finite solvable $G$.

\begin{figure}[!ht]
    \centering
    \begin{subfigure}[t]{0.25\textwidth}
        \centering
        \includegraphics[height=1in]{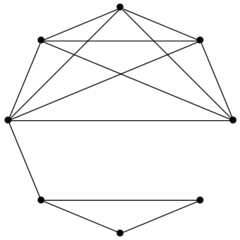}
        \caption*{$C_{1}$}
    \end{subfigure}
    \begin{subfigure}[t]{0.25\textwidth}
        \centering
        \includegraphics[height=1in]{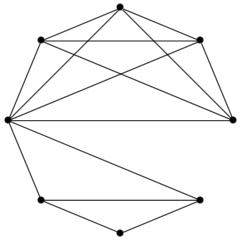}
        \caption*{$C_{2}$}
    \end{subfigure}
    \begin{subfigure}[t]{0.25\textwidth}
        \centering
        \includegraphics[height=1in]{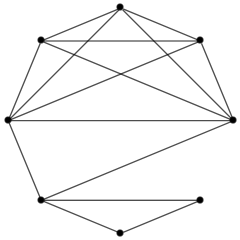}
        \caption*{$C_{4}$}
    \end{subfigure}
\end{figure}
\renewcommand\thefigure{\arabic{figure}}
\setcounter{figure}{14}
\begin{figure}[!ht]
    \centering
    \begin{subfigure}[t]{0.25\textwidth}
        \centering
        \includegraphics[height=1in]{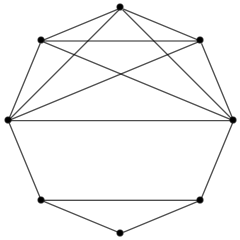}
        \caption*{$C_{5}$}
    \end{subfigure}
    \begin{subfigure}[t]{0.25\textwidth}
        \centering
        \includegraphics[height=1in]{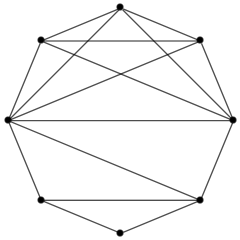}
        \caption*{$C_{6}$}
    \end{subfigure}
    \begin{subfigure}[t]{0.25\textwidth}
        \centering
        \includegraphics[height=1in]{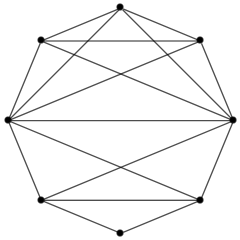}
        \caption*{$C_{9}$}
    \end{subfigure}
    \caption{The graphs $C_j$ for $j\in\mathfrak{J}_1$ which fail to satisfy $\lvert\rho_3\rvert\geq 3$}
    \label{fig:CNonOcc1}
\end{figure}

\FloatBarrier

We note as well that the graph $C_1$ has two cut vertices and therefore may also be eliminated by the main result of \cite{lewiscutvertex}.

Denote the remaining seventeen graphs by $C_j$ for $j\in\mathfrak{J}_2:=\{13,16,18,19,25,29,\\31,32,38,39,43,49,51,53,70,86,88\}$. Noting that each of these graphs satisfies $\lvert\rho_1\cup\rho_2\rvert\geq 3$, it follows from Theorem 4 of \cite{sass2016character} that $\lvert\rho_3\cup\rho_4\rvert\geq 2^3=8$ for each remaining graph. This is clearly impossible, however, as we are considering graphs with eight total vertices and $\rho_3\cup\rho_4$ is a proper subset of each graph's vertex set. Thus, the graphs in \autoref{fig:CNonOcc2} do not occur as $\Delta(G)$ for any finite solvable $G$.

\begin{figure}[!ht]
    \centering
    \begin{subfigure}[t]{0.15\textwidth}
        \centering
        \includegraphics[height=0.75in]{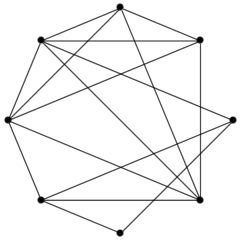}
        \caption*{$C_{13}$}
    \end{subfigure}
    \begin{subfigure}[t]{0.15\textwidth}
        \centering
        \includegraphics[height=0.75in]{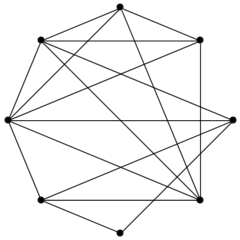}
        \caption*{$C_{16}$}
    \end{subfigure}
    \begin{subfigure}[t]{0.15\textwidth}
        \centering
        \includegraphics[height=0.75in]{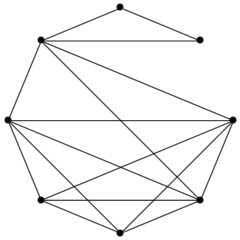}
        \caption*{$C_{18}$}
    \end{subfigure}
    \begin{subfigure}[t]{0.15\textwidth}
        \centering
        \includegraphics[height=0.75in]{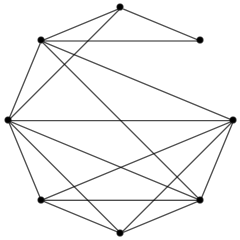}
        \caption*{$C_{19}$}
    \end{subfigure}
    \begin{subfigure}[t]{0.15\textwidth}
        \centering
        \includegraphics[height=0.75in]{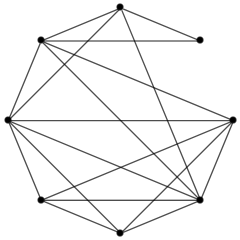}
        \caption*{$C_{25}$}
    \end{subfigure}
\end{figure}
\begin{figure}[!ht]
    \centering
    \begin{subfigure}[t]{0.15\textwidth}
        \centering
        \includegraphics[height=0.75in]{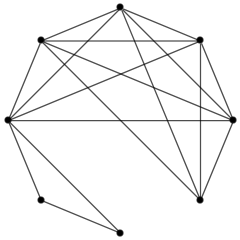}
        \caption*{$C_{29}$}
    \end{subfigure}
    \begin{subfigure}[t]{0.15\textwidth}
        \centering
        \includegraphics[height=0.75in]{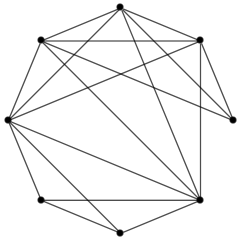}
        \caption*{$C_{31}$}
    \end{subfigure}
    \begin{subfigure}[t]{0.15\textwidth}
        \centering
        \includegraphics[height=0.75in]{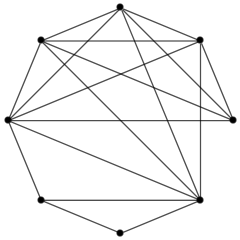}
        \caption*{$C_{32}$}
    \end{subfigure}
    \begin{subfigure}[t]{0.15\textwidth}
        \centering
        \includegraphics[height=0.75in]{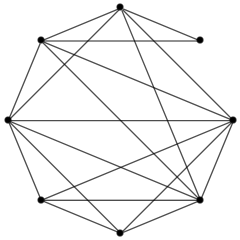}
        \caption*{$C_{38}$}
    \end{subfigure}
    \begin{subfigure}[t]{0.15\textwidth}
        \centering
        \includegraphics[height=0.75in]{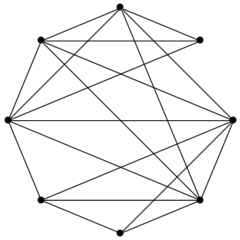}
        \caption*{$C_{39}$}
    \end{subfigure}
\end{figure}
\begin{figure}[!ht]
    \centering
    \begin{subfigure}[t]{0.15\textwidth}
        \centering
        \includegraphics[height=0.75in]{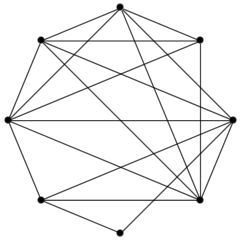}
        \caption*{$C_{43}$}
    \end{subfigure}
    \begin{subfigure}[t]{0.15\textwidth}
        \centering
        \includegraphics[height=0.75in]{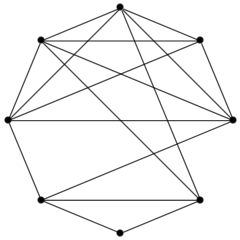}
        \caption*{$C_{49}$}
    \end{subfigure}
    \begin{subfigure}[t]{0.15\textwidth}
        \centering
        \includegraphics[height=0.75in]{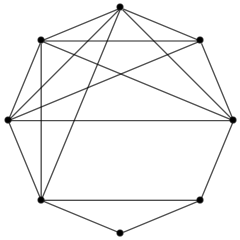}
        \caption*{$C_{51}$}
    \end{subfigure}
    \begin{subfigure}[t]{0.15\textwidth}
        \centering
        \includegraphics[height=0.75in]{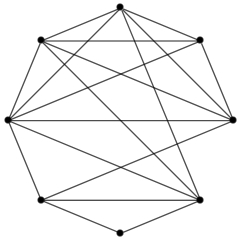}
        \caption*{$C_{53}$}
    \end{subfigure}
    \begin{subfigure}[t]{0.15\textwidth}
        \centering
        \includegraphics[height=0.75in]{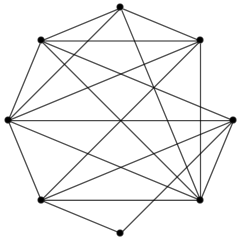}
        \caption*{$C_{70}$}
    \end{subfigure}
\end{figure}
\renewcommand\thefigure{\arabic{figure}}
\setcounter{figure}{15}
\begin{figure}[!ht]
    \centering
    \begin{subfigure}[t]{0.15\textwidth}
        \centering
        \includegraphics[height=0.75in]{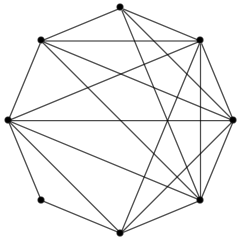}
        \caption*{$C_{86}$}
    \end{subfigure}
    \begin{subfigure}[t]{0.15\textwidth}
        \centering
        \includegraphics[height=0.75in]{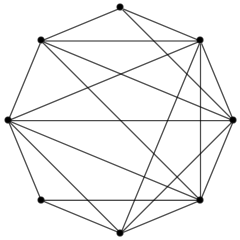}
        \caption*{$C_{88}$}
    \end{subfigure}
    \caption{The graphs $C_j$ for $j\in\mathfrak{J}_2$}
    \label{fig:CNonOcc2}
\end{figure}

\FloatBarrier

\subsection*{Previously Classified Graphs}
There are four graphs whose largest clique contains five vertices that have been previously classified. Specifically, these are the graphs $C_{7}$, $C_{14}$, $C_{50}$ and $C_{95}$ which can be seen in Appendix \hyperref[Sec:AppC]{C} and \autoref{fig:CNonOcc3}.

\begin{figure}
    \centering
    \begin{subfigure}[t]{0.25\textwidth}
        \centering
        \includegraphics[height=1in]{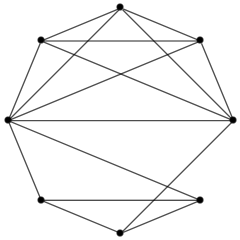}
        \caption*{$C_{7}$}
    \end{subfigure}
    \begin{subfigure}[t]{0.25\textwidth}
        \centering
        \includegraphics[height=1in]{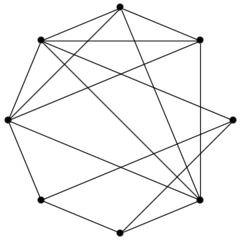}
        \caption*{$C_{14}$}
    \end{subfigure}
    \begin{subfigure}[t]{0.25\textwidth}
        \centering
        \includegraphics[height=1in]{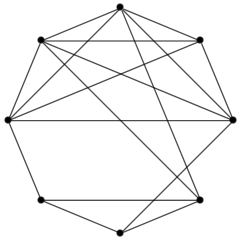}
        \caption*{$C_{50}$}
    \end{subfigure}
    \begin{subfigure}[t]{0.25\textwidth}
        \centering
        \includegraphics[height=1in]{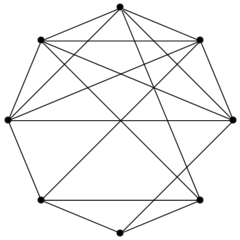}
        \caption*{$C_{95}$}
    \end{subfigure}
    \caption{The graphs $C_{7}$, $C_{14}$, $C_{50}$ and $C_{95}$ which were previously classified}
    \label{fig:CNonOcc3}
\end{figure}

Notice that graph $C_{14}$ is isomorphic to the graph $\Gamma_{5,3}$, which was shown not to occur in \cite{bissler2019classifyingfamilies}. The graph $C_{50}$ is isomorphic to the graph $\Sigma_{3,1}^R$ of \cite{degroot2022prime}, which does not occur by the main result of that paper. Additionally, the graph $C_7$ is realized as $\Sigma^{3L}_{2,1}$ in \cite{bissler2026prime}, where it was shown to be part of the non-occurring family $\Sigma^{mL}_{2,1}$ with $m\geq 2$. Finally, the graph $C_{95}$ was also classified in \cite{bissler2026prime}, where it was reconciled as the non-occurring graph $\Sigma^0_{3,2}$.

\subsection*{Unclassified Graphs}\label{Sec:Unc5and3}
One hundred fourteen graphs with a largest clique of size five currently require additional study, and likely new methods to be classified. Let $\mathfrak{K}:=\{3,8,10,11,12,15,17,20,21,22,23,24,26,27,28,30,33,34,36,37,40,41,42,44,45,46,47,48,\\ 52,54,55,56,57,58,59,60,62,64,65,66,67,68,69,71,72,73,74,75,76,77,78,79,80,81,82,\\ 83,84,85,87\}\cup\{x:89\leq x\leq 94\}\cup\{x:96\leq x\leq 111\}\cup\{x:114\leq x\leq 146\}$. Then the unclassified graphs with maximum clique size five are the graphs $C_{k}$ for $k\in\mathfrak{K}$.
In the interest of space, the full list will not be portrayed here, however all graphs are visible in Appendix \hyperref[Sec:AppC]{C} where they will be marked with an asterisk to emphasize that they are unclassified.

\subsection{Connected Graphs With a Largest Clique Size of Four}\label{Sec:Cl4and4}
The final possibility for the largest clique size is four. In this case, we have that the vertex set must be covered by two cliques of size four. This comprises the second largest family of graphs, with ninety-six graphs covered by two cliques of equal sizes, all of which can be seen in Appendix \hyperref[Sec:AppD]{D}. In this section, we classify twenty-seven graphs while sixty-nine remain unclassified. Namely, we show that the graph in \autoref{fig:DOcc} occurs while the graphs in figures \ref{fig:DNonOcc1}, \ref{fig:DNonOcc2}, \ref{fig:DNonOcc3}, and \ref{fig:DNonOcc4} do not. 

\subsection*{Graph Joins}
The only graph covered by two cliques of size four which can be constructed as the join of known occurring graphs is $D_{96}$, pictured in \autoref{fig:DOcc} and Appendix \hyperref[Sec:AppD]{D}.

\begin{figure}[!ht]
    \begin{center}
        \includegraphics[width=0.2\textwidth]{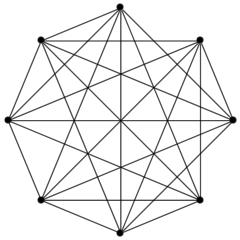}
        \caption{Graph $D_{96}$, which can be constructed via  joins of occurring graphs}
        \label{fig:DOcc}
    \end{center}
\end{figure} 

This graph is the join of the four-regular graph with six vertices (shown to occur in \cite{bissler2019classifying}) and the empty graph of order two. Noting further that the four-regular graph of order six is merely the join of three copies of the empty graph of order two, we can write $D_{96}\cong(K_1\cup K_1)\vee(K_1\cup K_1)\vee(K_1\cup K_1)\vee(K_1\cup K_1)$. Thus, it suffices to find four distinct primes $p_1,p_2,p_3,p_4$ such that $2^{p_i}-1=q_i$ are also distinct primes for each $1\leq i\leq 4$, then proceed with the disconnected graph construction used in Section \ref{Sec:Disconn}. As one example, we can choose the $p_i$ and corresponding $q_i$ to be $(p_i,q_i)\in\{(2,3),(5,31),(7,127),(13,8191)\}$ to achieve the graph $D_{96}$.

\subsection*{Diameter Three Graphs}
When considering graphs covered by cliques of sizes four and four, there are twenty-one which have diameter three. Namely the graphs $D_{j}$ where $j\in\{1,2,3,4,5,6,8,9,11,13,14,16,18,20,22,28,36,37,40,53,55\}$, which can be found in figures \ref{fig:DNonOcc1} and \ref{fig:DNonOcc2} as well as Appendix \hyperref[Sec:AppD]{D}. As was the case when studying graphs covered by cliques of sizes five and three, the $\rho$ notation and results of Lewis and Sass in \cite{lewis2002solvable5} and \cite{sass2016character} respectively, allow us to completely classify the twenty-one graphs with diameter three mentioned above. 

To begin, consider the graphs $D_j$ where $j\in\mathfrak{J}_1:=\{1,2,4,6,20\}$. As seen previously, for each of these graphs, regardless of which vertices are chosen for $p$ and $q$, it again follows that the set $\rho_3$ satisfies $\lvert\rho_3\rvert<3$, contradicting Theorem 2 of \cite{sass2016character}. It follows that none of the graphs in \autoref{fig:DNonOcc1} occur as $\Delta(G)$ for any finite solvable $G$.

\begin{figure}[!ht]
    \centering
    \begin{subfigure}[t]{0.2\textwidth}
        \centering
        \includegraphics[height=1in]{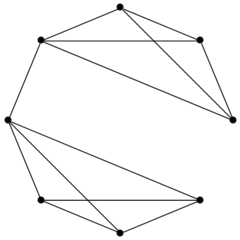}
        \caption*{$D_{1}$}
    \end{subfigure}
    \begin{subfigure}[t]{0.2\textwidth}
        \centering
        \includegraphics[height=1in]{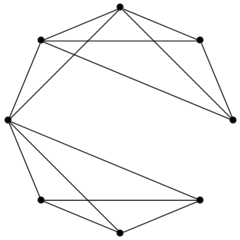}
        \caption*{$D_{2}$}
    \end{subfigure}
    \begin{subfigure}[t]{0.2\textwidth}
        \centering
        \includegraphics[height=1in]{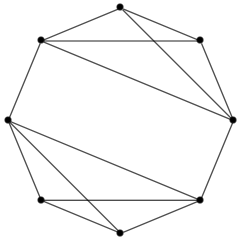}
        \caption*{$D_{4}$}
    \end{subfigure}
\end{figure}
\renewcommand\thefigure{\arabic{figure}}
\setcounter{figure}{18}
\begin{figure}[!ht]
    \centering
    \begin{subfigure}[t]{0.2\textwidth}
        \centering
        \includegraphics[height=1in]{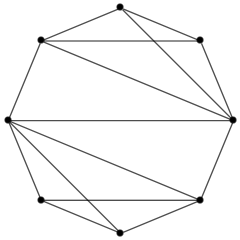}
        \caption*{$D_{6}$}
    \end{subfigure}
    \begin{subfigure}[t]{0.2\textwidth}
        \centering
        \includegraphics[height=1in]{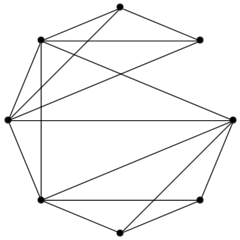}
        \caption*{$D_{20}$}
    \end{subfigure}
    \caption{The graphs $D_j$ for $j\in\mathfrak{J}_1$ which do not occur}
    \label{fig:DNonOcc1}
\end{figure}

\FloatBarrier

Again, we point out that graph $D_1$ can also be eliminated via the main result of \cite{lewiscutvertex}, as it has two cut vertices.

Next, we examine the graphs $D_j$ where $j\in\mathfrak{J}_2:=\{3,5,8,9,11,13,14,16,18,22,\\28,36,37,40,53,55\}$. For each of these graphs, we have that $\lvert\rho_1\cup\rho_2\rvert\geq 3$ which implies $\lvert\rho_3\cup\rho_4\rvert\geq 8$. Thus, by the same argument used in the previous section, we have that each graph in \autoref{fig:DNonOcc2} violates Theorem 4 of \cite{sass2016character} and therefore does not occur.

\begin{figure}[!ht]
    \centering
    \begin{subfigure}[t]{0.2\textwidth}
        \centering
        \includegraphics[height=1in]{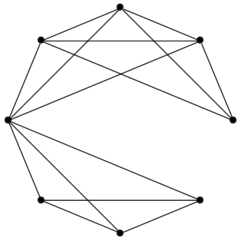}
        \caption*{$D_{3}$}
    \end{subfigure}
    \begin{subfigure}[t]{0.2\textwidth}
        \centering
        \includegraphics[height=1in]{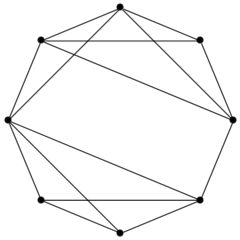}
        \caption*{$D_{5}$}
    \end{subfigure}
    \begin{subfigure}[t]{0.2\textwidth}
        \centering
        \includegraphics[height=1in]{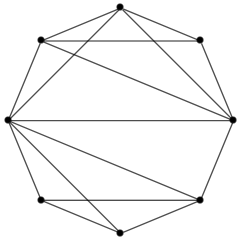}
        \caption*{$D_{8}$}
    \end{subfigure}
    \begin{subfigure}[t]{0.2\textwidth}
        \centering
        \includegraphics[height=1in]{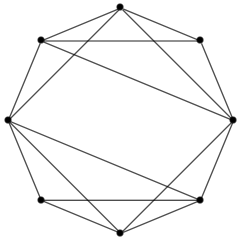}
        \caption*{$D_{9}$}
    \end{subfigure}
\end{figure}
\begin{figure}[!ht]
    \centering
    \begin{subfigure}[t]{0.2\textwidth}
        \centering
        \includegraphics[height=1in]{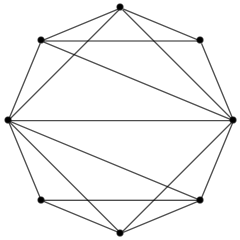}
        \caption*{$D_{11}$}
    \end{subfigure}
    \begin{subfigure}[t]{0.2\textwidth}
        \centering
        \includegraphics[height=1in]{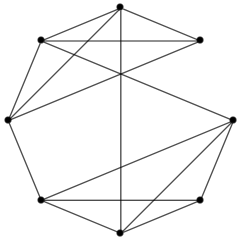}
        \caption*{$D_{13}$}
    \end{subfigure}
    \begin{subfigure}[t]{0.2\textwidth}
        \centering
        \includegraphics[height=1in]{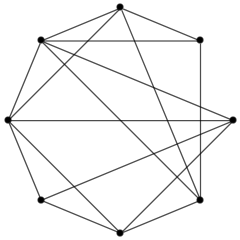}
        \caption*{$D_{14}$}
    \end{subfigure}
    \begin{subfigure}[t]{0.2\textwidth}
        \centering
        \includegraphics[height=1in]{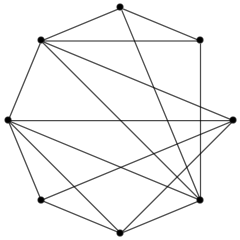}
        \caption*{$D_{16}$}
    \end{subfigure}
\end{figure}
\begin{figure}[!ht]
    \centering
    \begin{subfigure}[t]{0.2\textwidth}
        \centering
        \includegraphics[height=1in]{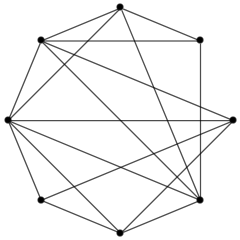}
        \caption*{$D_{18}$}
    \end{subfigure}
    \begin{subfigure}[t]{0.2\textwidth}
        \centering
        \includegraphics[height=1in]{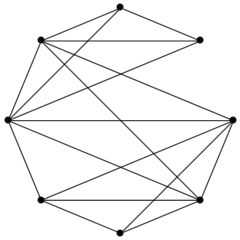}
        \caption*{$D_{22}$}
    \end{subfigure}
    \begin{subfigure}[t]{0.2\textwidth}
        \centering
        \includegraphics[height=1in]{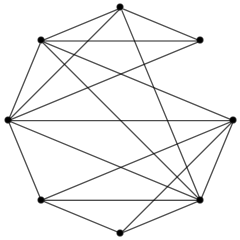}
        \caption*{$D_{28}$}
    \end{subfigure}
    \begin{subfigure}[t]{0.2\textwidth}
        \centering
        \includegraphics[height=1in]{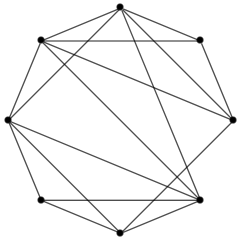}
        \caption*{$D_{36}$}
    \end{subfigure}
\end{figure}
\renewcommand\thefigure{\arabic{figure}}
\setcounter{figure}{19}
\begin{figure}[!ht]
    \centering
    \begin{subfigure}[t]{0.2\textwidth}
        \centering
        \includegraphics[height=1in]{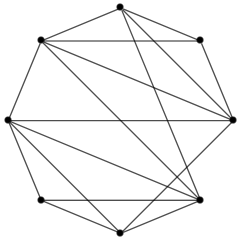}
        \caption*{$D_{37}$}
    \end{subfigure}
    \begin{subfigure}[t]{0.2\textwidth}
        \centering
        \includegraphics[height=1in]{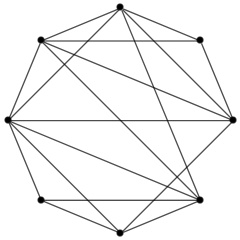}
        \caption*{$D_{40}$}
    \end{subfigure}
    \begin{subfigure}[t]{0.2\textwidth}
        \centering
        \includegraphics[height=1in]{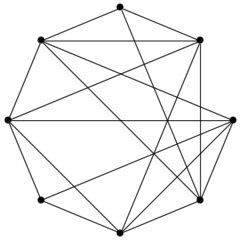}
        \caption*{$D_{53}$}
    \end{subfigure}
    \begin{subfigure}[t]{0.2\textwidth}
        \centering
        \includegraphics[height=1in]{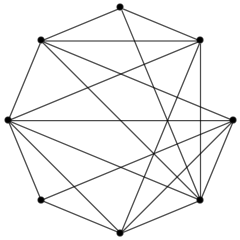}
        \caption*{$D_{55}$}
    \end{subfigure}
    \caption{The graphs $D_j$ for $j\in\mathfrak{J}_2$ which do not occur}
    \label{fig:DNonOcc2}
\end{figure}

\FloatBarrier

\subsection*{Previously Classified Graphs}
There are four graphs covered by cliques of sizes four and four which have been previously studied. Namely, the graphs $D_{15},D_{43},D_{77},$ and $D_{86}$ of \autoref{fig:DNonOcc3}. All four graphs also appear in Appendix \hyperref[Sec:AppD]{D}.

\begin{figure}[!ht]
    \centering
    \begin{subfigure}[t]{0.25\textwidth}
        \centering
        \includegraphics[height=1in]{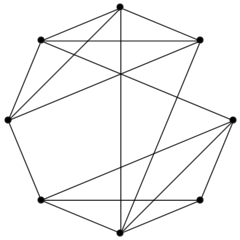}
        \caption*{$D_{15}$}
    \end{subfigure}
    \begin{subfigure}[t]{0.25\textwidth}
        \centering
        \includegraphics[height=1in]{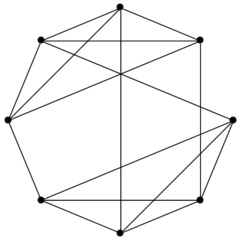}
        \caption*{$D_{43}$}
    \end{subfigure}
    \begin{subfigure}[t]{0.25\textwidth}
        \centering
        \includegraphics[height=1in]{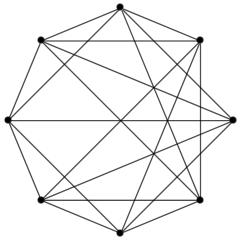}
        \caption*{$D_{77}$}
    \end{subfigure}
    \begin{subfigure}[t]{0.25\textwidth}
        \centering
        \includegraphics[height=1in]{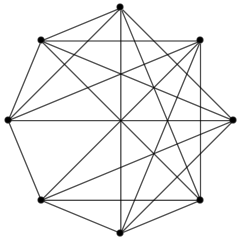}
        \caption*{$D_{86}$}
    \end{subfigure}
    \caption{The previously studied graphs $D_{15},D_{43},D_{77}$ and $D_{86}$}
    \label{fig:DNonOcc3}
\end{figure}

\FloatBarrier

Notice that $D_{15}$ is isomorphic to the graph $\Sigma_{3,1}^L$, which was shown not to occur in \cite{laubacher2021prime}. Also, the graph $D_{43}$ is simply the graph $\Gamma_{4,4}$ which was eliminated by the main result of \cite{bissler2019classifyingfamilies}. Meanwhile, the graphs $D_{77}$ and $D_{86}$ are both five-regular, which cannot occur by Theorem A of \cite{MORRESIZUCCARI2014215}. We note that this makes the complete graph $K_8$ and the graph $D_{96}$ (which is six-regular) the only regular graphs with eight vertices that occur.

\subsection*{An Additional Elimination}
Finally, we eliminate an additional graph which is covered by cliques of sizes four and four. Recall \hyperref[lem:Admiss]{Lemma 1.7} (Lemma 2.3 of \cite{bissler2025family}) which states that if every vertex of a graph $\Gamma$ is admissible, then $\Gamma$ cannot be $\Delta(G)$ for any finite solvable group $G$. We will show that every vertex of the graph $D_{19}$ is admissible, and therefore it does not occur.

\begin{figure}[!ht]
    \begin{center}
    \includegraphics[width=0.4\textwidth]{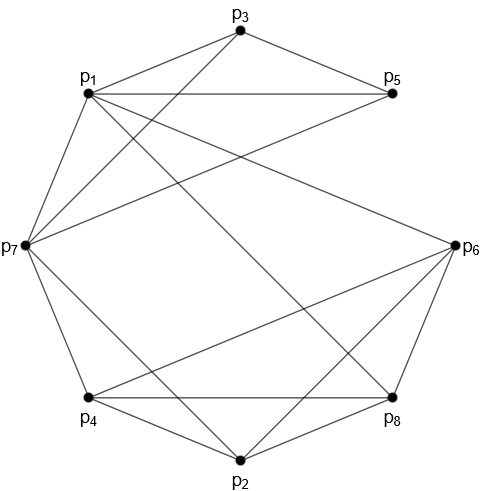}
    \caption{The non-occurring graph $D_{19}$ with all admissible vertices}
    \label{fig:DNonOcc4}
    \end{center}
\end{figure}

\FloatBarrier

\begin{lem}\label{lem:D19Admiss}
    The graph $D_{19}$ of \autoref{fig:DNonOcc4} has all admissible vertices and therefore does not occur.
\end{lem}

\begin{proof}
    For convenience, we label the vertices as in \autoref{fig:DNonOcc4}, with vertices in the same clique having indices of the same parity. For any choice of $p_i$ and $p_j$ in the same clique (except possibly $p_1p_7$), there exists $p_k$ in the opposite clique such that $p_ip_k\notin\textup{E}(\Gamma)$ and $p_jp_k\notin\textup{E}(\Gamma) $. Thus if the edge $p_ip_j$ is lost, the resulting graph will violate P\'alfy's condition. If $p_1p_7$ is lost, the complement of the resulting graph contains the cycle $p_1,p_7,p_6,p_5,p_4,p_1$ of length 5, which violates the main result from \cite{akhlaghi2018character}. Thus, no edge can be lost between two vertices in the same clique.

    Now suppose $p_i$ is a vertex with even index and $p_j$ is a vertex with odd index such that $p_ip_j\in\textup{E}(\Gamma)$. If $p_ip_j$ is lost, then the resulting graph satisfies $d(p_i,p_5)=3$ and has no vertex of order 1. Since all occurring diameter three graphs with eight vertices were shown to have a vertex of order 1, the resulting graph must be a non-occurring graph of diameter three. Thus, no individual edge of $\Gamma$ may be lost. Additionally, a combination of the arguments from the first two paragraphs shows that no combination of edges can be lost.

    Finally, suppose that $p_1$ and all incident edges are removed from the graph. The result is graph $C_2$ of \cite{laubacher2023classifying} which was shown not to occur. By symmetry, $p_7$ also cannot be removed. If $p_3$ (or by symmetry, $p_5$) is removed, the result is the graph $\Sigma_{2,2}^L / C_{10}$ of \cite{laubacher2021prime} and \cite{laubacher2023classifying}, respectively, which also does not occur. Finally, if $p_i$ is removed where $i$ is any even index, then the result is the non-occurring graph $C_{18}$ of \cite{laubacher2023classifying}. Thus, all vertices of $\Gamma$ are admissible, and so by \hyperref[lem:Admiss]{Lemma 1.7}, $\Gamma$ does not occur.
\end{proof}

We note that the graph above falls into a family which has since been studied. In fact, it can be realized as $\Sigma^{2L}_{2,2}$ of \cite{bissler2026prime}, where \hyperref[lem:D19Admiss]{Lemma 4.1} serves as the base case for an induction argument. These graphs with multiple special vertices were briefly mentioned in \cite{laubacher2021prime} and utilized in \cite{degroot2022prime} to investigate the $\Sigma_{k,n}^R$ family of graphs, before being studied generally in \cite{bissler2026prime}.

\subsection*{Unclassified Graphs}\label{Sec:Unc4and4}
Of the ninety-six graphs in Appendix \hyperref[Sec:AppD]{D}, sixty-nine evade classification. Defining $\mathfrak{K}:=\{7,10,12,17,21,23,24,25,26,27,29,30,31,32,33,34,35,38,39,41,42,44,45,\\46,47,48,49,50,51,52,54\}\cup\{x:56\leq x\leq 76\}\cup\{x:78\leq x\leq 85\}\cup\{x:87\leq x\leq 95\}$, we have that the possibly occurring graphs covered by cliques of four and four are $D_k$ for $k\in\mathfrak{K}$. Again, for the sake of space, we choose to display the unclassified graphs only in Appendix \hyperref[Sec:AppD]{D}, where they will be marked with an asterisk.

\section{Conclusions}\label{Sec:Conc}
We begin this section by summarizing our own results, along with the results of all similar classifications performed for graphs with fewer vertices. As in the introduction, we note that the same table also appears in \cite{laubacher2026bowtie}.

\begin{figure}[htb]
    \centering
    \resizebox{\textwidth}{!}{
    \begin{tabular}{c|c|c|c|c}
        Number of & Total Number & Number of Graphs & Number of Graphs & Number of Graphs\\
        Vertices & of Graphs & That Occur & That Do Not Occur & Unclassified\\
        \hline
        1 & 1 & 1 & 0 & 0 \\
        2 & 2 & 2 & 0 & 0 \\
        3 & 4 & 3 & 1 & 0 \\
        4 & 11 & 5 & 6 & 0 \\
        5 & 34 & 9 & 24 & 1 \\
        6 & 156 & 15 & 132 & 9 \\
        7 & 1044 & 24 & 976 & 44 \\
        8 & 12346 & 39 & 12103 & 204\\
    \end{tabular}
    }
    \caption{Number of occurring, non-occurring, and unclassified graphs}
    \label{fig:Table}
\end{figure}

We note the rightmost column of \autoref{fig:Table} and see that, starting with graphs of order five, the number of unclassified graphs begins to grow rapidly. This raises several questions; primarily, what features are shared by the graphs which remain unclassified? One may also ask why current techniques and constructions fail to classify these graphs, or, additionally, which graphs are currently well-controlled with current techniques? In the remainder of this section, we comment on questions such as these.

When examining the classifications of graphs with eight or fewer vertices, one notices that, in every case, the disconnected graphs have been completely classified. The combination of P\'alfy's condition, P\'alfy's inequality, and the $c(n)$ function quickly reduces the number of possibly occurring disconnected graphs. Once the possible component size pairs have been identified, the semi-direct product construction from \cite{lewis2004classifying} has consistently yielded examples of solvable groups which achieve the desired disconnected graphs. As discussed in \cite{lewis2025number}, P\'alfy's inequality is an important necessary condition for a disconnected character degree graph to occur, but the question of its sufficiency remains open and appears to be more number-theoretic in nature. Namely, can one always find a prime $p$ and a positive integer $n$ such that $n$ and $p^n-1$ are co-prime and each have a prescribed number of prime divisors? If so, then (subject to some divisibility conditions outlined in Example 2.4 of \cite{lewis2001solvable}) a solvable group whose character degree graph is disconnected can always be constructed, provided the component sizes satisfy P\'alfy's inequality. In this case, the $c(n)$ function would immediately yield the number of occurring disconnected character degree graphs for any order $n$. We currently do not know of any component size pairs which satisfy P\'alfy's inequality, but do not occur.

Following Corollary \ref{Cor:AkhB}, we may think of character degree graphs in terms of the clique sizes which cover them. This leads to the observation that, when considering graphs of order $n\geq 2$, the graphs covered by cliques of sizes $n-1$ and $1$ always occur. For a connected graph, Theorem \ref{thm:Gamma} yields that the graph $\Gamma_{n-i,1}$ occurs for all $1\leq i\leq n-2$. Following \cite{isaacs1994character}, the join discussion from Section \ref{Sec:Machinery} guarantees that the complete graph $K_{i-1}$ occurs for all $i\geq 2$. It follows that the join $\Gamma_{n-i,1}\vee K_{i-1}$ also occurs for each $2\leq i\leq n-2$ and, together with the complete graph $K_n$, constitutes all connected graphs of order $n$ covered by cliques of sizes $n-1$ and $1$. Additionally, on page 182 of \cite{lewis2008overview}, a solvable group $G$ is constructed such that $\Delta(G)$ is disconnected with component sizes $1$ and $N$, for any positive integer $N$. Taking $N=n-1$ in this extra-special $p$-group construction yields the final graph covered by cliques of sizes $n-1$ and $1$.

Finally, graphs of diameter three have been substantially tamed, with all diameter three graphs of order eight being classified. This remains true for the diameter three graphs of order six and order seven. This raises the question of which graphs remain. Looking at the 204 unclassified graphs, we see that all of them are connected, have diameter two, and are covered by cliques whose sizes differ by at most four vertices. When considering the unclassified graphs from previous investigations, we see that the same can be said, with all graphs being connected, having diameter two, and covered by cliques whose sizes differ by, in general, at most $n-4$ vertices. It is precisely these graphs which will require a closer investigation and new approaches.

In regards to where current techniques fail, the main roadblock is that of occurring/unclassified subgraphs. While quite useful, the results on admissible vertices rely on showing that subgraphs obtained by removing edges or vertices do not occur. When this fails to happen, little can often be said about whether a graph occurs or not. For example, the unclassified graph $C_{12}$ in Appendix \hyperref[Sec:AppC]{C} contains the unclassified graph $C_{19}$ of \cite{laubacher2023classifying} as a subgraph, so no conclusion can be made as to the admissibility of some vertices. This issue continues to compound as some unclassified graphs with eight vertices contain other occurring/unclassified graphs as subgraphs. For example, the unclassified graph $B_{9}$ has the unclassified graph $B_{5}$ as a subgraph, and both of these have the occurring graph $B_{2}$ as a subgraph. It follows that classification methods which do not depend on the classification of subgraphs would be immensely valuable.

\section*{Acknowledgments}\label{Sec:Ack}
  \noindent The authors would like to thank the referees for their time, thorough reading, attention and insightful comments which resulted in an improved paper.

\bibliographystyle{plain}
\bibliography{8GraphBibliography}

\vspace{1cm}

\renewcommand{\theHsection}{app.\thesection}

\begin{appendix}

\FloatBarrier

\section*{Appendix A: Connected Graphs With a Largest Clique Size of Seven}\label{Sec:AppA}

\FloatBarrier

\begin{figure}[!ht]
    \centering
    \begin{subfigure}[t]{0.15\textwidth}
        \centering
        \includegraphics[height=0.75in]{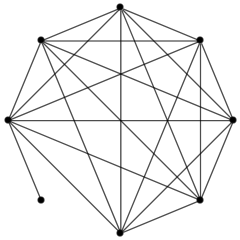}
        \caption*{$A_1$}
    \end{subfigure}
    \begin{subfigure}[t]{0.15\textwidth}
        \centering
        \includegraphics[height=0.75in]{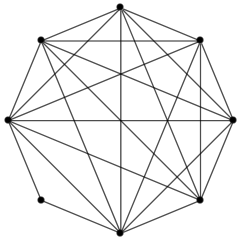}
        \caption*{$A_2$}
    \end{subfigure}
    \begin{subfigure}[t]{0.15\textwidth}
        \centering
        \includegraphics[height=0.75in]{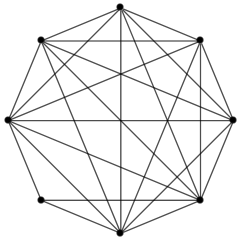}
        \caption*{$A_3$}
    \end{subfigure}
    \begin{subfigure}[t]{0.15\textwidth}
        \centering
        \includegraphics[height=0.75in]{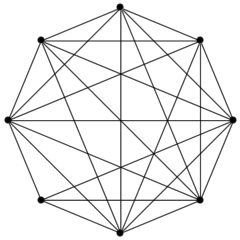}
        \caption*{$A_4$}
    \end{subfigure}
\end{figure}
\begin{figure}[!ht]
\ContinuedFloat
    \centering
    \begin{subfigure}[t]{0.15\textwidth}
        \centering
        \includegraphics[height=0.75in]{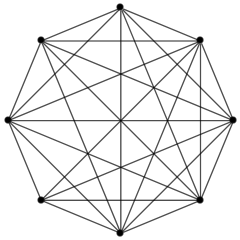}
        \caption*{$A_5$}
    \end{subfigure}
    \begin{subfigure}[t]{0.15\textwidth}
        \centering
        \includegraphics[height=0.75in]{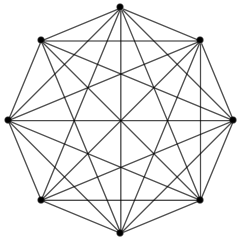}
        \caption*{$A_6$}
    \end{subfigure}
    \begin{subfigure}[t]{0.15\textwidth}
        \centering
        \includegraphics[height=0.75in]{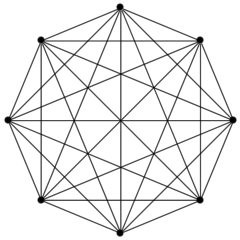}
        \caption*{$A_7$}
    \end{subfigure}
\end{figure}

\FloatBarrier
\clearpage

\section*{Appendix B: Connected Graphs With a Largest Clique Size of Six}\label{Sec:AppB}

\FloatBarrier

\renewcommand{\arraystretch}{2}

\begin{figure}[!h]
    \centering
    \begin{subfigure}[t]{0.15\textwidth}
        \centering
        \includegraphics[height=0.75in]{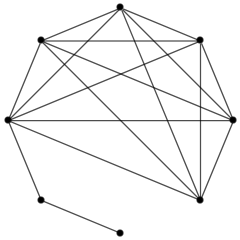}
        \caption*{$B_1$}
    \end{subfigure}
    \begin{subfigure}[t]{0.15\textwidth}
        \centering
        \includegraphics[height=0.75in]{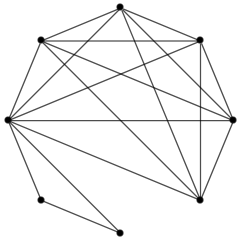}
        \caption*{$B_2$}
    \end{subfigure}
    \begin{subfigure}[t]{0.15\textwidth}
        \centering
        \includegraphics[height=0.75in]{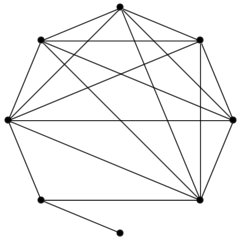}
        \caption*{$B_3$}
    \end{subfigure}
    \begin{subfigure}[t]{0.15\textwidth}
        \centering
        \includegraphics[height=0.75in]{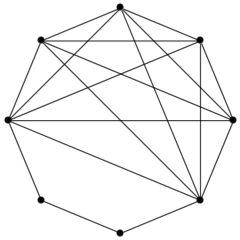}
        \caption*{$B_4$}
    \end{subfigure}
    \begin{subfigure}[t]{0.15\textwidth}
        \centering
        \includegraphics[height=0.75in]{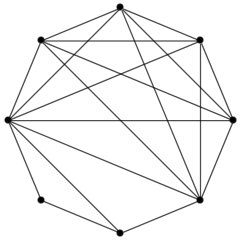}
        \caption*{$B_5*$}
    \end{subfigure}
    \begin{subfigure}[t]{0.15\textwidth}
        \centering
        \includegraphics[height=0.75in]{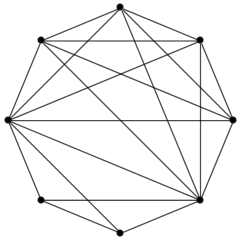}
        \caption*{$B_6$}
    \end{subfigure}
\end{figure}
\begin{figure}[!h]
\ContinuedFloat
    \centering
    \begin{subfigure}[t]{0.15\textwidth}
        \centering
        \includegraphics[height=0.75in]{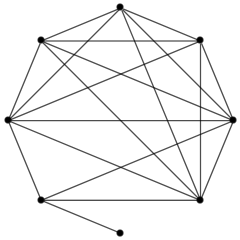}
        \caption*{$B_7$}
    \end{subfigure}
    \begin{subfigure}[t]{0.15\textwidth}
        \centering
        \includegraphics[height=0.75in]{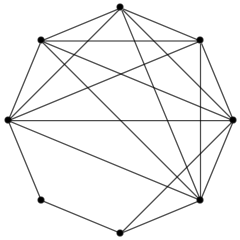}
        \caption*{$B_8*$}
    \end{subfigure}
    \begin{subfigure}[t]{0.15\textwidth}
        \centering
        \includegraphics[height=0.75in]{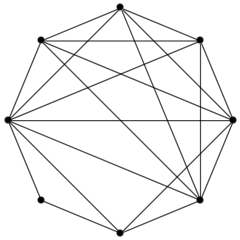}
        \caption*{$B_9*$}
    \end{subfigure}
    \begin{subfigure}[t]{0.15\textwidth}
        \centering
        \includegraphics[height=0.75in]{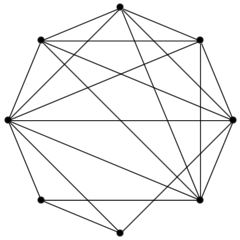}
        \caption*{$B_{10}*$}
    \end{subfigure}
    \begin{subfigure}[t]{0.15\textwidth}
        \centering
        \includegraphics[height=0.75in]{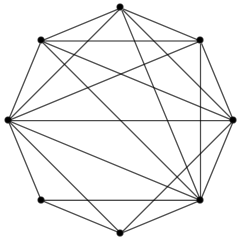}
        \caption*{$B_{11}*$}
    \end{subfigure}
    \begin{subfigure}[t]{0.15\textwidth}
        \centering
        \includegraphics[height=0.75in]{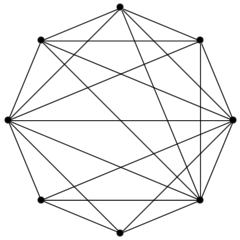}
        \caption*{$B_{12}$}
    \end{subfigure}
\end{figure}
\begin{figure}[!h]
\ContinuedFloat
    \centering
    \begin{subfigure}[t]{0.15\textwidth}
        \centering
        \includegraphics[height=0.75in]{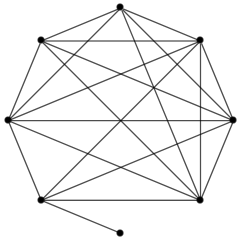}
        \caption*{$B_{13}$}
    \end{subfigure}
    \begin{subfigure}[t]{0.15\textwidth}
        \centering
        \includegraphics[height=0.75in]{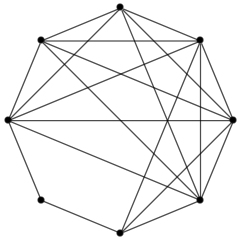}
        \caption*{$B_{14}*$}
    \end{subfigure}
    \begin{subfigure}[t]{0.15\textwidth}
        \centering
        \includegraphics[height=0.75in]{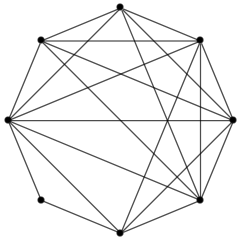}
        \caption*{$B_{15}$}
    \end{subfigure}
    \begin{subfigure}[t]{0.15\textwidth}
        \centering
        \includegraphics[height=0.75in]{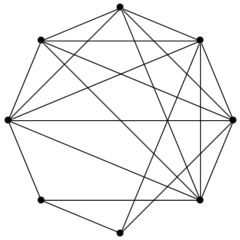}
        \caption*{$B_{16}*$}
    \end{subfigure}
    \begin{subfigure}[t]{0.15\textwidth}
        \centering
        \includegraphics[height=0.75in]{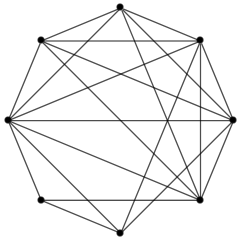}
        \caption*{$B_{17}*$}
    \end{subfigure}
    \begin{subfigure}[t]{0.15\textwidth}
        \centering
        \includegraphics[height=0.75in]{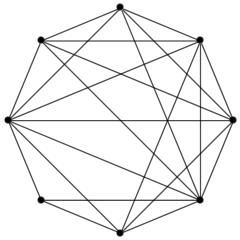}
        \caption*{$B_{18}*$}
    \end{subfigure}
\end{figure}
\begin{figure}[!h]
\ContinuedFloat
    \centering
    \begin{subfigure}[t]{0.15\textwidth}
        \centering
        \includegraphics[height=0.75in]{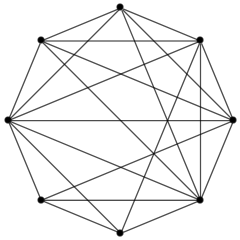}
        \caption*{$B_{19}*$}
    \end{subfigure}
    \begin{subfigure}[t]{0.15\textwidth}
        \centering
        \includegraphics[height=0.75in]{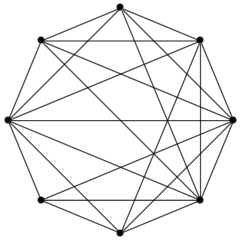}
        \caption*{$B_{20}*$}
    \end{subfigure}
    \begin{subfigure}[t]{0.15\textwidth}
        \centering
        \includegraphics[height=0.75in]{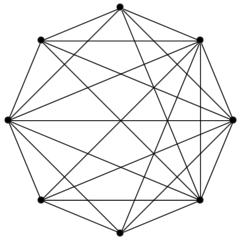}
        \caption*{$B_{21}$}
    \end{subfigure}
    \begin{subfigure}[t]{0.15\textwidth}
        \centering
        \includegraphics[height=0.75in]{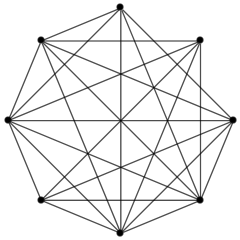}
        \caption*{$B_{22}*$}
    \end{subfigure}
    \begin{subfigure}[t]{0.15\textwidth}
        \centering
        \includegraphics[height=0.75in]{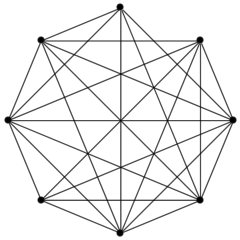}
        \caption*{$B_{23}*$}
    \end{subfigure}
    \begin{subfigure}[t]{0.15\textwidth}
        \centering
        \includegraphics[height=0.75in]{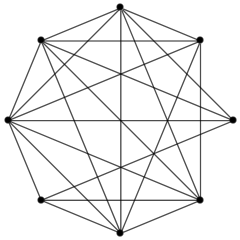}
        \caption*{$B_{24}*$}
    \end{subfigure}
\end{figure}
\begin{figure}[!h]
\ContinuedFloat
    \centering
    \begin{subfigure}[t]{0.15\textwidth}
        \centering
        \includegraphics[height=0.75in]{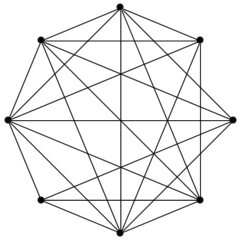}
        \caption*{$B_{25}*$}
    \end{subfigure}
    \begin{subfigure}[t]{0.15\textwidth}
        \centering
        \includegraphics[height=0.75in]{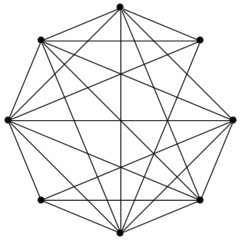}
        \caption*{$B_{26}*$}
    \end{subfigure}
    \begin{subfigure}[t]{0.15\textwidth}
        \centering
        \includegraphics[height=0.75in]{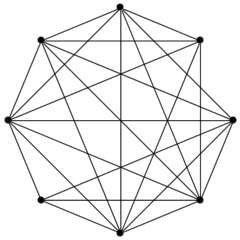}
        \caption*{$B_{27}*$}
    \end{subfigure}
    \begin{subfigure}[t]{0.15\textwidth}
        \centering
        \includegraphics[height=0.75in]{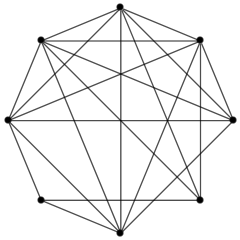}
        \caption*{$B_{28}*$}
    \end{subfigure}
    \begin{subfigure}[t]{0.15\textwidth}
        \centering
        \includegraphics[height=0.75in]{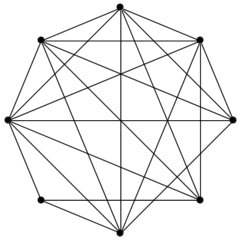}
        \caption*{$B_{29}*$}
    \end{subfigure}
    \begin{subfigure}[t]{0.15\textwidth}
        \centering
        \includegraphics[height=0.75in]{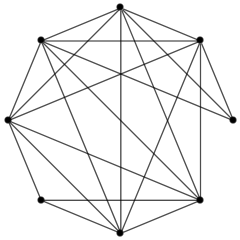}
        \caption*{$B_{30}$}
    \end{subfigure}
\end{figure}
\begin{figure}[!h]
\ContinuedFloat
    \centering
    \begin{subfigure}[t]{0.15\textwidth}
        \centering
        \includegraphics[height=0.75in]{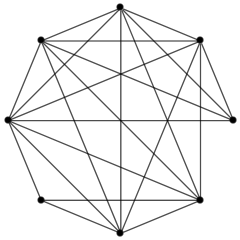}
        \caption*{$B_{31}*$}
    \end{subfigure}
    \begin{subfigure}[t]{0.15\textwidth}
        \centering
        \includegraphics[height=0.75in]{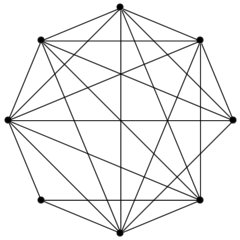}
        \caption*{$B_{32}$}
    \end{subfigure}
    \begin{subfigure}[t]{0.15\textwidth}
        \centering
        \includegraphics[height=0.75in]{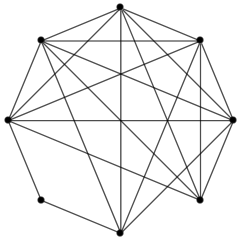}
        \caption*{$B_{33}*$}
    \end{subfigure}
    \begin{subfigure}[t]{0.15\textwidth}
        \centering
        \includegraphics[height=0.75in]{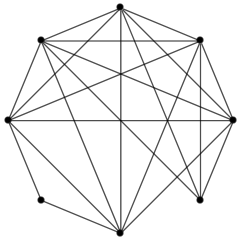}
        \caption*{$B_{34}$}
    \end{subfigure}
    \begin{subfigure}[t]{0.15\textwidth}
        \centering
        \includegraphics[height=0.75in]{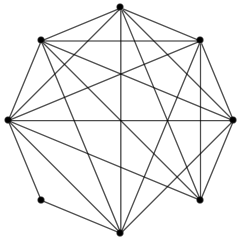}
        \caption*{$B_{35}$}
    \end{subfigure}
    \begin{subfigure}[t]{0.15\textwidth}
        \centering
        \includegraphics[height=0.75in]{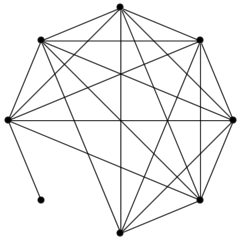}
        \caption*{$B_{36}$}
    \end{subfigure}
\end{figure}
\begin{figure}[!h]
\ContinuedFloat
    \centering
    \begin{subfigure}[t]{0.15\textwidth}
        \centering
        \includegraphics[height=0.75in]{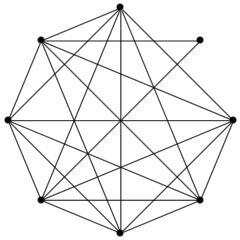}
        \caption*{$B_{37}$}
    \end{subfigure}
    \begin{subfigure}[t]{0.15\textwidth}
        \centering
        \includegraphics[height=0.75in]{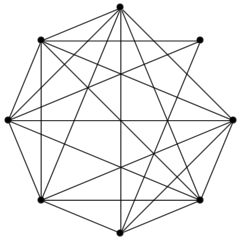}
        \caption*{$B_{38}$}
    \end{subfigure}
    \begin{subfigure}[t]{0.15\textwidth}
        \centering
        \includegraphics[height=0.75in]{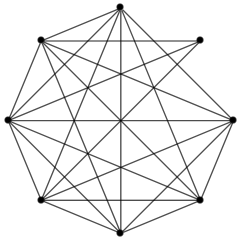}
        \caption*{$B_{39}$}
    \end{subfigure}
    \begin{subfigure}[t]{0.15\textwidth}
        \centering
        \includegraphics[height=0.75in]{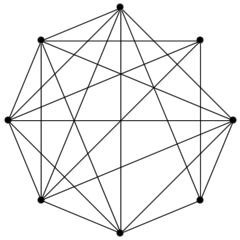}
        \caption*{$B_{40}$}
    \end{subfigure}
    \begin{subfigure}[t]{0.15\textwidth}
        \centering
        \includegraphics[height=0.75in]{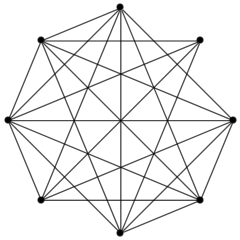}
        \caption*{$B_{41}$}
    \end{subfigure}
    \begin{subfigure}[t]{0.15\textwidth}
        \centering
        \includegraphics[height=0.75in]{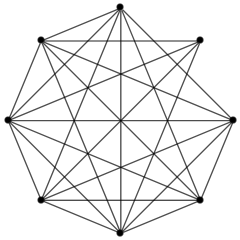}
        \caption*{$B_{42}$}
    \end{subfigure}\\[3ex]
\end{figure}
\begin{figure}[!h]
\ContinuedFloat
    \centering
    \begin{subfigure}[t]{0.15\textwidth}
        \centering
        \includegraphics[height=0.75in]{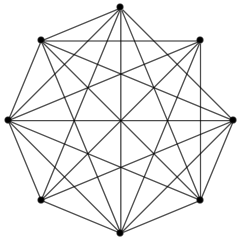}
        \caption*{$B_{43}$}
    \end{subfigure}
    \begin{subfigure}[t]{0.15\textwidth}
        \centering
        \includegraphics[height=0.75in]{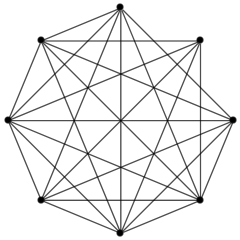}
        \caption*{$B_{44}$}
    \end{subfigure}
    \begin{subfigure}[t]{0.15\textwidth}
        \centering
        \includegraphics[height=0.75in]{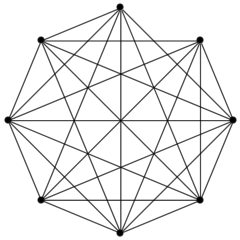}
        \caption*{$B_{45}$}
    \end{subfigure}
\end{figure}

\FloatBarrier

\section*{Appendix C: Connected Graphs With a Largest Clique Size of Five}\label{Sec:AppC}

\FloatBarrier

\begin{figure}[!ht]
    \centering
    \begin{subfigure}[t]{0.15\textwidth}
        \centering
        \includegraphics[height=0.75in]{3}
        \caption*{$C_{1}$}
    \end{subfigure}
    \begin{subfigure}[t]{0.15\textwidth}
        \centering
        \includegraphics[height=0.75in]{5}
        \caption*{$C_{2}$}
    \end{subfigure}
    \begin{subfigure}[t]{0.15\textwidth}
        \centering
        \includegraphics[height=0.75in]{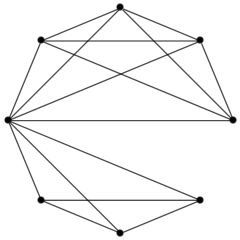}
        \caption*{$C_{3}*$}
    \end{subfigure}
    \begin{subfigure}[t]{0.15\textwidth}
        \centering
        \includegraphics[height=0.75in]{9}
        \caption*{$C_{4}$}
    \end{subfigure}
    \begin{subfigure}[t]{0.15\textwidth}
        \centering
        \includegraphics[height=0.75in]{10}
        \caption*{$C_{5}$}
    \end{subfigure}
    \begin{subfigure}[t]{0.15\textwidth}
        \centering
        \includegraphics[height=0.75in]{13}
        \caption*{$C_{6}$}
    \end{subfigure}
\end{figure}
\begin{figure}[!ht]
\ContinuedFloat
    \centering
    \begin{subfigure}[t]{0.15\textwidth}
        \centering
        \includegraphics[height=0.75in]{15}
        \caption*{$C_{7}$}
    \end{subfigure}
    \begin{subfigure}[t]{0.15\textwidth}
        \centering
        \includegraphics[height=0.75in]{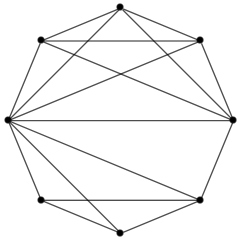}
        \caption*{$C_{8}*$}
    \end{subfigure}
    \begin{subfigure}[t]{0.15\textwidth}
        \centering
        \includegraphics[height=0.75in]{19}
        \caption*{$C_{9}$}
    \end{subfigure}
    \begin{subfigure}[t]{0.15\textwidth}
        \centering
        \includegraphics[height=0.75in]{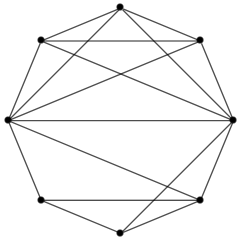}
        \caption*{$C_{10}*$}
    \end{subfigure}
    \begin{subfigure}[t]{0.15\textwidth}
        \centering
        \includegraphics[height=0.75in]{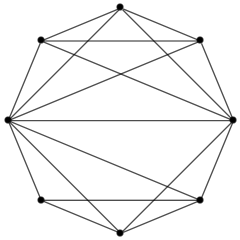}
        \caption*{$C_{11}*$}
    \end{subfigure}
    \begin{subfigure}[t]{0.15\textwidth}
        \centering
        \includegraphics[height=0.75in]{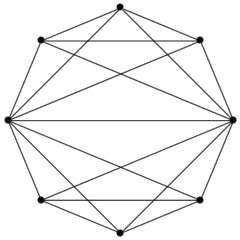}
        \caption*{$C_{12}*$}
    \end{subfigure}
\end{figure}
\begin{figure}[!ht]
\ContinuedFloat
    \centering
    \begin{subfigure}[t]{0.15\textwidth}
        \centering
        \includegraphics[height=0.75in]{28}
        \caption*{$C_{13}$}
    \end{subfigure}
    \begin{subfigure}[t]{0.15\textwidth}
        \centering
        \includegraphics[height=0.75in]{29}
        \caption*{$C_{14}$}
    \end{subfigure}
    \begin{subfigure}[t]{0.15\textwidth}
        \centering
        \includegraphics[height=0.75in]{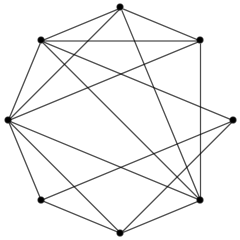}
        \caption*{$C_{15}*$}
    \end{subfigure}
    \begin{subfigure}[t]{0.15\textwidth}
        \centering
        \includegraphics[height=0.75in]{34}
        \caption*{$C_{16}$}
    \end{subfigure}
    \begin{subfigure}[t]{0.15\textwidth}
        \centering
        \includegraphics[height=0.75in]{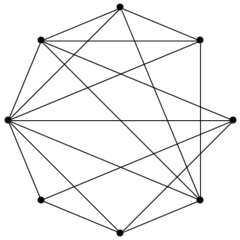}
        \caption*{$C_{17}*$}
    \end{subfigure}
    \begin{subfigure}[t]{0.15\textwidth}
        \centering
        \includegraphics[height=0.75in]{37}
        \caption*{$C_{18}$}
    \end{subfigure}
\end{figure}
\begin{figure}[!ht]
\ContinuedFloat
    \centering
    \begin{subfigure}[t]{0.15\textwidth}
        \centering
        \includegraphics[height=0.75in]{40}
        \caption*{$C_{19}$}
    \end{subfigure}
    \begin{subfigure}[t]{0.15\textwidth}
        \centering
        \includegraphics[height=0.75in]{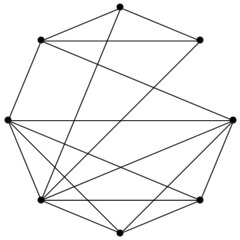}
        \caption*{$C_{20}*$}
    \end{subfigure}
    \begin{subfigure}[t]{0.15\textwidth}
        \centering
        \includegraphics[height=0.75in]{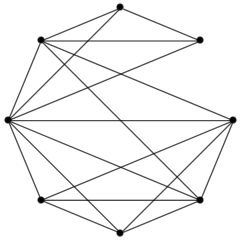}
        \caption*{$C_{21}*$}
    \end{subfigure}
    \begin{subfigure}[t]{0.15\textwidth}
        \centering
        \includegraphics[height=0.75in]{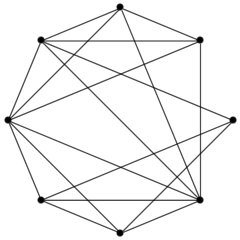}
        \caption*{$C_{22}*$}
    \end{subfigure}
    \begin{subfigure}[t]{0.15\textwidth}
        \centering
        \includegraphics[height=0.75in]{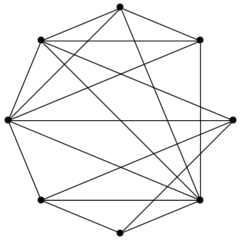}
        \caption*{$C_{23}*$}
    \end{subfigure}
    \begin{subfigure}[t]{0.15\textwidth}
        \centering
        \includegraphics[height=0.75in]{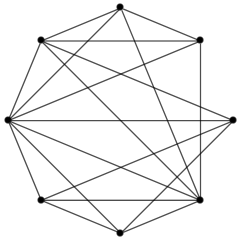}
        \caption*{$C_{24}*$}
    \end{subfigure}
\end{figure}
\begin{figure}[!ht]
\ContinuedFloat
    \centering
    \begin{subfigure}[t]{0.15\textwidth}
        \centering
        \includegraphics[height=0.75in]{52}
        \caption*{$C_{25}$}
    \end{subfigure}
    \begin{subfigure}[t]{0.15\textwidth}
        \centering
        \includegraphics[height=0.75in]{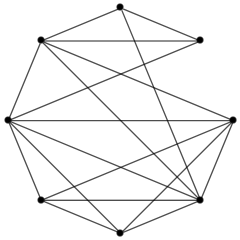}
        \caption*{$C_{26}*$}
    \end{subfigure}
    \begin{subfigure}[t]{0.15\textwidth}
        \centering
        \includegraphics[height=0.75in]{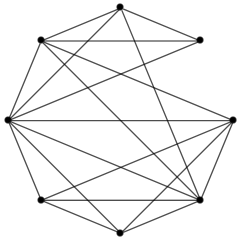}
        \caption*{$C_{27}*$}
    \end{subfigure}
    \begin{subfigure}[t]{0.15\textwidth}
        \centering
        \includegraphics[height=0.75in]{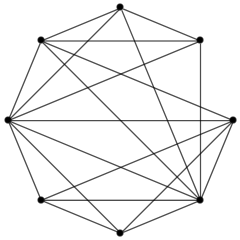}
        \caption*{$C_{28}*$}
    \end{subfigure}
    \begin{subfigure}[t]{0.15\textwidth}
        \centering
        \includegraphics[height=0.75in]{58}
        \caption*{$C_{29}$}
    \end{subfigure}
    \begin{subfigure}[t]{0.15\textwidth}
        \centering
        \includegraphics[height=0.75in]{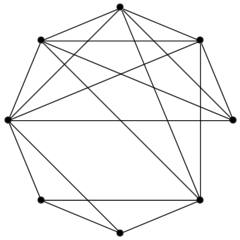}
        \caption*{$C_{30}*$}
    \end{subfigure}\\[3ex]
\end{figure}
\begin{figure}[!ht]
\ContinuedFloat
    \centering
    \begin{subfigure}[t]{0.15\textwidth}
        \centering
        \includegraphics[height=0.75in]{62}
        \caption*{$C_{31}$}
    \end{subfigure}
    \begin{subfigure}[t]{0.15\textwidth}
        \centering
        \includegraphics[height=0.75in]{64}
        \caption*{$C_{32}$}
    \end{subfigure}
    \begin{subfigure}[t]{0.15\textwidth}
        \centering
        \includegraphics[height=0.75in]{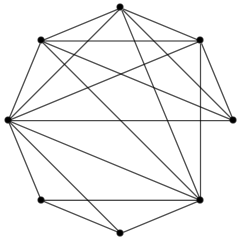}
        \caption*{$C_{33}*$}
    \end{subfigure}
    \begin{subfigure}[t]{0.15\textwidth}
        \centering
        \includegraphics[height=0.75in]{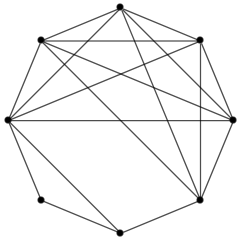}
        \caption*{$C_{34}*$}
    \end{subfigure}
    \begin{subfigure}[t]{0.15\textwidth}
        \centering
        \includegraphics[height=0.75in]{70}
        \caption*{$C_{35}$}
    \end{subfigure}
    \begin{subfigure}[t]{0.15\textwidth}
        \centering
        \includegraphics[height=0.75in]{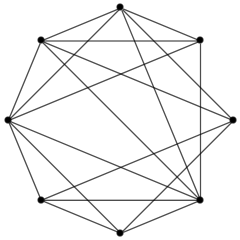}
        \caption*{$C_{36}*$}
    \end{subfigure}\\[3ex]
\end{figure}
\begin{figure}[!ht]
\ContinuedFloat
    \centering
    \begin{subfigure}[t]{0.15\textwidth}
        \centering
        \includegraphics[height=0.75in]{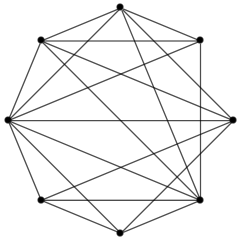}
        \caption*{$C_{37}*$}
    \end{subfigure}
    \begin{subfigure}[t]{0.15\textwidth}
        \centering
        \includegraphics[height=0.75in]{76}
        \caption*{$C_{38}$}
    \end{subfigure}
    \begin{subfigure}[t]{0.15\textwidth}
        \centering
        \includegraphics[height=0.75in]{77}
        \caption*{$C_{39}$}
    \end{subfigure}
    \begin{subfigure}[t]{0.15\textwidth}
        \centering
        \includegraphics[height=0.75in]{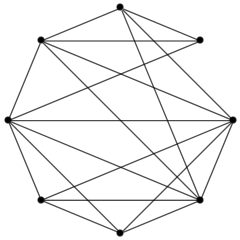}
        \caption*{$C_{40}*$}
    \end{subfigure}
    \begin{subfigure}[t]{0.15\textwidth}
        \centering
        \includegraphics[height=0.75in]{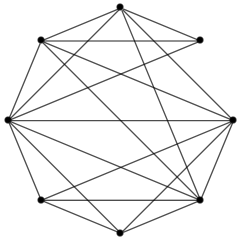}
        \caption*{$C_{41}*$}
    \end{subfigure}
    \begin{subfigure}[t]{0.15\textwidth}
        \centering
        \includegraphics[height=0.75in]{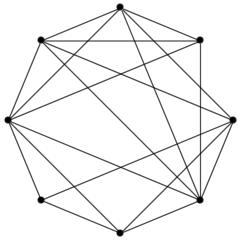}
        \caption*{$C_{42}*$}
    \end{subfigure}\\[3ex]
\end{figure}
\begin{figure}[!ht]
\ContinuedFloat
    \centering
    \begin{subfigure}[t]{0.15\textwidth}
        \centering
        \includegraphics[height=0.75in]{82}
        \caption*{$C_{43}$}
    \end{subfigure}
    \begin{subfigure}[t]{0.15\textwidth}
        \centering
        \includegraphics[height=0.75in]{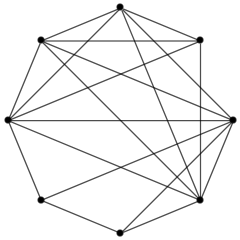}
        \caption*{$C_{44}*$}
    \end{subfigure}
    \begin{subfigure}[t]{0.15\textwidth}
        \centering
        \includegraphics[height=0.75in]{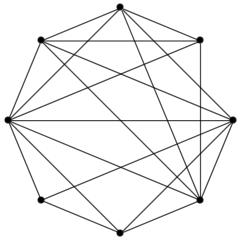}
        \caption*{$C_{45}*$}
    \end{subfigure}
    \begin{subfigure}[t]{0.15\textwidth}
        \centering
        \includegraphics[height=0.75in]{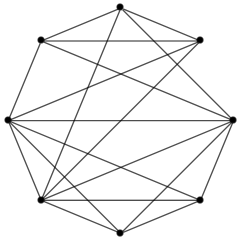}
        \caption*{$C_{46}*$}
    \end{subfigure}
    \begin{subfigure}[t]{0.15\textwidth}
        \centering
        \includegraphics[height=0.75in]{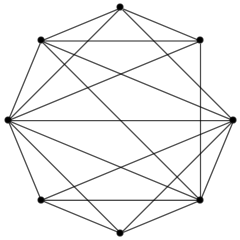}
        \caption*{$C_{47}*$}
    \end{subfigure}
    \begin{subfigure}[t]{0.15\textwidth}
        \centering
        \includegraphics[height=0.75in]{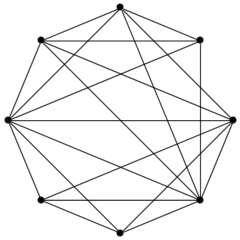}
        \caption*{$C_{48}*$}
    \end{subfigure}\\[3ex]
\end{figure}
\begin{figure}[!ht]
\ContinuedFloat
    \centering
    \begin{subfigure}[t]{0.15\textwidth}
        \centering
        \includegraphics[height=0.75in]{90}
        \caption*{$C_{49}$}
    \end{subfigure}
    \begin{subfigure}[t]{0.15\textwidth}
        \centering
        \includegraphics[height=0.75in]{91}
        \caption*{$C_{50}$}
    \end{subfigure}
    \begin{subfigure}[t]{0.15\textwidth}
        \centering
        \includegraphics[height=0.75in]{94}
        \caption*{$C_{51}$}
    \end{subfigure}
    \begin{subfigure}[t]{0.15\textwidth}
        \centering
        \includegraphics[height=0.75in]{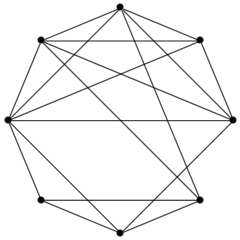}
        \caption*{$C_{52}*$}
    \end{subfigure}
    \begin{subfigure}[t]{0.15\textwidth}
        \centering
        \includegraphics[height=0.75in]{99}
        \caption*{$C_{53}$}
    \end{subfigure}
    \begin{subfigure}[t]{0.15\textwidth}
        \centering
        \includegraphics[height=0.75in]{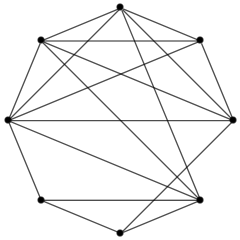}
        \caption*{$C_{54}*$}
    \end{subfigure}\\[3ex]
\end{figure}
\begin{figure}[!ht]
\ContinuedFloat
    \centering
    \begin{subfigure}[t]{0.15\textwidth}
        \centering
        \includegraphics[height=0.75in]{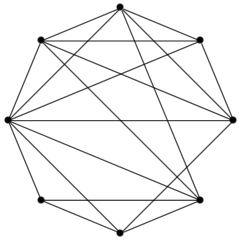}
        \caption*{$C_{55}*$}
    \end{subfigure}
    \begin{subfigure}[t]{0.15\textwidth}
        \centering
        \includegraphics[height=0.75in]{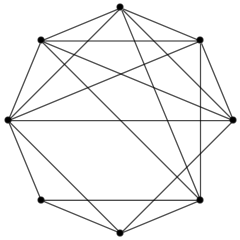}
        \caption*{$C_{56}*$}
    \end{subfigure}
    \begin{subfigure}[t]{0.15\textwidth}
        \centering
        \includegraphics[height=0.75in]{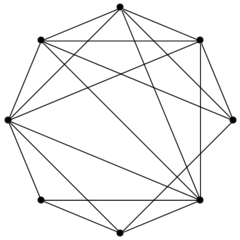}
        \caption*{$C_{57}*$}
    \end{subfigure}
    \begin{subfigure}[t]{0.15\textwidth}
        \centering
        \includegraphics[height=0.75in]{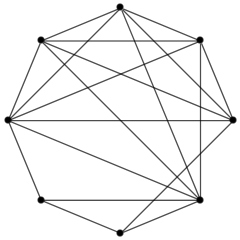}
        \caption*{$C_{58}*$}
    \end{subfigure}
    \begin{subfigure}[t]{0.15\textwidth}
        \centering
        \includegraphics[height=0.75in]{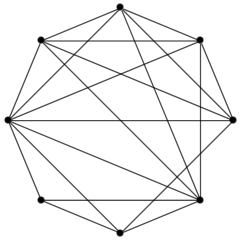}
        \caption*{$C_{59}*$}
    \end{subfigure}
    \begin{subfigure}[t]{0.15\textwidth}
        \centering
        \includegraphics[height=0.75in]{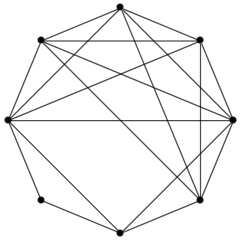}
        \caption*{$C_{60}*$}
    \end{subfigure}
\end{figure}
\begin{figure}[!ht]
\ContinuedFloat
    \centering
    \begin{subfigure}[t]{0.15\textwidth}
        \centering
        \includegraphics[height=0.75in]{113}
        \caption*{$C_{61}$}
    \end{subfigure}
    \begin{subfigure}[t]{0.15\textwidth}
        \centering
        \includegraphics[height=0.75in]{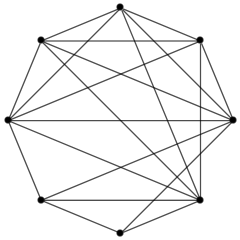}
        \caption*{$C_{62}*$}
    \end{subfigure}
    \begin{subfigure}[t]{0.15\textwidth}
        \centering
        \includegraphics[height=0.75in]{116}
        \caption*{$C_{63}$}
    \end{subfigure}
    \begin{subfigure}[t]{0.15\textwidth}
        \centering
        \includegraphics[height=0.75in]{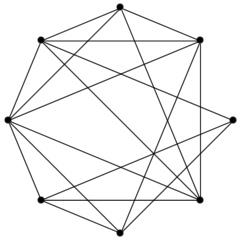}
        \caption*{$C_{64}*$}
    \end{subfigure}
    \begin{subfigure}[t]{0.15\textwidth}
        \centering
        \includegraphics[height=0.75in]{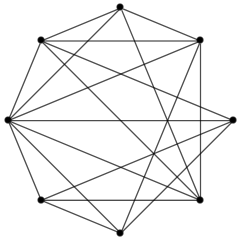}
        \caption*{$C_{65}*$}
    \end{subfigure}
    \begin{subfigure}[t]{0.15\textwidth}
        \centering
        \includegraphics[height=0.75in]{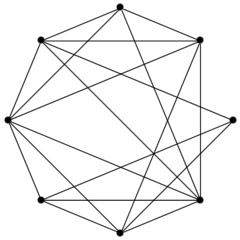}
        \caption*{$C_{66}*$}
    \end{subfigure}\\[3ex]
\end{figure}
\begin{figure}[!ht]
\ContinuedFloat
    \centering
    \begin{subfigure}[t]{0.15\textwidth}
        \centering
        \includegraphics[height=0.75in]{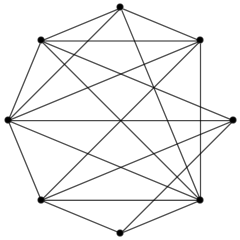}
        \caption*{$C_{67}*$}
    \end{subfigure}
    \begin{subfigure}[t]{0.15\textwidth}
        \centering
        \includegraphics[height=0.75in]{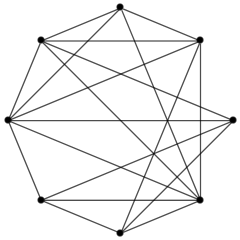}
        \caption*{$C_{68}*$}
    \end{subfigure}
    \begin{subfigure}[t]{0.15\textwidth}
        \centering
        \includegraphics[height=0.75in]{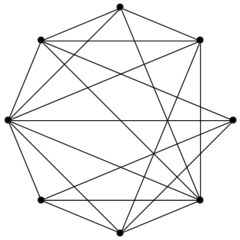}
        \caption*{$C_{69}*$}
    \end{subfigure}
    \begin{subfigure}[t]{0.15\textwidth}
        \centering
        \includegraphics[height=0.75in]{136}
        \caption*{$C_{70}$}
    \end{subfigure}
    \begin{subfigure}[t]{0.15\textwidth}
        \centering
        \includegraphics[height=0.75in]{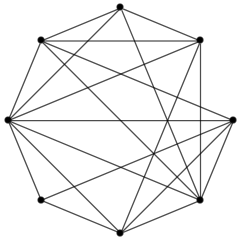}
        \caption*{$C_{71}*$}
    \end{subfigure}
    \begin{subfigure}[t]{0.15\textwidth}
        \centering
        \includegraphics[height=0.75in]{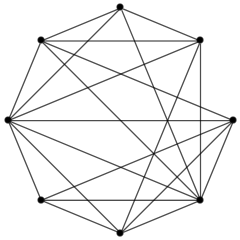}
        \caption*{$C_{72}*$}
    \end{subfigure}\\[3ex]
\end{figure}
\begin{figure}[!ht]
\ContinuedFloat
    \centering
    \begin{subfigure}[t]{0.15\textwidth}
        \centering
        \includegraphics[height=0.75in]{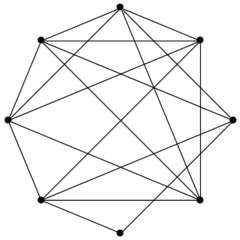}
        \caption*{$C_{73}*$}
    \end{subfigure}
    \begin{subfigure}[t]{0.15\textwidth}
        \centering
        \includegraphics[height=0.75in]{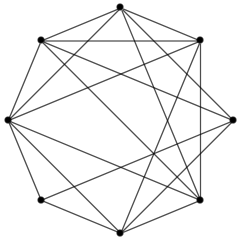}
        \caption*{$C_{74}*$}
    \end{subfigure}
    \begin{subfigure}[t]{0.15\textwidth}
        \centering
        \includegraphics[height=0.75in]{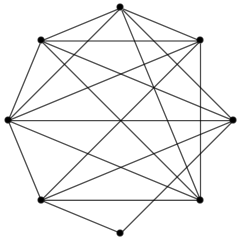}
        \caption*{$C_{75}*$}
    \end{subfigure}
    \begin{subfigure}[t]{0.15\textwidth}
        \centering
        \includegraphics[height=0.75in]{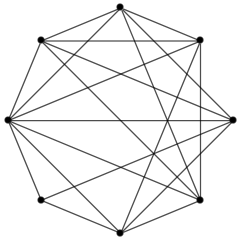}
        \caption*{$C_{76}*$}
    \end{subfigure}
    \begin{subfigure}[t]{0.15\textwidth}
        \centering
        \includegraphics[height=0.75in]{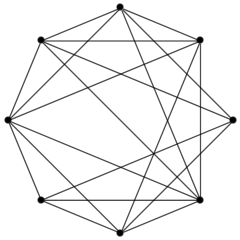}
        \caption*{$C_{77}*$}
    \end{subfigure}
    \begin{subfigure}[t]{0.15\textwidth}
        \centering
        \includegraphics[height=0.75in]{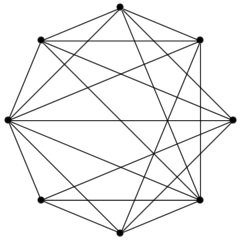}
        \caption*{$C_{78}*$}
    \end{subfigure}\\[3ex]
\end{figure}
\begin{figure}[!ht]
\ContinuedFloat
    \centering
    \begin{subfigure}[t]{0.15\textwidth}
        \centering
        \includegraphics[height=0.75in]{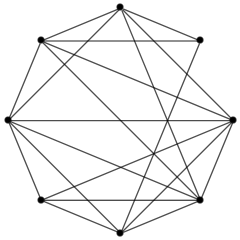}
        \caption*{$C_{79}*$}
    \end{subfigure}
    \begin{subfigure}[t]{0.15\textwidth}
        \centering
        \includegraphics[height=0.75in]{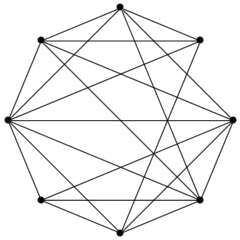}
        \caption*{$C_{80}*$}
    \end{subfigure}
    \begin{subfigure}[t]{0.15\textwidth}
        \centering
        \includegraphics[height=0.75in]{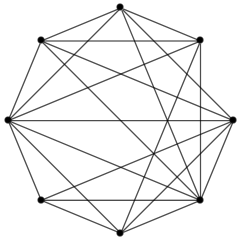}
        \caption*{$C_{81}*$}
    \end{subfigure}
    \begin{subfigure}[t]{0.15\textwidth}
        \centering
        \includegraphics[height=0.75in]{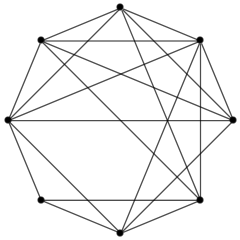}
        \caption*{$C_{82}*$}
    \end{subfigure}
    \begin{subfigure}[t]{0.15\textwidth}
        \centering
        \includegraphics[height=0.75in]{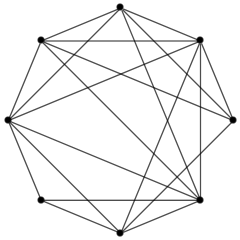}
        \caption*{$C_{83}*$}
    \end{subfigure}
    \begin{subfigure}[t]{0.15\textwidth}
        \centering
        \includegraphics[height=0.75in]{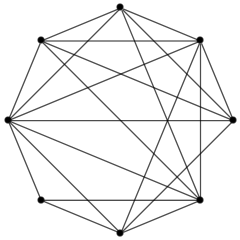}
        \caption*{$C_{84}*$}
    \end{subfigure}\\[3ex]
\end{figure}
\begin{figure}[!ht]
\ContinuedFloat
    \centering
    \begin{subfigure}[t]{0.15\textwidth}
        \centering
        \includegraphics[height=0.75in]{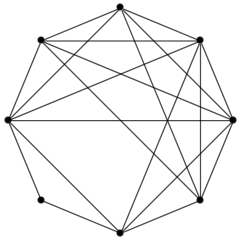}
        \caption*{$C_{85}*$}
    \end{subfigure}
    \begin{subfigure}[t]{0.15\textwidth}
        \centering
        \includegraphics[height=0.75in]{162}
        \caption*{$C_{86}$}
    \end{subfigure}
    \begin{subfigure}[t]{0.15\textwidth}
        \centering
        \includegraphics[height=0.75in]{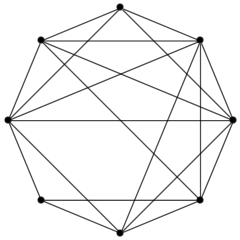}
        \caption*{$C_{87}*$}
    \end{subfigure}
    \begin{subfigure}[t]{0.15\textwidth}
        \centering
        \includegraphics[height=0.75in]{167}
        \caption*{$C_{88}$}
    \end{subfigure}
    \begin{subfigure}[t]{0.15\textwidth}
        \centering
        \includegraphics[height=0.75in]{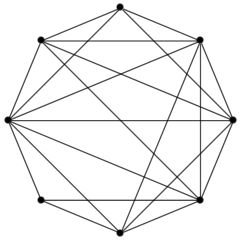}
        \caption*{$C_{89}*$}
    \end{subfigure}
    \begin{subfigure}[t]{0.15\textwidth}
        \centering
        \includegraphics[height=0.75in]{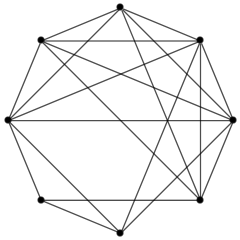}
        \caption*{$C_{90}*$}
    \end{subfigure}\\[3ex]
\end{figure}
\begin{figure}[!ht]
\ContinuedFloat
    \centering
    \begin{subfigure}[t]{0.15\textwidth}
        \centering
        \includegraphics[height=0.75in]{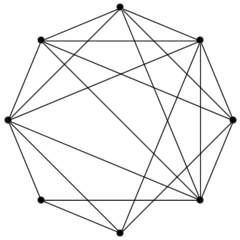}
        \caption*{$C_{91}*$}
    \end{subfigure}
    \begin{subfigure}[t]{0.15\textwidth}
        \centering
        \includegraphics[height=0.75in]{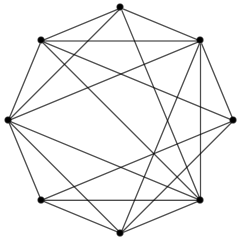}
        \caption*{$C_{92}*$}
    \end{subfigure}
    \begin{subfigure}[t]{0.15\textwidth}
        \centering
        \includegraphics[height=0.75in]{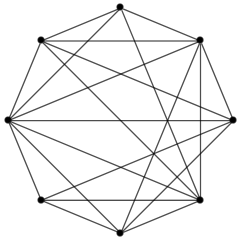}
        \caption*{$C_{93}*$}
    \end{subfigure}
    \begin{subfigure}[t]{0.15\textwidth}
        \centering
        \includegraphics[height=0.75in]{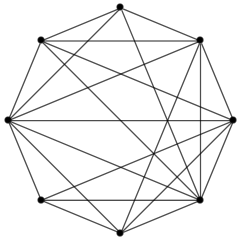}
        \caption*{$C_{94}*$}
    \end{subfigure}
    \begin{subfigure}[t]{0.15\textwidth}
        \centering
        \includegraphics[height=0.75in]{179}
        \caption*{$C_{95}$}
    \end{subfigure}
    \begin{subfigure}[t]{0.15\textwidth}
        \centering
        \includegraphics[height=0.75in]{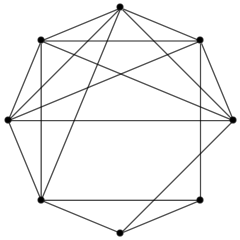}
        \caption*{$C_{96}*$}
    \end{subfigure}\\[3ex]
\end{figure}
\begin{figure}[!ht]
\ContinuedFloat
    \centering
    \begin{subfigure}[t]{0.15\textwidth}
        \centering
        \includegraphics[height=0.75in]{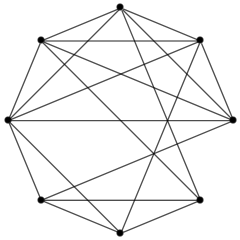}
        \caption*{$C_{97}*$}
    \end{subfigure}
    \begin{subfigure}[t]{0.15\textwidth}
        \centering
        \includegraphics[height=0.75in]{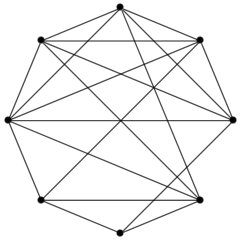}
        \caption*{$C_{98}*$}
    \end{subfigure}
    \begin{subfigure}[t]{0.15\textwidth}
        \centering
        \includegraphics[height=0.75in]{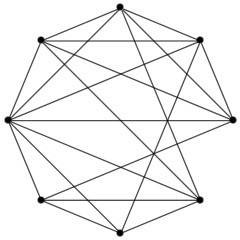}
        \caption*{$C_{99}*$}
    \end{subfigure}
    \begin{subfigure}[t]{0.15\textwidth}
        \centering
        \includegraphics[height=0.75in]{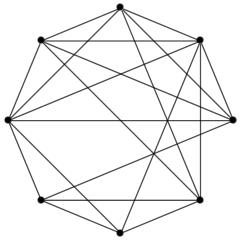}
        \caption*{$C_{100}*$}
    \end{subfigure}
    \begin{subfigure}[t]{0.15\textwidth}
        \centering
        \includegraphics[height=0.75in]{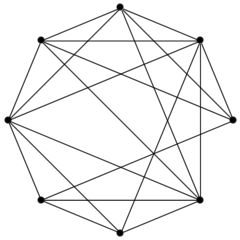}
        \caption*{$C_{101}*$}
    \end{subfigure}
    \begin{subfigure}[t]{0.15\textwidth}
        \centering
        \includegraphics[height=0.75in]{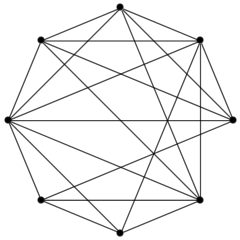}
        \caption*{$C_{102}*$}
    \end{subfigure}\\[3ex]
\end{figure}
\begin{figure}[!ht]
\ContinuedFloat
    \centering
    \begin{subfigure}[t]{0.15\textwidth}
        \centering
        \includegraphics[height=0.75in]{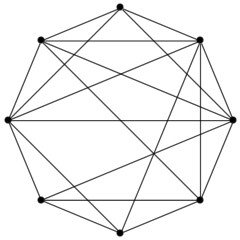}
        \caption*{$C_{103}*$}
    \end{subfigure}
    \begin{subfigure}[t]{0.15\textwidth}
        \centering
        \includegraphics[height=0.75in]{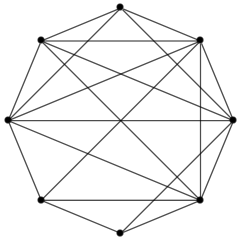}
        \caption*{$C_{104}*$}
    \end{subfigure}
    \begin{subfigure}[t]{0.15\textwidth}
        \centering
        \includegraphics[height=0.75in]{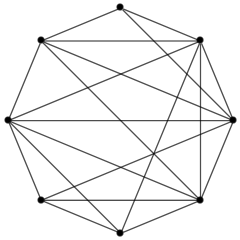}
        \caption*{$C_{105}*$}
    \end{subfigure}
    \begin{subfigure}[t]{0.15\textwidth}
        \centering
        \includegraphics[height=0.75in]{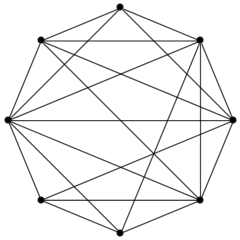}
        \caption*{$C_{106}*$}
    \end{subfigure}
    \begin{subfigure}[t]{0.15\textwidth}
        \centering
        \includegraphics[height=0.75in]{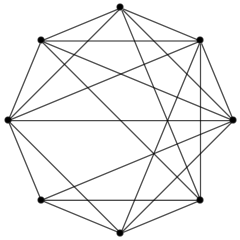}
        \caption*{$C_{107}*$}
    \end{subfigure}
    \begin{subfigure}[t]{0.15\textwidth}
        \centering
        \includegraphics[height=0.75in]{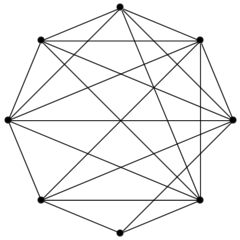}
        \caption*{$C_{108}*$}
    \end{subfigure}\\[3ex]
\end{figure}
\begin{figure}[!ht]
\ContinuedFloat
    \centering
    \begin{subfigure}[t]{0.15\textwidth}
        \centering
        \includegraphics[height=0.75in]{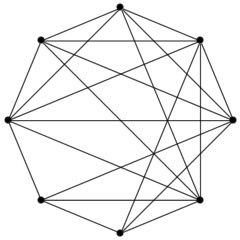}
        \caption*{$C_{109}*$}
    \end{subfigure}
    \begin{subfigure}[t]{0.15\textwidth}
        \centering
        \includegraphics[height=0.75in]{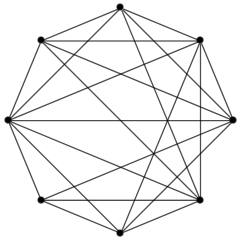}
        \caption*{$C_{110}*$}
    \end{subfigure}
    \begin{subfigure}[t]{0.15\textwidth}
        \centering
        \includegraphics[height=0.75in]{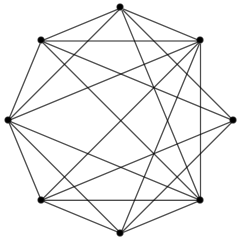}
        \caption*{$C_{111}*$}
    \end{subfigure}
    \begin{subfigure}[t]{0.15\textwidth}
        \centering
        \includegraphics[height=0.75in]{210}
        \caption*{$C_{112}$}
    \end{subfigure}
    \begin{subfigure}[t]{0.15\textwidth}
        \centering
        \includegraphics[height=0.75in]{212}
        \caption*{$C_{113}$}
    \end{subfigure}
    \begin{subfigure}[t]{0.15\textwidth}
        \centering
        \includegraphics[height=0.75in]{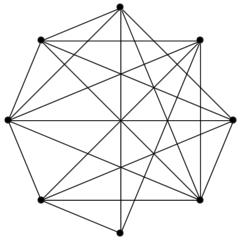}
        \caption*{$C_{114}*$}
    \end{subfigure}\\[3ex]
\end{figure}
\begin{figure}[!ht]
\ContinuedFloat
    \centering
    \begin{subfigure}[t]{0.15\textwidth}
        \centering
        \includegraphics[height=0.75in]{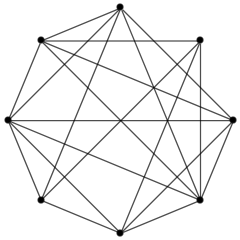}
        \caption*{$C_{115}*$}
    \end{subfigure}
    \begin{subfigure}[t]{0.15\textwidth}
        \centering
        \includegraphics[height=0.75in]{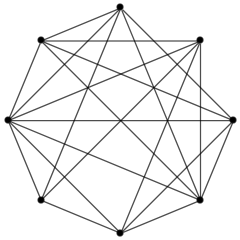}
        \caption*{$C_{116}*$}
    \end{subfigure}
    \begin{subfigure}[t]{0.15\textwidth}
        \centering
        \includegraphics[height=0.75in]{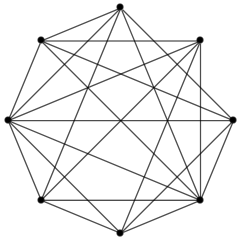}
        \caption*{$C_{117}*$}
    \end{subfigure}
    \begin{subfigure}[t]{0.15\textwidth}
        \centering
        \includegraphics[height=0.75in]{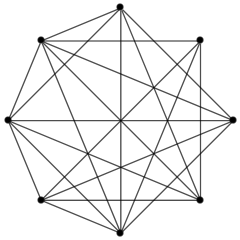}
        \caption*{$C_{118}*$}
    \end{subfigure}
    \begin{subfigure}[t]{0.15\textwidth}
        \centering
        \includegraphics[height=0.75in]{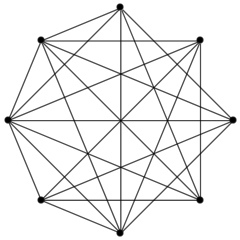}
        \caption*{$C_{119}*$}
    \end{subfigure}
    \begin{subfigure}[t]{0.15\textwidth}
        \centering
        \includegraphics[height=0.75in]{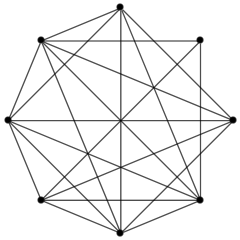}
        \caption*{$C_{120}*$}
    \end{subfigure}\\[3ex]
\end{figure}
\begin{figure}[!ht]
\ContinuedFloat
    \centering
    \begin{subfigure}[t]{0.15\textwidth}
        \centering
        \includegraphics[height=0.75in]{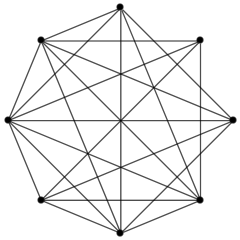}
        \caption*{$C_{121}*$}
    \end{subfigure}
    \begin{subfigure}[t]{0.15\textwidth}
        \centering
        \includegraphics[height=0.75in]{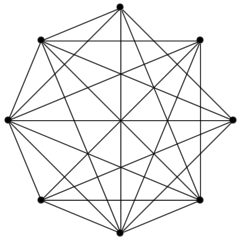}
        \caption*{$C_{122}*$}
    \end{subfigure}
    \begin{subfigure}[t]{0.15\textwidth}
        \centering
        \includegraphics[height=0.75in]{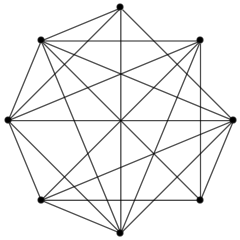}
        \caption*{$C_{123}*$}
    \end{subfigure}
    \begin{subfigure}[t]{0.15\textwidth}
        \centering
        \includegraphics[height=0.75in]{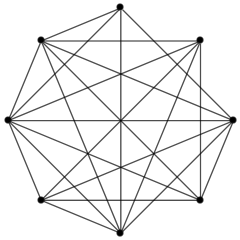}
        \caption*{$C_{124}*$}
    \end{subfigure}
    \begin{subfigure}[t]{0.15\textwidth}
        \centering
        \includegraphics[height=0.75in]{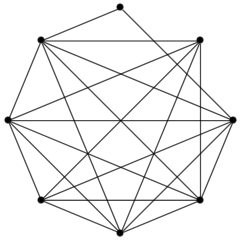}
        \caption*{$C_{125}*$}
    \end{subfigure}
    \begin{subfigure}[t]{0.15\textwidth}
        \centering
        \includegraphics[height=0.75in]{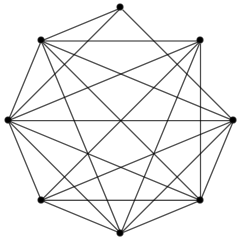}
        \caption*{$C_{126}*$}
    \end{subfigure}\\[3ex]
\end{figure}
\begin{figure}[!ht]
\ContinuedFloat
    \centering
    \begin{subfigure}[t]{0.15\textwidth}
        \centering
        \includegraphics[height=0.75in]{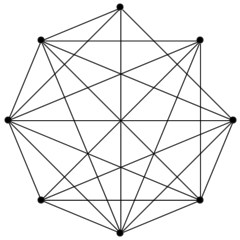}
        \caption*{$C_{127}*$}
    \end{subfigure}
    \begin{subfigure}[t]{0.15\textwidth}
        \centering
        \includegraphics[height=0.75in]{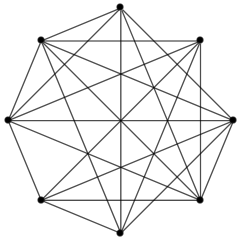}
        \caption*{$C_{128}*$}
    \end{subfigure}
    \begin{subfigure}[t]{0.15\textwidth}
        \centering
        \includegraphics[height=0.75in]{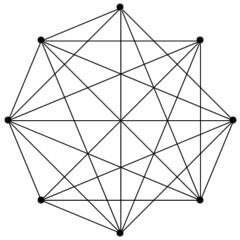}
        \caption*{$C_{129}*$}
    \end{subfigure}
    \begin{subfigure}[t]{0.15\textwidth}
        \centering
        \includegraphics[height=0.75in]{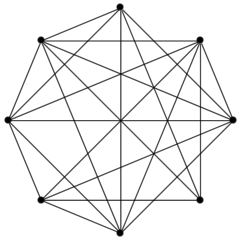}
        \caption*{$C_{130}*$}
    \end{subfigure}
    \begin{subfigure}[t]{0.15\textwidth}
        \centering
        \includegraphics[height=0.75in]{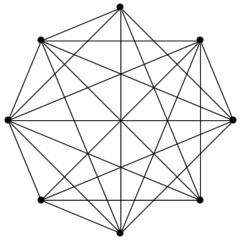}
        \caption*{$C_{131}*$}
    \end{subfigure}
    \begin{subfigure}[t]{0.15\textwidth}
        \centering
        \includegraphics[height=0.75in]{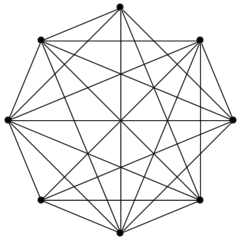}
        \caption*{$C_{132}*$}
    \end{subfigure}\\[3ex]
\end{figure}
\begin{figure}[!ht]
\ContinuedFloat
    \centering
    \begin{subfigure}[t]{0.15\textwidth}
        \centering
        \includegraphics[height=0.75in]{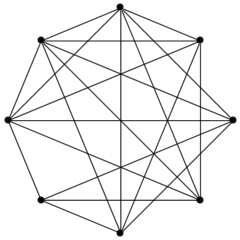}
        \caption*{$C_{133}*$}
    \end{subfigure}
    \begin{subfigure}[t]{0.15\textwidth}
        \centering
        \includegraphics[height=0.75in]{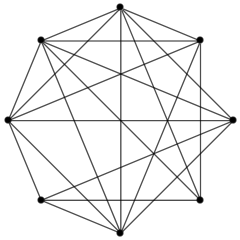}
        \caption*{$C_{134}*$}
    \end{subfigure}
    \begin{subfigure}[t]{0.15\textwidth}
        \centering
        \includegraphics[height=0.75in]{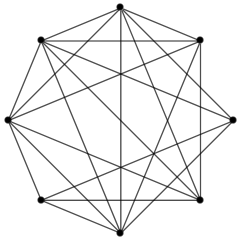}
        \caption*{$C_{135}*$}
    \end{subfigure}
    \begin{subfigure}[t]{0.15\textwidth}
        \centering
        \includegraphics[height=0.75in]{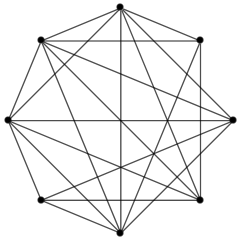}
        \caption*{$C_{136}*$}
    \end{subfigure}
    \begin{subfigure}[t]{0.15\textwidth}
        \centering
        \includegraphics[height=0.75in]{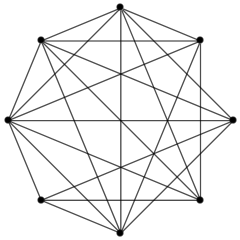}
        \caption*{$C_{137}*$}
    \end{subfigure}
    \begin{subfigure}[t]{0.15\textwidth}
        \centering
        \includegraphics[height=0.75in]{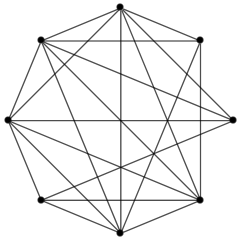}
        \caption*{$C_{138}*$}
    \end{subfigure}\\[3ex]
\end{figure}
\begin{figure}[!ht]
\ContinuedFloat
    \centering
    \begin{subfigure}[t]{0.15\textwidth}
        \centering
        \includegraphics[height=0.75in]{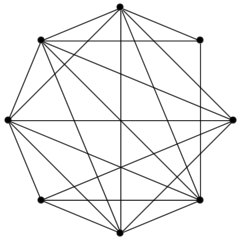}
        \caption*{$C_{139}*$}
    \end{subfigure}
    \begin{subfigure}[t]{0.15\textwidth}
        \centering
        \includegraphics[height=0.75in]{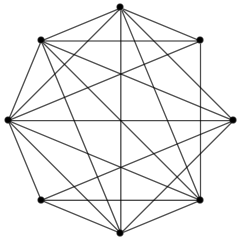}
        \caption*{$C_{140}*$}
    \end{subfigure}
    \begin{subfigure}[t]{0.15\textwidth}
        \centering
        \includegraphics[height=0.75in]{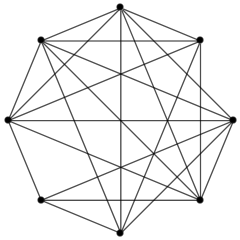}
        \caption*{$C_{141}*$}
    \end{subfigure}
    \begin{subfigure}[t]{0.15\textwidth}
        \centering
        \includegraphics[height=0.75in]{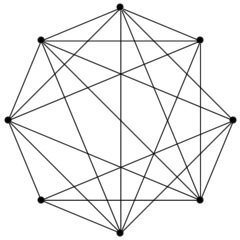}
        \caption*{$C_{142}*$}
    \end{subfigure}
    \begin{subfigure}[t]{0.15\textwidth}
        \centering
        \includegraphics[height=0.75in]{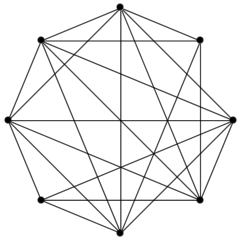}
        \caption*{$C_{143}*$}
    \end{subfigure}
    \begin{subfigure}[t]{0.15\textwidth}
        \centering
        \includegraphics[height=0.75in]{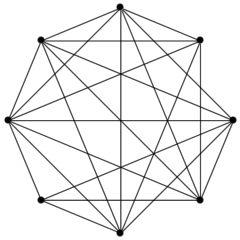}
        \caption*{$C_{144}*$}
    \end{subfigure}\\[3ex]
\end{figure}
\begin{figure}[!ht]
\ContinuedFloat
    \centering
    \begin{subfigure}[t]{0.15\textwidth}
        \centering
        \includegraphics[height=0.75in]{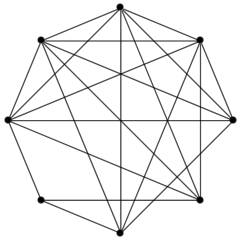}
        \caption*{$C_{145}*$}
    \end{subfigure}
    \begin{subfigure}[t]{0.15\textwidth}
        \centering
        \includegraphics[height=0.75in]{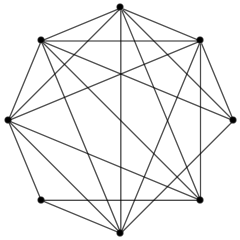}
        \caption*{$C_{146}*$}
    \end{subfigure}
    \begin{subfigure}[t]{0.15\textwidth}
        \centering
        \includegraphics[height=0.75in]{282}
        \caption*{$C_{147}$}
    \end{subfigure}
    \begin{subfigure}[t]{0.15\textwidth}
        \centering
        \includegraphics[height=0.75in]{292}
        \caption*{$C_{148}$}
    \end{subfigure}
    \begin{subfigure}[t]{0.15\textwidth}
        \centering
        \includegraphics[height=0.75in]{293}
        \caption*{$C_{149}$}
    \end{subfigure}
    \begin{subfigure}[t]{0.15\textwidth}
        \centering
        \includegraphics[height=0.75in]{294}
        \caption*{$C_{150}$}
    \end{subfigure}\\[3ex]
\end{figure}
\begin{figure}[!h]
\ContinuedFloat
    \centering
    \begin{subfigure}[t]{0.15\textwidth}
        \centering
        \includegraphics[height=0.75in]{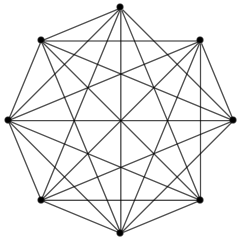}
        \caption*{$C_{151}$}
    \end{subfigure}
\end{figure}

\FloatBarrier

\section*{Appendix D: Connected Graphs With a Largest Clique Size of Four}\label{Sec:AppD}

\begin{figure}[!h]
    \centering
    \begin{subfigure}[t]{0.15\textwidth}
        \centering
        \includegraphics[height=0.75in]{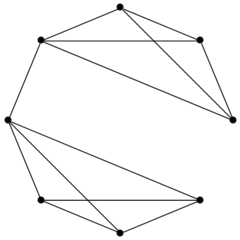}
        \caption*{$D_{1}$}
    \end{subfigure}
    \begin{subfigure}[t]{0.15\textwidth}
        \centering
        \includegraphics[height=0.75in]{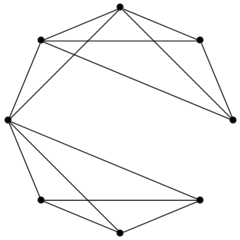}
        \caption*{$D_{2}$}
    \end{subfigure}
    \begin{subfigure}[t]{0.15\textwidth}
        \centering
        \includegraphics[height=0.75in]{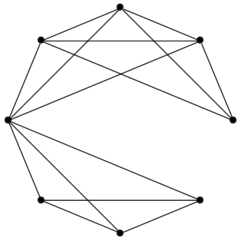}
        \caption*{$D_{3}$}
    \end{subfigure}
    \begin{subfigure}[t]{0.15\textwidth}
        \centering
        \includegraphics[height=0.75in]{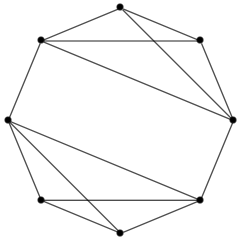}
        \caption*{$D_{4}$}
    \end{subfigure}
    \begin{subfigure}[t]{0.15\textwidth}
        \centering
        \includegraphics[height=0.75in]{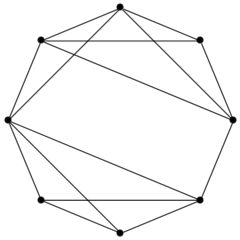}
        \caption*{$D_{5}$}
    \end{subfigure}
    \begin{subfigure}[t]{0.15\textwidth}
        \centering
        \includegraphics[height=0.75in]{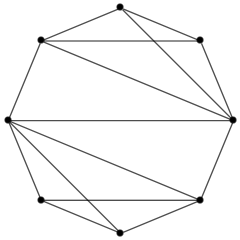}
        \caption*{$D_{6}$}
    \end{subfigure}
\end{figure}
\begin{figure}[!ht]
    \centering
    \begin{subfigure}[t]{0.15\textwidth}
        \centering
        \includegraphics[height=0.75in]{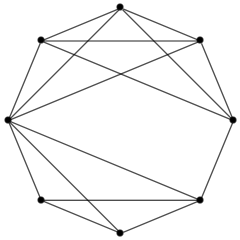}
        \caption*{$D_{7}*$}
    \end{subfigure}
    \begin{subfigure}[t]{0.15\textwidth}
        \centering
        \includegraphics[height=0.75in]{14}
        \caption*{$D_{8}$}
    \end{subfigure}
    \begin{subfigure}[t]{0.15\textwidth}
        \centering
        \includegraphics[height=0.75in]{17}
        \caption*{$D_{9}$}
    \end{subfigure}
    \begin{subfigure}[t]{0.15\textwidth}
        \centering
        \includegraphics[height=0.75in]{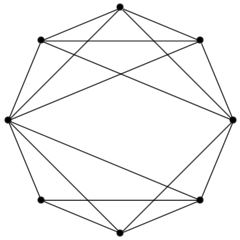}
        \caption*{$D_{10}*$}
    \end{subfigure}
    \begin{subfigure}[t]{0.15\textwidth}
        \centering
        \includegraphics[height=0.75in]{20}
        \caption*{$D_{11}$}
    \end{subfigure}
    \begin{subfigure}[t]{0.15\textwidth}
        \centering
        \includegraphics[height=0.75in]{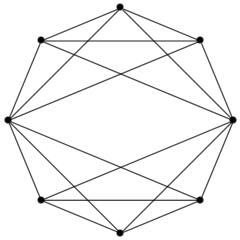}
        \caption*{$D_{12}*$}
    \end{subfigure}
\end{figure}
\begin{figure}[!ht]
    \centering
    \begin{subfigure}[t]{0.15\textwidth}
        \centering
        \includegraphics[height=0.75in]{25}
        \caption*{$D_{13}$}
    \end{subfigure}
    \begin{subfigure}[t]{0.15\textwidth}
        \centering
        \includegraphics[height=0.75in]{26}
        \caption*{$D_{14}$}
    \end{subfigure}
    \begin{subfigure}[t]{0.15\textwidth}
        \centering
        \includegraphics[height=0.75in]{27}
        \caption*{$D_{15}$}
    \end{subfigure}
    \begin{subfigure}[t]{0.15\textwidth}
        \centering
        \includegraphics[height=0.75in]{30}
        \caption*{$D_{16}$}
    \end{subfigure}
    \begin{subfigure}[t]{0.15\textwidth}
        \centering
        \includegraphics[height=0.75in]{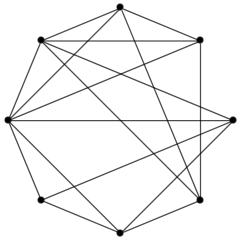}
        \caption*{$D_{17}*$}
    \end{subfigure}
    \begin{subfigure}[t]{0.15\textwidth}
        \centering
        \includegraphics[height=0.75in]{33}
        \caption*{$D_{18}$}
    \end{subfigure}
\end{figure}
\begin{figure}[!ht]
    \centering
    \begin{subfigure}[t]{0.15\textwidth}
        \centering
        \includegraphics[height=0.75in]{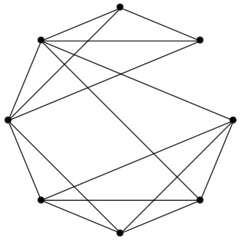}
        \caption*{$D_{19}$}
    \end{subfigure}
    \begin{subfigure}[t]{0.15\textwidth}
        \centering
        \includegraphics[height=0.75in]{38}
        \caption*{$D_{20}$}
    \end{subfigure}
    \begin{subfigure}[t]{0.15\textwidth}
        \centering
        \includegraphics[height=0.75in]{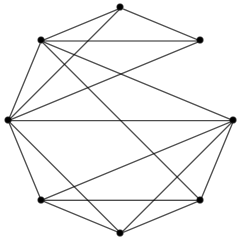}
        \caption*{$D_{21}*$}
    \end{subfigure}
    \begin{subfigure}[t]{0.15\textwidth}
        \centering
        \includegraphics[height=0.75in]{41}
        \caption*{$D_{22}$}
    \end{subfigure}
    \begin{subfigure}[t]{0.15\textwidth}
        \centering
        \includegraphics[height=0.75in]{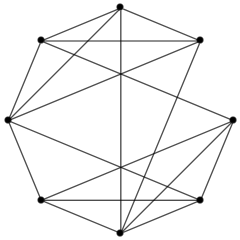}
        \caption*{$D_{23}*$}
    \end{subfigure}
    \begin{subfigure}[t]{0.15\textwidth}
        \centering
        \includegraphics[height=0.75in]{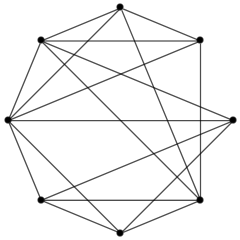}
        \caption*{$D_{24}*$}
    \end{subfigure}
\end{figure}
\begin{figure}[!ht]
    \centering
    \begin{subfigure}[t]{0.15\textwidth}
        \centering
        \includegraphics[height=0.75in]{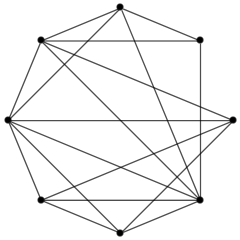}
        \caption*{$D_{25}*$}
    \end{subfigure}
    \begin{subfigure}[t]{0.15\textwidth}
        \centering
        \includegraphics[height=0.75in]{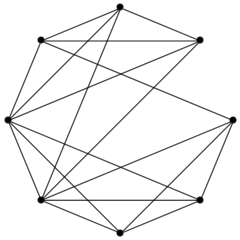}
        \caption*{$D_{26}*$}
    \end{subfigure}
    \begin{subfigure}[t]{0.15\textwidth}
        \centering
        \includegraphics[height=0.75in]{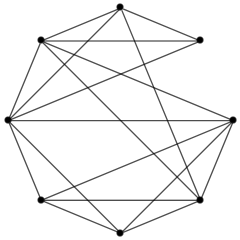}
        \caption*{$D_{27}*$}
    \end{subfigure}
    \begin{subfigure}[t]{0.15\textwidth}
        \centering
        \includegraphics[height=0.75in]{53}
        \caption*{$D_{28}$}
    \end{subfigure}
    \begin{subfigure}[t]{0.15\textwidth}
        \centering
        \includegraphics[height=0.75in]{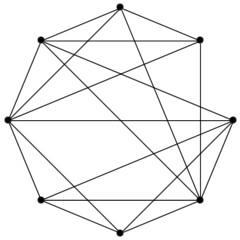}
        \caption*{$D_{29}*$}
    \end{subfigure}
    \begin{subfigure}[t]{0.15\textwidth}
        \centering
        \includegraphics[height=0.75in]{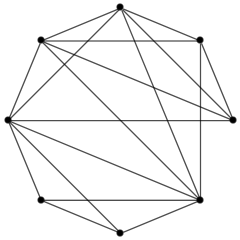}
        \caption*{$D_{30}*$}
    \end{subfigure}
\end{figure}
\begin{figure}[!ht]
    \centering
    \begin{subfigure}[t]{0.15\textwidth}
        \centering
        \includegraphics[height=0.75in]{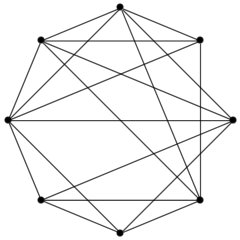}
        \caption*{$D_{31}*$}
    \end{subfigure}
    \begin{subfigure}[t]{0.15\textwidth}
        \centering
        \includegraphics[height=0.75in]{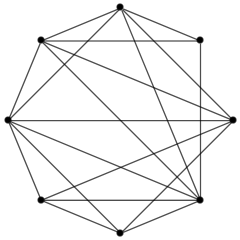}
        \caption*{$D_{32}*$}
    \end{subfigure}
    \begin{subfigure}[t]{0.15\textwidth}
        \centering
        \includegraphics[height=0.75in]{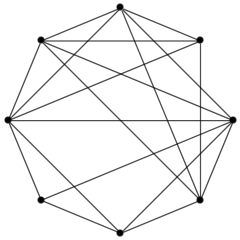}
        \caption*{$D_{33}*$}
    \end{subfigure}
    \begin{subfigure}[t]{0.15\textwidth}
        \centering
        \includegraphics[height=0.75in]{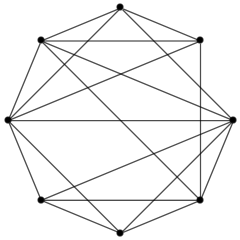}
        \caption*{$D_{34}*$}
    \end{subfigure}
    \begin{subfigure}[t]{0.15\textwidth}
        \centering
        \includegraphics[height=0.75in]{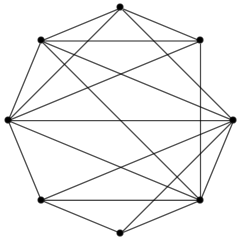}
        \caption*{$D_{35}*$}
    \end{subfigure}
    \begin{subfigure}[t]{0.15\textwidth}
        \centering
        \includegraphics[height=0.75in]{92}
        \caption*{$D_{36}$}
    \end{subfigure}
\end{figure}
\begin{figure}[!ht]
    \centering
    \begin{subfigure}[t]{0.15\textwidth}
        \centering
        \includegraphics[height=0.75in]{93}
        \caption*{$D_{37}$}
    \end{subfigure}
    \begin{subfigure}[t]{0.15\textwidth}
        \centering
        \includegraphics[height=0.75in]{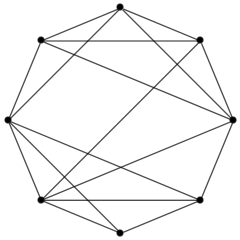}
        \caption*{$D_{38}*$}
    \end{subfigure}
    \begin{subfigure}[t]{0.15\textwidth}
        \centering
        \includegraphics[height=0.75in]{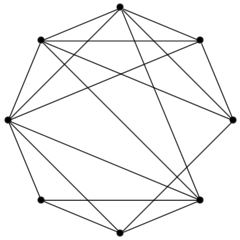}
        \caption*{$D_{39}*$}
    \end{subfigure}
    \begin{subfigure}[t]{0.15\textwidth}
        \centering
        \includegraphics[height=0.75in]{98}
        \caption*{$D_{40}$}
    \end{subfigure}
    \begin{subfigure}[t]{0.15\textwidth}
        \centering
        \includegraphics[height=0.75in]{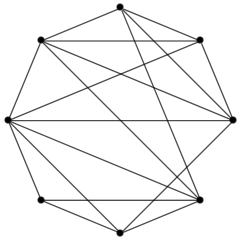}
        \caption*{$D_{41}*$}
    \end{subfigure}
    \begin{subfigure}[t]{0.15\textwidth}
        \centering
        \includegraphics[height=0.75in]{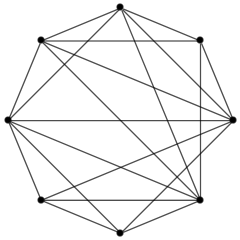}
        \caption*{$D_{42}*$}
    \end{subfigure}
\end{figure}
\begin{figure}[!ht]
    \centering
    \begin{subfigure}[t]{0.15\textwidth}
        \centering
        \includegraphics[height=0.75in]{118}
        \caption*{$D_{43}$}
    \end{subfigure}
    \begin{subfigure}[t]{0.15\textwidth}
        \centering
        \includegraphics[height=0.75in]{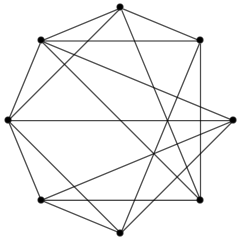}
        \caption*{$D_{44}*$}
    \end{subfigure}
    \begin{subfigure}[t]{0.15\textwidth}
        \centering
        \includegraphics[height=0.75in]{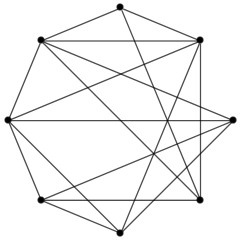}
        \caption*{$D_{45}*$}
    \end{subfigure}
    \begin{subfigure}[t]{0.15\textwidth}
        \centering
        \includegraphics[height=0.75in]{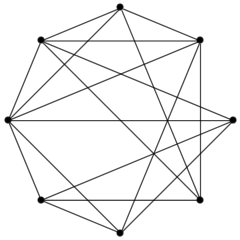}
        \caption*{$D_{46}*$}
    \end{subfigure}
    \begin{subfigure}[t]{0.15\textwidth}
        \centering
        \includegraphics[height=0.75in]{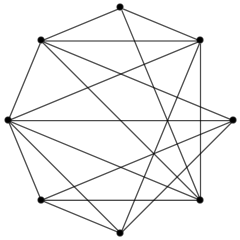}
        \caption*{$D_{47}*$}
    \end{subfigure}
    \begin{subfigure}[t]{0.15\textwidth}
        \centering
        \includegraphics[height=0.75in]{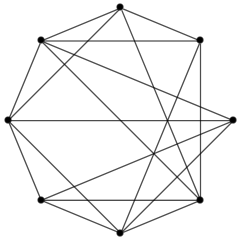}
        \caption*{$D_{48}*$}
    \end{subfigure}
\end{figure}
\begin{figure}[!ht]
    \centering
    \begin{subfigure}[t]{0.15\textwidth}
        \centering
        \includegraphics[height=0.75in]{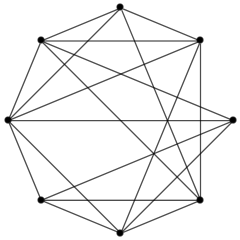}
        \caption*{$D_{49}*$}
    \end{subfigure}
    \begin{subfigure}[t]{0.15\textwidth}
        \centering
        \includegraphics[height=0.75in]{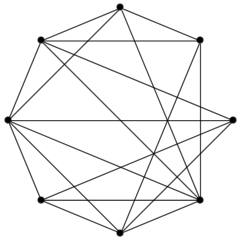}
        \caption*{$D_{50}*$}
    \end{subfigure}
    \begin{subfigure}[t]{0.15\textwidth}
        \centering
        \includegraphics[height=0.75in]{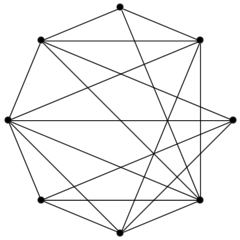}
        \caption*{$D_{51}*$}
    \end{subfigure}
    \begin{subfigure}[t]{0.15\textwidth}
        \centering
        \includegraphics[height=0.75in]{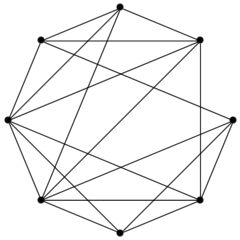}
        \caption*{$D_{52}*$}
    \end{subfigure}
    \begin{subfigure}[t]{0.15\textwidth}
        \centering
        \includegraphics[height=0.75in]{134}
        \caption*{$D_{53}$}
    \end{subfigure}
    \begin{subfigure}[t]{0.15\textwidth}
        \centering
        \includegraphics[height=0.75in]{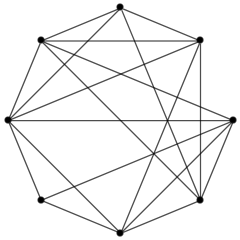}
        \caption*{$D_{54}*$}
    \end{subfigure}
\end{figure}
\begin{figure}[!ht]
    \centering
    \begin{subfigure}[t]{0.15\textwidth}
        \centering
        \includegraphics[height=0.75in]{137}
        \caption*{$D_{55}$}
    \end{subfigure}
    \begin{subfigure}[t]{0.15\textwidth}
        \centering
        \includegraphics[height=0.75in]{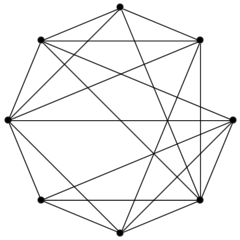}
        \caption*{$D_{56}*$}
    \end{subfigure}
    \begin{subfigure}[t]{0.15\textwidth}
        \centering
        \includegraphics[height=0.75in]{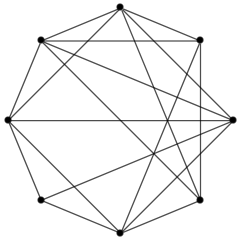}
        \caption*{$D_{57}*$}
    \end{subfigure}
    \begin{subfigure}[t]{0.15\textwidth}
        \centering
        \includegraphics[height=0.75in]{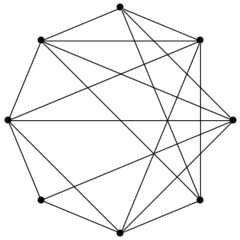}
        \caption*{$D_{58}*$}
    \end{subfigure}
    \begin{subfigure}[t]{0.15\textwidth}
        \centering
        \includegraphics[height=0.75in]{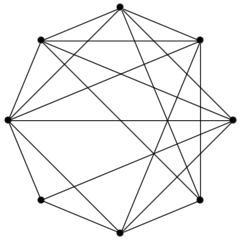}
        \caption*{$D_{59}*$}
    \end{subfigure}
    \begin{subfigure}[t]{0.15\textwidth}
        \centering
        \includegraphics[height=0.75in]{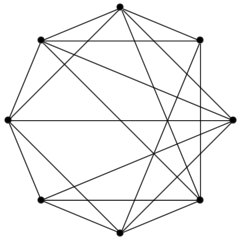}
        \caption*{$D_{60}*$}
    \end{subfigure}
\end{figure}
\begin{figure}[!ht]
    \centering
    \begin{subfigure}[t]{0.15\textwidth}
        \centering
        \includegraphics[height=0.75in]{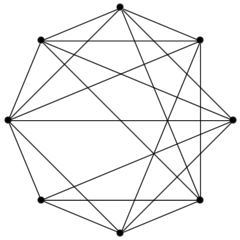}
        \caption*{$D_{61}*$}
    \end{subfigure}
    \begin{subfigure}[t]{0.15\textwidth}
        \centering
        \includegraphics[height=0.75in]{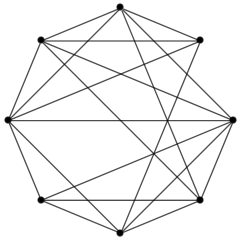}
        \caption*{$D_{62}*$}
    \end{subfigure}
    \begin{subfigure}[t]{0.15\textwidth}
        \centering
        \includegraphics[height=0.75in]{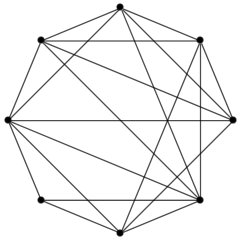}
        \caption*{$D_{63}*$}
    \end{subfigure}
    \begin{subfigure}[t]{0.15\textwidth}
        \centering
        \includegraphics[height=0.75in]{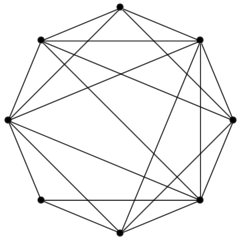}
        \caption*{$D_{64}*$}
    \end{subfigure}
    \begin{subfigure}[t]{0.15\textwidth}
        \centering
        \includegraphics[height=0.75in]{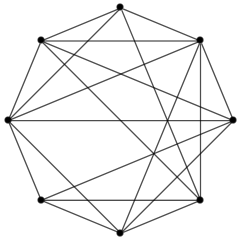}
        \caption*{$D_{65}*$}
    \end{subfigure}
    \begin{subfigure}[t]{0.15\textwidth}
        \centering
        \includegraphics[height=0.75in]{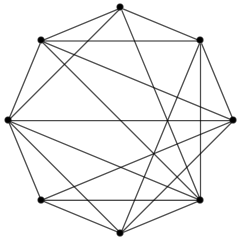}
        \caption*{$D_{66}*$}
    \end{subfigure}
\end{figure}
\begin{figure}[!ht]
    \centering
    \begin{subfigure}[t]{0.15\textwidth}
        \centering
        \includegraphics[height=0.75in]{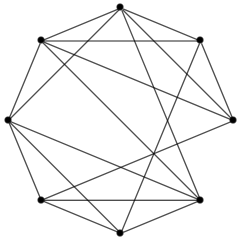}
        \caption*{$D_{67}*$}
    \end{subfigure}
    \begin{subfigure}[t]{0.15\textwidth}
        \centering
        \includegraphics[height=0.75in]{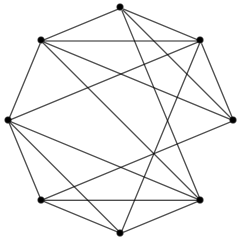}
        \caption*{$D_{68}*$}
    \end{subfigure}
    \begin{subfigure}[t]{0.15\textwidth}
        \centering
        \includegraphics[height=0.75in]{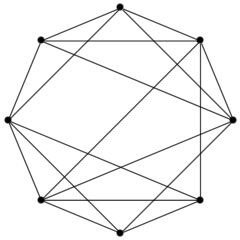}
        \caption*{$D_{69}*$}
    \end{subfigure}
    \begin{subfigure}[t]{0.15\textwidth}
        \centering
        \includegraphics[height=0.75in]{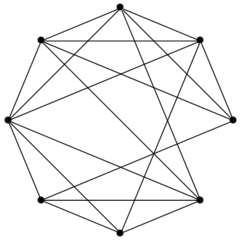}
        \caption*{$D_{70}*$}
    \end{subfigure}
    \begin{subfigure}[t]{0.15\textwidth}
        \centering
        \includegraphics[height=0.75in]{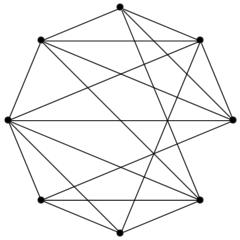}
        \caption*{$D_{71}*$}
    \end{subfigure}
    \begin{subfigure}[t]{0.15\textwidth}
        \centering
        \includegraphics[height=0.75in]{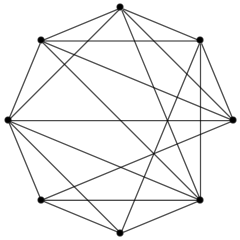}
        \caption*{$D_{72}*$}
    \end{subfigure}
\end{figure}
\begin{figure}[!ht]
    \centering
    \begin{subfigure}[t]{0.15\textwidth}
        \centering
        \includegraphics[height=0.75in]{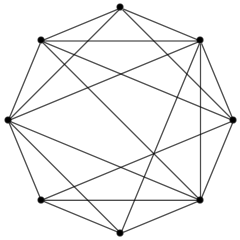}
        \caption*{$D_{73}*$}
    \end{subfigure}
    \begin{subfigure}[t]{0.15\textwidth}
        \centering
        \includegraphics[height=0.75in]{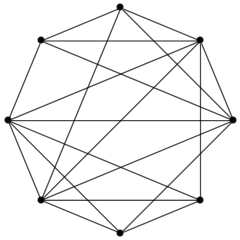}
        \caption*{$D_{74}*$}
    \end{subfigure}
    \begin{subfigure}[t]{0.15\textwidth}
        \centering
        \includegraphics[height=0.75in]{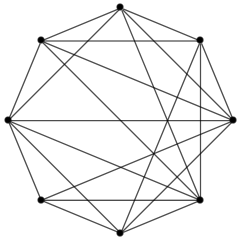}
        \caption*{$D_{75}*$}
    \end{subfigure}
    \begin{subfigure}[t]{0.15\textwidth}
        \centering
        \includegraphics[height=0.75in]{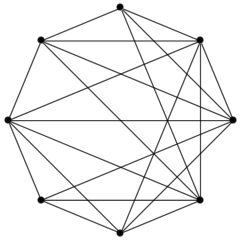}
        \caption*{$D_{76}*$}
    \end{subfigure}
    \begin{subfigure}[t]{0.15\textwidth}
        \centering
        \includegraphics[height=0.75in]{207}
        \caption*{$D_{77}$}
    \end{subfigure}
    \begin{subfigure}[t]{0.15\textwidth}
        \centering
        \includegraphics[height=0.75in]{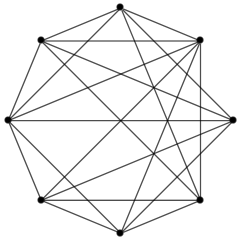}
        \caption*{$D_{78}*$}
    \end{subfigure}
\end{figure}
\begin{figure}[!ht]
    \centering
    \begin{subfigure}[t]{0.15\textwidth}
        \centering
        \includegraphics[height=0.75in]{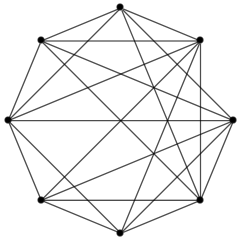}
        \caption*{$D_{79}*$}
    \end{subfigure}
    \begin{subfigure}[t]{0.15\textwidth}
        \centering
        \includegraphics[height=0.75in]{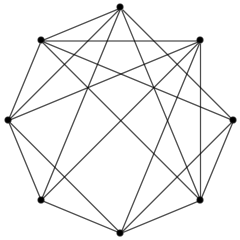}
        \caption*{$D_{80}*$}
    \end{subfigure}
    \begin{subfigure}[t]{0.15\textwidth}
        \centering
        \includegraphics[height=0.75in]{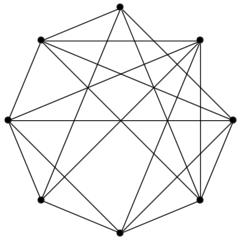}
        \caption*{$D_{81}*$}
    \end{subfigure}
    \begin{subfigure}[t]{0.15\textwidth}
        \centering
        \includegraphics[height=0.75in]{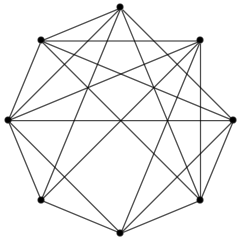}
        \caption*{$D_{82}*$}
    \end{subfigure}
    \begin{subfigure}[t]{0.15\textwidth}
        \centering
        \includegraphics[height=0.75in]{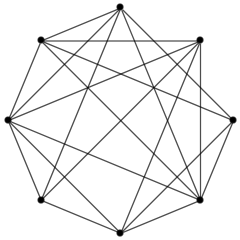}
        \caption*{$D_{83}*$}
    \end{subfigure}
    \begin{subfigure}[t]{0.15\textwidth}
        \centering
        \includegraphics[height=0.75in]{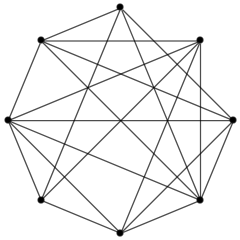}
        \caption*{$D_{84}*$}
    \end{subfigure}
\end{figure}
\begin{figure}[!ht]
    \centering
    \begin{subfigure}[t]{0.15\textwidth}
        \centering
        \includegraphics[height=0.75in]{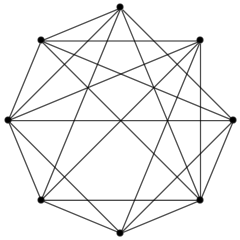}
        \caption*{$D_{85}*$}
    \end{subfigure}
    \begin{subfigure}[t]{0.15\textwidth}
        \centering
        \includegraphics[height=0.75in]{224}
        \caption*{$D_{86}$}
    \end{subfigure}
    \begin{subfigure}[t]{0.15\textwidth}
        \centering
        \includegraphics[height=0.75in]{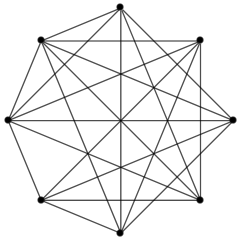}
        \caption*{$D_{87}*$}
    \end{subfigure}
    \begin{subfigure}[t]{0.15\textwidth}
        \centering
        \includegraphics[height=0.75in]{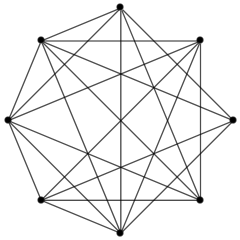}
        \caption*{$D_{88}*$}
    \end{subfigure}
    \begin{subfigure}[t]{0.15\textwidth}
        \centering
        \includegraphics[height=0.75in]{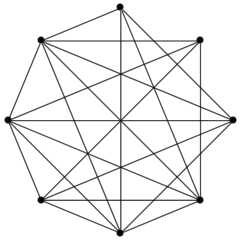}
        \caption*{$D_{89}*$}
    \end{subfigure}
    \begin{subfigure}[t]{0.15\textwidth}
        \centering
        \includegraphics[height=0.75in]{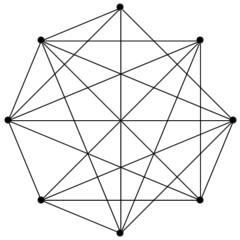}
        \caption*{$D_{90}*$}
    \end{subfigure}
\end{figure}
\begin{figure}[!ht]
    \centering
    \begin{subfigure}[t]{0.15\textwidth}
        \centering
        \includegraphics[height=0.75in]{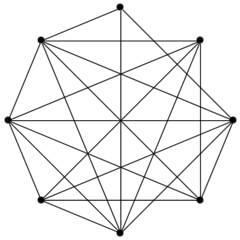}
        \caption*{$D_{91}*$}
    \end{subfigure}
    \begin{subfigure}[t]{0.15\textwidth}
        \centering
        \includegraphics[height=0.75in]{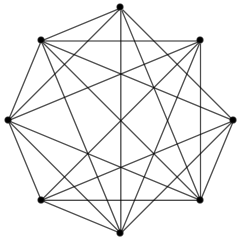}
        \caption*{$D_{92}*$}
    \end{subfigure}
    \begin{subfigure}[t]{0.15\textwidth}
        \centering
        \includegraphics[height=0.75in]{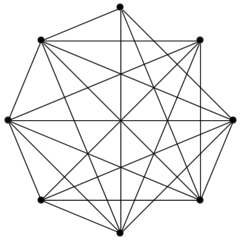}
        \caption*{$D_{93}*$}
    \end{subfigure}
    \begin{subfigure}[t]{0.15\textwidth}
        \centering
        \includegraphics[height=0.75in]{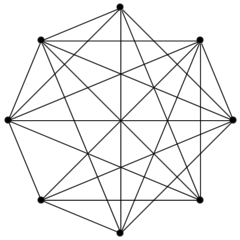}
        \caption*{$D_{94}*$}
    \end{subfigure}
    \begin{subfigure}[t]{0.15\textwidth}
        \centering
        \includegraphics[height=0.75in]{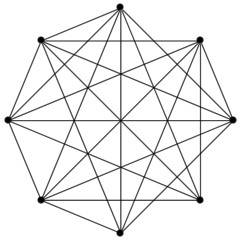}
        \caption*{$D_{95}*$}
    \end{subfigure}
    \begin{subfigure}[t]{0.15\textwidth}
        \centering
        \includegraphics[height=0.75in]{295}
        \caption*{$D_{96}$}
    \end{subfigure}
\end{figure}

\end{appendix}

\end{document}